\documentclass[10pt]{amsart}
\usepackage{amssymb, amsthm, amsmath}
\usepackage[all]{xy}
\usepackage{graphicx}
\usepackage{psfrag}

\newtheorem{prop}{Proposition}
\newtheorem{thm}{Theorem}
\newtheorem*{thmCB*}{Theorem CB}
\newtheorem*{thmD*}{Theorem D}
\newtheorem{cor}{Corollary}
\newtheorem{lemma}{Lemma}
\theoremstyle{definition}

\newtheorem{defn}{Definition}
\newtheorem{example}{Example}
\newtheorem*{exA1*}{Example A.1}
\newtheorem*{exA2*}{Example A.2}
\newtheorem{remark}{Remark}

\newcommand\C{{\mathbb C}}

\newcommand\N{{\mathbb N}}

\newcommand{\Ti}{\Theta}

\newcommand{\om}{{\varpi}}
\newcommand{\ome}{{\omega}}

\newcommand\X{{\mathfrak X}}

\newcommand\ci{{\mathfrak q}}

\newcommand\Z{{\mathbb Z}}

\newcommand\AS{{\mathfrak S}}
\newcommand\BS{{\mathfrak B}}
\newcommand\CS{{\mathfrak C}}
\newcommand\DS{{\mathfrak D}}

\newcommand\fraka{{\mathfrak a}}
\newcommand\frakb{{\mathfrak b}}
\newcommand\frakc{{\mathfrak c}}
\newcommand\frakd{{\mathfrak d}}

\newcommand\al{\alpha}
\newcommand{\be}{\beta}
\newcommand\la{\lambda}
\newcommand\s{{\sigma}}
\newcommand\kk{{\kappa}}
\newcommand\up{{\upsilon}}
\newcommand\supp{{\mathrm{supp}}}
\newcommand\Eta{H}

\newcommand\ssm{\smallsetminus}
\newcommand\noin{\noindent}

\newcommand\eqto{\stackrel{\lower1.5pt\hbox{$\scriptstyle\sim\,$}}\to}
\newcommand\ov{\overline}
\newcommand\hra{\hookrightarrow}
\newcommand\wh{\widehat}
\newcommand\wt{\widetilde}
\newcommand\dis{\displaystyle}

\DeclareMathOperator{\OG}{OG}

\DeclareMathOperator{\f}{{\mathfrak f}}
\DeclareMathOperator{\g}{{\mathfrak g}}

\DeclareMathOperator{\HH}{\mathrm{H}}

\DeclareMathOperator{\type}{\mathrm{type}}
\DeclareMathOperator{\rank}{\mathrm{rank}}

\begin{document}

\title[Degeneracy locus formulas for amenable elements]
{Degeneracy locus formulas for amenable Weyl group elements}

\date{February 13, 2026}

\author{Harry~Tamvakis} 
\address{University of Maryland, Department of
Mathematics, William E. Kirwan Hall, 4176 Campus Drive, 
College Park, MD 20742, USA}
\email{harryt@umd.edu}



\begin{abstract}
We define a class of amenable Weyl group elements in the Lie types B,
C, and D, which we propose as the analogues of vexillary permutations
in these Lie types.  Our amenable signed permutations index flagged
theta and eta polynomials, which generalize the double theta and eta
polynomials of Wilson and the author. In geometry, we obtain
corresponding formulas for the cohomology classes of symplectic and
orthogonal degeneracy loci.
\end{abstract}

\maketitle

\section{Introduction}

A fundamental problem in Schubert calculus is that of finding
polynomial representatives for the cohomology classes of the Schubert
varieties. In the mid 1990s, Fulton and Pragacz \cite{FP} asked a
relative version of the same question, seeking explicit formulas which
represent the classes of degeneracy loci for the classical
groups, in the sense of \cite{F1, F2, PR}. These loci pull back from
the universal Schubert varieties in a $G/P$-bundle, where $G$ is a
classical Lie group, and $P$ a parabolic subgroup of $G$. As such,
they are indexed by elements $w$ in the Weyl group of $G$, which
describe the relative position of two flags of (isotropic) subspaces
of a fixed (symplectic or orthogonal) vector space.

The above {\em Giambelli} and {\em degeneracy locus} problems were
solved in full generality in \cite{T1}. The answer given there is a
positive Chern class formula which respects the symmetries of the Weyl
group element $w$ and its inverse. The paper \cite{T1} introduced a
new, intrinsic point of view in Schubert calculus, showing that
formulas native to the homogeneous space $G/P$ are possible, for any
parabolic subgroup $P$, and in all classical Lie types. In special
cases, there are alternatives to the general formulas of \cite{T1},
but they must all be equivalent to the formulas found there, modulo an
explicit ideal of relations among the variables involved.
 
The seminal work of Lascoux and Sch\"{u}tzenberger \cite{LS1, LS2} on
Schubert polynomials exposed intrinsic formulas for an important class
of permutations, which they called {\em vexillary}. They defined the
{\em shape} of a general permutation to be the partition obtained by
arranging the entries of its code in decreasing order.  The key
defining property of a vexillary permutation was that its Schubert
polynomial can be expressed as a flagged Schur polynomial indexed by
its shape. In particular, the prototype for vexillary permutations
were the {\em Grassmannian permutations}, whose Schubert polynomials
are the classical Schur polynomials. Our aim in the present paper is
to define a family of {\em amenable signed permutations}, which serve
as the analogues of vexillary permutations in the other classical Lie
types.

There have been two attempts in the past to define a notion of
vexillary signed permutation in the Lie types B, C, and D, by Billey
and Lam \cite{BL} and Anderson and Fulton \cite{AF1}.  These
definitions miss the mark because according to either of them, the
Grassmannian signed permutations are not all vexillary. The revision
\cite{AF2} of \cite{AF1} sought to generalize the latter paper by
incorporating the theta and eta polynomials of Buch, Kresch, and the
author \cite{BKT2, BKT3} and Wilson \cite{W, TW}, which are the
analogues of the Schur polynomials in the aforementioned Lie
types. Unfortunately, although \cite{AF2} is in the right direction,
the proofs given there contain serious errors in all Lie types except
type A, and the main theorem is false, at least in type D.  Moreover,
Anderson and Fulton have not acknowledged that intrinsic Chern class
formulas for the cohomology classes of all the degeneracy loci are
known in any of their writings to date.

Our approach to amenable elements is based on a careful study of how
the corresponding raising operator formulas transform under divided
differences. The outline of the argument is similar to the one found
in Macdonald's notes \cite{M}, but there are important differences in
the details. The proof is new even in type A, where we obtain a new
characterization of vexillary permutations (see below). It is critical
to work with double polynomials throughout and use both left and right
divided differences to maximum effect, starting from the known formula
for the top polynomial, which is indexed by the longest length
element. In types B, C, and D, we employ the Schubert polynomials of
Ikeda, Mihalcea, and Naruse \cite{IMN}, which extend the work of
Billey and Haiman \cite{BH} to a theory suitable for applications to
equivariant cohomology and degeneracy loci.  The paper \cite{T6}
provides another key ingredient: the definition of the {\em shape} of
a signed permutation, which plays the role of Lascoux and
Sch\"{u}tzenberger's shape in the latter Lie types.

The difficulty when working with sequences of divided differences
applied to polynomials lies in choosing which path to follow in the
weak Bruhat order, as the Leibnitz rule tends to destroy any nice
formulas. The papers \cite{TW, T4, T5} showed how divided differences
can be used to obtain combinatorial proofs of the raising operator
formulas for double theta and double eta polynomials, exploiting the
fact that these polynomials behave well under the action of {\em left}
divided differences. Therefore, as long as one remains among
the Grassmannian elements, the choice of path through the left weak
Bruhat order is immaterial. However this surprising property, first observed
in the symplectic case by Ikeda and Matsumura \cite{IM}, completely
fails once one leaves the Grassmannian regime.

To solve this problem, we introduce the notion of {\em leading
  elements} of the Weyl group, which generalize the Grassmannian
elements. In the Lie types A, B, and C, a (signed) permutation
$w=(w_1,\ldots,w_n)$ is leading if the A-code of the extended sequence
$(0,w_1,\ldots,w_n)$ is unimodal. The analogous treatment of type D
elements involves some subtleties, which we discuss later.  The
leading signed permutations are partitioned into equivalence classes
defined by their {\em truncated A-code}. Each of them is in bijection
with the class of Grassmannian elements, where the truncated A-code
vanishes.  The longest length elements within each class give rise to
Pfaffian formulas which are proved using divided differences, starting
from the formula for the longest element in the Weyl group. Following
this, any sequence of left divided differences used to establish the
double theta/eta polynomial formula in the Grassmannian case works --
in the same way! -- to prove a corresponding `factorial' formula for
the elements of the other equivalence classes.

Once the formulas for leading elements are obtained, one can continue
to apply type A divided differences, in a manner that preserves the
shape of these formulas, and proceed a bit further down the left weak
order. We thus arrive at our definition of {\em amenable elements} of
the Weyl group: they are modifications of leading elements, obtained
by multiplying them on the left by suitable permutations. In the
symmetric group, this reflects the (apparently new) fact that the
vexillary permutations are exactly those which can be written as
products $\omega\om$, with $\ell(\omega\om)=\ell(\om)-\ell(\omega)$,
where $\omega$ and $\om$ are $312$-avoiding and $132$-avoiding
permutations, respectively.

Finally, one has to deal with the problem that the above formulas do
not respect the symmetries (that is, the descent sets) of the amenable
Weyl group element involved. This issue was dealt with in \cite{M} by
exploiting the alternating properties of determinants, and a similar
argument works for the Pfaffian examples of \cite{Ka, AF1}.  In the
situation at hand, we require variants of the key technical lemmas
obtained in \cite{BKT2}, which exposed the more subtle alternating
properties of the raising operator expressions that define theta
polynomials.

To understand some of the additional challenges one faces in the even
orthogonal type D, consider first the question of how to define the
shape of an element $w$ in the associated Weyl group $\wt{W}_n$. There
seems to be no consistent way to do this, since e.g.\ the element
$(\ov{3},\ov{1},2)$ has shape $\la=2$ when considered as a
$\Box$-Grassmannian element, but shape $\la=(1,1)$ when considered as
a $1$-Grassmannian element. The definition given in \cite[Def.\ 5]{T6}
prefers the latter shape over the former, but the more difficult
question before us here requires a further refinement.

Our solution is to define the shape of $w$ to be a {\em typed}
partition, where the type is an integer in $\{0,1,2\}$, extending the
corresponding notion for Grassmannian elements from \cite{BKT1}.  The
$\Box$-Grassmannian elements and their Pfaffian formulas are abandoned
entirely; instead, we view them all as $1$-Grassmannian elements!
This fits in well with our previous papers \cite{BKT1, BKT3, T2, T4}
on the orthogonal Grassmannians $\OG(n-k,2n)$ and (double) eta
polynomials, where we assumed $k\geq 1$ from the beginning -- but for
a different reason.

Another obstacle appears when one tries to define the leading elements
of $\wt{W}_n$. It was observed in \cite[Sec.\ 3.3]{T4} that the
compatibility of double eta polynomials with left divided differences
is more delicate than the corresponding fact in types B and C. In
order to preserve this crucial property for the polynomials indexed by
leading elements, we must demand that they are all {\em proper}
elements of $\wt{W}_n$ (Definition \ref{properdef}).  There is no
analogue of this subtle condition in the other classical Lie
types. Once all the definitions which are special to the type D theory
are found, the proof of the main result proceeds in a manner parallel
to the other three types.

\medskip

We now provide the statements of our main theorems. Let $E\to \X$ be a
symplectic or orthogonal vector bundle of rank $N$ on a smooth complex
algebraic variety $\X$. We are given two complete flags of subbundles
of $E$
\[
0 \subset E_1\subset \cdots \subset E_N=E \ \ \, \mathrm{and} \, \ \ 
0 \subset F_1\subset \cdots \subset F_N=E
\]
with $\rank E_r=\rank F_r=r$ for each $r$. If $N=2n$ is even, we have
$E_{n+s}=E_{n-s}^{\perp}$ and $F_{n+s}=F_{n-s}^{\perp}$ for $0\leq s <
n$, while if $N=2n+1$ is odd, we are in the orthogonal case, and have
$E_{n+s}=E_{n+1-s}^{\perp}$ and $F_{n+s}=F_{n+1-s}^{\perp}$ for $1\leq
s \leq n$.  Consider the {\em degeneracy locus} $\X_w\subset \X$,
which we assume has pure codimension $\ell(w)$ in $\X$ (the precise
definition of $\X_w$ is given in Sections \ref{sdl} and \ref{odl}). If
$w$ is an amenable Weyl group element, we obtain formulas for the class
of $\X_w$ in the cohomology ring of $\X$, which are given by 
{\em flagged theta} and {\em flagged eta polynomials}.

Fix an amenable signed permutation $w$ in the hyperoctahedral group
$W_n$.  Let $k\geq 0$ be the first right descent of $w$, list the
entries $w_{k+1},\ldots,w_n$ in increasing order:
\[
u_1<\cdots < u_m < 0 < u_{m+1} < \cdots < u_{n-k}
\]
and define 
\[
\beta:=(u_1+1,\ldots,u_m+1,u_{m+1},\ldots,u_{n-k}),
\]
\[
D:=\{(i,j) \ |\ 1\leq i<j \leq n-k \ \, \text{and} \ \, 
u_i+ u_j < 0 \},
\]
and the raising operator expression
\[
R^D: = \prod_{i<j}(1-R_{ij})\prod_{i<j\, :\, (i,j)\in D}(1+R_{ij})^{-1}.
\]

The A-code of $w$ is the sequence $\gamma$ with $\gamma_i:=
\#\{j>i\ |\ w_j<w_i\}$. Define two partitions $\nu$ and $\xi$ by
setting $\nu_j:= \#\{i\ |\ \gamma_i\geq j\}$ and
$\xi_j:=\#\{i\ |\ \gamma_{k+i}\geq j\}$ for each $j\geq 1$. Following
\cite{T6}, the shape of $w$ is the partition $\la=\mu+\nu$, where
$\mu:=(-u_1,\ldots,-u_m)$.  If $\ell$ denotes the length of $\la$, we say
that $\ci\in [1,\ell]$ is a {\em critical index} if $\beta_{\ci+1}>
\beta_\ci+1$, or if $\la_\ci>\la_{\ci+1}+1$ (respectively,
$\la_\ci>\la_{\ci+1}$) and $\ci<m$ (respectively, $\ci>m$). Define two
sequences $\f$ and $\g$ of length $\ell$ by setting
\[
\f_j:=k+\max(i\ |\ \gamma_{k+i}\geq j)
\]
for each $j$, and $\g_j:=\f_\ci+\beta_\ci-\xi_\ci-k,$ where $\ci$ is
the least critical index such that $\ci\geq j$. The sequences $\f$ and
$\g$ are the {\em right} and {\em left flags} of $w$, and $\f$
(respectively $|\g|$) consists of right (respectively left) descents
of $w$. Define the sequence $\ov{\g}$ by setting $\ov{\g}_j:=\g_j$, if
$\g_j\geq 0$, and $\ov{\g}_j:=\g_j-1$, if $\g_j<0$.

\begin{thmCB*} Let $w$ be an amenable element of $W_n$.

\medskip
\noin
{\em (a)} If $E$ is a symplectic vector bundle, then we have
\begin{equation}
\label{introCChern}
[\X_w] = \Theta_w(E-E_{n-\f}-F_{n+\g}) =
R^D\, c_{\la}(E-E_{n-\f}-F_{n+\g})
\end{equation}
in the cohomology ring $\HH^*(\X)$.

\medskip
\noin
{\em (b)} If $E$ is an odd orthogonal vector bundle, then we have
\begin{equation}
\label{introBChern}
[\X_w] = 2^{-m} \Theta_w(E-E_{n-\f}-F_{n+1+\ov{\g}}) =
2^{-m}R^D\, c_{\la}(E-E_{n-\f}-F_{n+1+\ov{\g}})
\end{equation}
in the cohomology ring $\HH^*(\X)$.
\end{thmCB*}
\noin Following \cite{TW}, the Chern polynomial in equation
(\ref{introCChern}) is interpreted as the image of $R^D
\mathfrak{c}_\la$ under the $\Z$-linear map which sends the
noncommutative monomial
$\frakc_\al=\frakc_{\al_1}\frakc_{\al_2}\cdots$ to $\prod_j
c_{\al_j}(E-E_{n-\f_j}-F_{n+\g_j})$, for every integer sequence $\al$.
The Chern polynomial in equation (\ref{introBChern}) is defined
similarly.

As discussed above, the main theorem in the even orthogonal case
involves nuances in its formulation and its proof; we therefore state it separately.

\begin{thmD*}
Let $w$ be an amenable element of $\wt{W}_n$.  If $E$ is an even
orthogonal vector bundle, then we have
\begin{equation}
\label{introDChern}
[\X_w] = \Eta_w(E-E_{n-\f}-F_{n+\g}) =
2^{-m}R^D \star \wh{c}_{\la}(E-E_{n-\f}-F_{n+\g})
\end{equation}
in the cohomology ring $\HH^*(\X)$.
\end{thmD*}
\noin We refer the reader to Sections \ref{LtBD} and \ref{Dtheory} for
the precise meaning and basic properties of the terms which appear in
equation (\ref{introDChern}).

This paper is organized as follows. Section \ref{prelims} contains
background material on divided differences and Schubert polynomials,
and defines the shape of a (signed) permutation in all the classical
types. Section \ref{ro} deals with raising operators and provides
variants of the lemmas from \cite{BKT2} that we require here. Sections
\ref{aeA}, \ref{Ctheory}, and \ref{Dtheory} define and study amenable
elements and their applications in types A, C, and B/D,
respectively. In particular, we give our notion of flagged theta
and flagged eta polynomials; these are indexed by amenable Weyl group
elements. Finally, Appendix \ref{AFexs} contains
counterexamples to several statements in \cite{AF2}.

I thank Andrew Kresch for encouraging me to work on this article,
providing useful comments, and, more importantly, for being a good
friend. Thanks are also due to the referees for a careful reading of
the paper and suggestions which helped to improve the exposition and
to simplify the proof of Theorem \ref{amenvexthm}.

\section{Preliminaries}
\label{prelims}

This section gathers together background material on the divided 
differences and Schubert polynomials used in this work. We also 
discuss the notion of the shape of a (signed) permutation. Our
notation is compatible with that found in \cite{T6}.

\subsection{Lie type A}

Throughout this paper we will employ integer sequences
$\al=(\al_1,\al_2,\ldots)$, which are assumed to have finite support,
and we identify with integer vectors. The integer sequence $\al$ is a
{\em composition} if $\al_j\geq 0$ for all $j$. A weakly decreasing
composition is called a {\em partition}.  If $\la$ is a partition, the
{\em length} of $\la$ is the integer $\ell(\la):=\#\{i\ |\ \la_i\neq
0\}$, and the {\em conjugate} of $\la$ is the partition $\la'$ with
$\la'_j:=\#\{i\ |\ \la_i\geq j\}$ for all $j\geq 1$. As is customary,
we identify partitions with their Young diagrams of boxes, arranged in
left justified rows.  An inclusion $\mu\subset\la$ of partitions
corresponds to the containment of their respective diagrams; in this
case, the skew diagram $\la/\mu$ is the set-theoretic difference
$\la\ssm\mu$.  For each integer $r\geq 1$, let
$\delta_r:=(r,r-1,\ldots,1)$, $\delta^{\vee}_r:=(1,2,\ldots,r)$, and
set $\delta_0:=0$. Denote by $\epsilon_r$ the sequence whose $r$th
term is 1 and all other terms are zero.

The symmetric group $S_n$ is generated by the simple transpositions
$s_i=(i,i+1)$ for $1\leq i \leq n-1$.  There is a natural embedding of
$S_n$ in $S_{n+1}$ by adjoining $n+1$ as a fixed point, and we let
$S_\infty:=\cup_n S_n$.  We will write a permutation $\om\in S_n$
using one line notation, as the word $(\om_1,\ldots,\om_n)$ where
$\om_i=\om(i)$.  

The {\em length} of a permutation $\om$, denoted $\ell(\om)$, is the
least integer $r$ such that we have an expression $\om=s_{i_1} \cdots
s_{i_r}$. The word $s_{i_1}\cdots s_{i_r}$ is called a {\em reduced
  decomposition} for $\om$.  An element $\om\in S_\infty$ has a {\em
  left descent} (respectively, a {\em right descent}) at position
$i\geq 1$ if $\ell(s_i\om)<\ell(\om)$ (respectively, if $\ell(\om
s_i)<\ell(\om)$).  The permutation $\om=(\om_1,\om_2,\ldots)$ has a
right descent at $i$ if and only if $\om_i>\om_{i+1}$, and a left
descent at $i$ if and only if $\om^{-1}(i)>\om^{-1}(i+1)$.

The {\em code} $\gamma=\gamma(\om)$ of a permutation $\om\in S_n$ is
the sequence $\{\gamma_i\}$ with
$\gamma_i:=\#\{j>i\ |\ \om_j<\om_i\}$.  The code $\gamma$ determines
$\om$, as follows. We have $\om_1=\gamma_1+1$, and for $i>1$, $\om_i$
is the $(\gamma_i+1)$st element in the complement of
$\{\om_1,\ldots,\om_{i-1}\}$ in the sequence $(1,\ldots,n)$.
Following \cite{LS1, M}, the {\em shape} $\la=\la(\om)$ of $\om$ is the
partition whose parts are the non-zero entries $\gamma_i$ of the code
$\gamma(\om)$, arranged in weakly decreasing order. We have 
$|\la|:=\sum_i \la_i=\sum_i \gamma_i = \ell(\om)$.

For any integer $p\geq 0$ and sequence of variables
$Z:=(z_1,z_2,\ldots)$, the elementary and complete symmetric functions
$e_p(Z)$ and $h_p(Z)$ are defined by the generating series
\[
\prod_{i=1}^{\infty}(1+z_it) = \sum_{p=0}^{\infty}
e_p(Z)t^p \ \ \ \text{and} \ \ \ 
\prod_{i=1}^{\infty}(1-z_it)^{-1} = \sum_{p=0}^{\infty}
h_p(Z)t^p,
\]
respectively.  If $r\geq 1$ then we let
$e^r_p(Z):=e_p(z_1,\ldots,z_r)$ and $h^r_p(Z):=h_p(z_1,\ldots,z_r)$
denote the polynomials obtained from $e_p(Z)$ and $h_p(Z)$ by setting
$z_i=0$ for all $i>r$. Let $e^0_p(Z)=h^0_p(Z):=\delta_{0p}$, where
$\delta_{0p}$ denotes the Kronecker delta, and for $r<0$, define
$h^r_p(Z):=e^{-r}_p(Z)$ and $e^r_p(Z):=h^{-r}_p(Z)$.

Let $X:=(x_1,x_2,\ldots)$ and $Y:=(y_1,y_2,\ldots)$ be two sequences
of independent variables. There is an action of $S_\infty$ on
$\Z[X,Y]$ by ring automorphisms, defined by letting the simple
reflections $s_i$ act by interchanging $x_i$ and $x_{i+1}$ while
leaving all the remaining variables fixed. Define the {\em divided
  difference operator} $\partial_i^x$ on $\Z[X,Y]$ by
\[
\partial_i^xf := \frac{f-s_if}{x_i-x_{i+1}} \ \ \ \text{for $i\geq 1$}.
\]
Consider the ring involution $\pi:\Z[X,Y]\to\Z[X,Y]$ determined by
$\pi(x_i)=-y_i$ and $\pi(y_i)=-x_i$ for each $i$, and set
$\partial_i^y:=\pi\partial_i^x\pi$.

For any $p,r,s\in \Z$, define the polynomial ${}^rh^s_p$ by
\[
{}^rh^s_p:=\sum_{i=0}^p h^r_i(X)e_{p-i}^s(-Y).
\]
We have the following basic lemma.

\begin{lemma}
\label{ddylemmA}
Suppose that $p,r,s\in \Z$.
For all $i\geq 1$, we have
\[
\partial^x_i ({}^rh_p^s)= 
\begin{cases}
{}^{r+1}h_{p-1}^s & \text{if $r=\pm i$}, \\
0 & \text{otherwise}
\end{cases}
\ \quad \mathrm{and} \ \quad
\partial^y_i ({}^rh_p^s)= 
\begin{cases}
{}^rh_{p-1}^{s-1} & \text{if $s=\pm i$}, \\
0 & \text{otherwise}.
\end{cases}
\]
\end{lemma}

The double Schubert polynomials $\AS_\om$ for $\om\in S_{\infty}$ of
Lascoux and Sch\"utzenberger \cite{Las, LS1} are the unique family of
polynomials in $\Z[X,Y]$ such that
\begin{equation}
\label{ddAeqin}
\partial_i^x\AS_\om = \begin{cases}
\AS_{\om s_i} & \text{if $\ell(\om s_i)<\ell(\om)$}, \\ 
0 & \text{otherwise},
\end{cases}
\quad
\partial_i^y\AS_\om = \begin{cases}
\AS_{s_i\om} & \text{if $\ell(s_i\om)<\ell(\om)$}, \\ 
0 & \text{otherwise},
\end{cases}
\end{equation}
for all $i\geq 1$, together with the condition that the constant term
of $\AS_\om$ is $1$ if $\om=1$, and $0$ otherwise.

\subsection{Lie type C}
The Weyl group for the root system of type $\text{C}_n$ is the group
of signed permutations on the set $\{1,\ldots,n\}$, denoted $W_n$. The
group $W_n$ is generated by the simple transpositions $s_i=(i,i+1)$
for $1\leq i \leq n-1$ together with the sign change $s_0$, which
fixes all $j\in [2,n]$ and sends $1$ to $\ov{1}$ (a bar over an
integer here means a negative sign). We write the elements of $W_n$ as
$n$-tuples $(w_1,\ldots, w_n)$, where $w_i:=w(i)$ for each $i\in
[1,n]$.  There is a natural embedding of $W_n$ in $W_{n+1}$ by
adjoining $n+1$ as a fixed point, and we let $W_\infty:=\cup_n W_n$.
The symmetric groups $S_n$ and $S_\infty$ are the subgroups of $W_n$
and $W_\infty$, respectively, generated by the reflections $s_i$ for
$i$ positive.  The {\em length} $\ell(w)$ and the reduced
decompositions of an element $w\in W_\infty$ is defined as in type
A. We have
\[
\ell(w) = \#\{i<j\ |\ w_i>w_j\} + \sum_{i\, :\, w_i <0}|w_i|
\]
for every $w\in W_\infty$.

An element $w\in W_\infty$ has a {\em right descent} (respectively, a
{\em left descent}) at position $i\geq 0$ if $\ell(ws_i)<\ell(w)$
(respectively, if $\ell(s_iw)<\ell(w)$). The signed permutation
$w=(w_1,w_2,\ldots)$ has a right descent at 0 if and only if $w_1<0$,
and a right descent at $i \geq 1$ if and only if $w_i>w_{i+1}$.  The
element $w$ has a left descent at 0 if and only if $w^{-1}(1)<0$, that
is, $w=(\cdots \ov{1} \cdots)$. The element $w$ has a left descent at
$i \geq 1$ if and only if $w^{-1}(i)>w^{-1}(i+1)$, that is, $w$ has one of
the following four forms:
\[
(\cdots i+1 \cdots i \cdots), \quad
(\cdots i \cdots \ov{i+1} \cdots), \quad
(\cdots \ov{i+1} \cdots i \cdots), \quad 
(\cdots \ov{i} \cdots \ov{i+1} \cdots).
\]

Let $w\in W_\infty$ be a signed permutation. Following
\cite[Def.\ 2]{T6}, the strict partition $\mu=\mu(w)$ is the one whose
parts are the absolute values of the negative entries of $w$, arranged
in decreasing order. The {\em A-code} of $w$ is the sequence
$\gamma=\gamma(w)$ with $\gamma_i:=\#\{j>i\ |\ w_j<w_i\}$. We define a
partition $\nu=\nu(w)$ by
\[
\nu_j = \#\{i\ |\ \gamma_i\geq j\}, \ \ \text{for all $j\geq 1$}.
\]
Finally, the {\em shape} of $w$ is the partition
$\la(w):=\mu(w)+\nu(w)$.  The element $w$ is uniquely determined by
$\mu(w)$ and $\gamma(w)$, and we have $|\la(w)|=\ell(w)$.

\begin{example}
(a) For the signed permutation $w := (\ov{5},3, \ov{4}, 7, \ov{1}, \ov{6}, 2)$ 
in $W_7$, we obtain $\mu = (6,5,4,1)$, $\gamma=(1, 4, 1, 3, 1,0,0)$, 
$\nu= (5,2, 2, 1)$, and $\la = (11, 7, 6, 2)$.

\smallskip \noin (b) Let $k\geq 0$. An element $w\in W_\infty$ is
$k$-Grassmannian if $\ell(ws_i)>\ell(w)$ for all $i\neq k$. This is
equivalent to the conditions
\[
0<w_1<\cdots < w_k \quad \mathrm{and} \quad  w_{k+1}<w_{k+2}<\cdots.
\]
If $w$ is a $k$-Grassmannian element of $W_\infty$, then $\la(w)$ is
the $k$-strict partition associated to $w$ in \cite[Sec.\ 6.1]{BKT2}.

\smallskip
\noin (c) Suppose that the first right descent of $w\in W_n$ is $k\geq
0$, and let $m=\ell(\mu)$ and $\ell=\ell(\la)$. Then $\mu$ is a strict
partition and $\nu\subset k^{n-k}+\delta_{n-k-1}$, with $\nu_j\geq k$
for all $j\in [1,m]$. It follows that
\[
\la_1>\la_2>\cdots > \la_m >\max(\la_{m+1},k) 
\geq \la_{m+1}\geq \la_{m+2}\geq \cdots \geq \la_{\ell}.
\]
\end{example}

\begin{lemma}[\cite{M, T6}]
\label{Mcdlemma}
If $i\geq 1$, $w\in W_\infty$, and $\gamma=\gamma(w)$, then 
\[
\gamma_i>\gamma_{i+1} \Leftrightarrow w_i>w_{i+1} \Leftrightarrow
\ell(ws_i)=\ell(w) -1.
\]
If any of the above conditions hold, then 
\[
\gamma(ws_i) = (\gamma_1,\ldots, \gamma_{i-1},
\gamma_{i+1},\gamma_i-1,\gamma_{i+2},\gamma_{i+3},\ldots).
\]
\end{lemma}

Let $c:=(c_1,c_2,\ldots)$ be a sequence of commuting
variables, and set $c_0:=1$ and $c_p:=0$ for $p<0$. Consider the graded
ring $\Gamma$ which is the quotient of the polynomial ring $\Z[c]$
modulo the ideal generated by the relations
\begin{equation}
\label{basicrelsN}
c_pc_p+2\sum_{i=1}^p(-1)^ic_{p+i}c_{p-i}=0, \ \ \ \text{for all $p\geq 1$}.
\end{equation}

Let $X:=(x_1,x_2,\ldots)$ and $Y:=(y_1,y_2,\ldots)$ be two sequences
of variables. Following \cite{BH, IMN}, there is an action of
$W_\infty$ on $\Gamma[X,Y]$ by ring automorphisms, defined as
follows. The simple reflections $s_i$ for $i\geq 1$ act by
interchanging $x_i$ and $x_{i+1}$ while leaving all the remaining
variables fixed. The reflection $s_0$ maps $x_1$ to $-x_1$, fixes the
$x_j$ for $j\geq 2$ and all the $y_j$, and satisfies
\[
s_0(c_p) := c_p+2\sum_{j=1}^p x_1^jc_{p-j} \ \ \text{for all $p\geq 1$}.
\]
For each $i\geq 0$, define the {\em divided difference operator}
$\partial_i^x$ on $\Gamma[X,Y]$ by
\[
\partial_0^xf := \frac{f-s_0f}{-2x_1}, \qquad
\partial_i^xf := \frac{f-s_if}{x_i-x_{i+1}} \ \ \ \text{for $i\geq 1$}.
\]
Consider the ring involution $\varphi:\Gamma[X,Y]\to\Gamma[X,Y]$
determined by
\[
\varphi(x_j) = -y_j, \qquad
\varphi(y_j) = -x_j, \qquad
\varphi(c_p)=c_p
\]
and set $\partial_i^y:=\varphi\partial_i^x\varphi$ for each $i\geq 0$.
The right and left divided difference operators $\partial^x_i$ and
$\partial^y_i$ on $\Gamma[X,Y]$ satisfy the right and left Leibnitz
rules
\begin{equation}
\label{LeibRN}
\partial^x_i(fg) = (\partial^x_if)g+(s_if)\partial^x_ig \quad \ \mathrm{and} \quad \
\partial^y_i(fg) = (\partial^y_if)g+(s^y_if)\partial^y_ig,
\end{equation}
where $s_i^y:=\varphi s_i \varphi$, for every $i\geq 0$.

For any $p,r,s\in \Z$, define the polynomial ${}^rc^s_p$ by
\[
{}^rc^s_p:=\sum_{i=0}^p\sum_{j=0}^p c_{p-i-j}e^{r}_i(X)h_j^s(-Y).
\]
We have the following basic lemma, which stems from \cite[Sec.\ 5.1]{IM}.

\begin{lemma}
\label{ddlemN}
{\em (a)} Suppose that $p,r,s\in \Z$. For all $i\geq 0$, we have
\[
\partial^x_i ({}^rc_p^s)= 
\begin{cases}
{}^{r-1}c_{p-1}^s & \text{if $r=\pm i$}, \\
0 & \text{otherwise}.
\end{cases}
\ \quad \mathrm{and} \ \quad
\partial^y_i ({}^rc_p^s)= 
\begin{cases}
{}^rc_{p-1}^{s+1} & \text{if $s=\pm i$}, \\
0 & \text{otherwise}.
\end{cases}
\]

\medskip
\noin
{\em (b)} For all $i\geq 1$, $r,s\geq 0$, and indices $p$ and $q$, 
we have
\[
\partial^y_i({}^rc_p^{-i}\,{}^sc_q^i) = {}^rc_{p-1}^{-i+1}\,{}^sc_q^{i+1} +
{}^rc_p^{-i+1}\,{}^sc_{q-1}^{i+1}.
\]
\end{lemma}

Suppose $r,s\geq 0$, and let $\frakc_p:={}^rc^{-s}_p$ for each $p\in \Z$.
We then have the relations
\begin{equation}
\label{fundrelsC}
\frakc_p\frakc_p+2\sum_{i=1}^p (-1)^i \frakc_{p+i}\frakc_{p-i}=0 \ \ \text{for
  all} \ \, p > r+s
\end{equation}
in $\Gamma[X,Y]$.  Indeed, if ${\mathcal C}(t):=\sum_{p=0}^{\infty}
\frakc_pt^p$ is the generating function for the $\frakc_p$, we have
\[
{\mathcal C}(t) = \prod_{i=1}^r(1+x_it)\prod_{j=1}^s(1-y_jt)
\left(\sum_{p=0}^{\infty}c_pt^p\right)
\]
and hence
\[
{\mathcal C}(t){\mathcal C}(-t) = 
\prod_{i=1}^r(1-x^2_it^2)\prod_{j=1}^s(1-y^2_jt^2),
\]
which is a polynomial in $t$ of degree $2(r+s)$.

The type C double Schubert polynomials $\CS_w$ for $w\in W_{\infty}$
of Ikeda, Mihalcea, and Naruse \cite{IMN} are the unique family of
elements of $\Gamma[X,Y]$ such that
\begin{equation}
\label{ddCeqinitN}
\partial_i^x\CS_w = \begin{cases}
\CS_{ws_i} & \text{if $\ell(ws_i)<\ell(w)$}, \\ 
0 & \text{otherwise},
\end{cases}
\quad
\partial_i^y\CS_w = \begin{cases}
\CS_{s_iw} & \text{if $\ell(s_iw)<\ell(w)$}, \\ 
0 & \text{otherwise},
\end{cases}
\end{equation}
for all $i\geq 0$, together with the condition that the constant term
of $\CS_w$ is $1$ if $w=1$, and $0$ otherwise.

\subsection{Lie types B and D}
\label{LtBD}

When working with the orthogonal Lie types, we use coefficients in the
ring $\Z[\frac{1}{2}]$. For any $w\in W_\infty$, the type B double
Schubert polynomial $\BS_w$ of \cite{IMN} satisfies
$\BS_w=2^{-s(w)}\CS_w$, where $s(w)$ is the number of indices $j$ such
that $w_j<0$. The odd orthogonal case is therefore entirely similar to
the symplectic case.  In the rest of this section we provide the
corresponding preliminaries for the even orthogonal group, that is, in
Lie type D, and assume that $n\geq 2$.

The Weyl group $\wt{W}_n$ for the root system $\text{D}_n$ is the 
subgroup of $W_n$ consisting of all signed permutations with an even
number of sign changes.  The group $\wt{W}_n$ is an extension of $S_n$
by the element $s_\Box=s_0s_1s_0$, which acts on the right by
\[
(w_1,w_2,\ldots,w_n)s_\Box=(\ov{w}_2,\ov{w}_1,w_3,\ldots,w_n).
\]
There is a natural embedding $\wt{W}_n\hookrightarrow \wt{W}_{n+1}$ of
Weyl groups, induced by the embedding $W_n\hookrightarrow W_{n+1}$,
and we let $\wt{W}_\infty := \cup_n \wt{W}_n$. The elements of the set
$\N_\Box :=\{\Box,1,\ldots\}$ index the simple reflections in
$\wt{W}_\infty$.  The {\em length} $\ell(w)$ and reduced
decompositions of an element $w\in \wt{W}_\infty$ are defined as
before. We have
\[
\ell(w) = \#\{i<j\ |\ w_i>w_j\} + \sum_{i\, :\, w_i <0}(|w_i|-1)
\]
for every $w\in \wt{W}_\infty$.

An element $w\in \wt{W}_\infty$ has a {\em right descent} (respectively, a 
{\em left descent}) at position $i\in\N_\Box$ if $\ell(ws_i)<\ell(w)$ 
(respectively, if $\ell(s_iw)<\ell(w)$). The element
$w=(w_1,w_2,\ldots)$ has a right descent at $\Box$ if and only if $w_1<-w_2$,
and a right descent at $i\geq 1$ if and only if $w_i>w_{i+1}$. We use the notation 
$\wh{1}$ to denote $1$ or $\ov{1}$, determined by the parity of the number of 
negative entries of $w$. The following result
corrects \cite[Lemma 4]{T4}:

\begin{lemma}
\label{WDlem}
Suppose that $w$ is an element of $\wt{W}_\infty$. 

\medskip
\noin
{\em (a)} We have $\ell(s_\Box w)<\ell(w)$ if and only if
$w$ has one of the following four forms:
\[
(\cdots \wh{1} \cdots \ov{2} \cdots), \quad
(\cdots \ov{2} \cdots \ov{1} \cdots), \quad
(\cdots 2 \cdots \ov{1} \cdots).
\]

\noin
{\em (b)} Assume that $i\geq 1$. We have $\ell(s_iw)<\ell(w)$ 
if and only if $w$ has one of the following four forms:
\[
(\cdots i+1 \cdots i \cdots),  \quad
(\cdots i \cdots \ov{i+1} \cdots), \quad
(\cdots \ov{i+1} \cdots i \cdots), \quad
(\cdots\ov{i} \cdots \ov{i+1} \cdots).
\]
\end{lemma}

\begin{defn}
We say that $w$ has {\em type 0} if $|w_1|=1$, {\em type 1} if
$w_1>1$, and {\em type 2} if $w_1<-1$. 
\end{defn}

There is an involution $\iota:\wt{W}_\infty\to \wt{W}_\infty$ which
interchanges $s_\Box$ and $s_1$; we have $\iota(w) = s_0ws_0$ in the
hyperoctahedral group $W_\infty$. We deduce that $\iota(w)=w$ if and
only if $w$ has type 0, while if $w$ has positive type and $|w_r|=1$
for some $r>1$, then
\[
\iota(w)=(-w_1,w_2,\ldots,w_{r-1},-w_r,w_{r+1},\ldots). 
\]
It follows that $\iota$ interchanges type 1 and type 2 elements. The
next definition refines the notion of the shape of an element of
$\wt{W}_\infty$ introduced in \cite[Def.\ 5]{T6}.

\begin{defn}
Let $w\in \wt{W}_\infty$ have type 0 or type 1. The strict partition
$\mu(w)$ is the one whose parts are the absolute values of the
negative entries of $w$ minus one, arranged in decreasing order.  Let
$\gamma=\gamma(w)$ be the A-code of $w$, and define the parts of the
partition $\nu=\nu(w)$ by $\nu_j:=\#\{i\ |\ \gamma_i\geq j\}$. If $w$
has type $2$, then set $\mu(w):=\mu(\iota(w))$,
$\gamma(w):=\gamma(\iota(w))$, and $\nu(w):=\nu(\iota(w))$. 

A {\em typed partition} is a pair consisting of a partition $\la$
together with an integer $\type(\la)\in \{0,1,2\}$. The {\em shape} of
$w$ is the typed partition $\la=\la(w)$ defined by
$\la(w):=\mu(w)+\nu(w)$, with $\type(\la):=\type(w)$. 
\end{defn}

Observe that the element $w$ is uniquely determined by $\mu(w)$,
$\gamma(w)$, and $\type(w)$. Moreover, we have $|\la(w)|=\ell(w)$.

\begin{defn}
Let $w\in\wt{W}_\infty\ssm\{1\}$, let $d$ denote the first right
descent of $w$, and set $k:=d$, if $d\neq \Box$, and $k:=1$, if
$d=\Box$. We call $k$ the {\em primary index} of $w$.
\end{defn}

\begin{example}
(a) For the signed permutation $w := (3, 2, \ov{7}, 1, 5, 4, \ov{6})$
  in $\wt{W}_7$, we obtain $\mu = (6,5)$, $\gamma=(4,3, 0, 1, 2, 1,
  0)$, $\nu= (5, 3, 2,1)$, and $\la = (11, 8, 2,1)$ with
  $\type(\la)=1$. The element $\iota(w)= (\ov{3}, 2, \ov{7},
    \ov{1}, 5, 4, \ov{6})$ has shape $\ov{\la} = (11, 8, 2,1)$ with
  $\type(\ov{\la})=2$. Both $w$ and $\iota(w)$ have primary index $k=1$.

\smallskip \noin (b) Let $k\geq 1$. An element $w$ of $\wt{W}_\infty$
is $k$-Grassmannian if $\ell(ws_i)>\ell(w)$ for all $i\neq k$, if
$k>1$, and for all $i\notin\{\Box,1\}$, if $k=1$. This is equivalent
to the conditions
\[
|w_1|<w_2<\cdots < w_k \quad \mathrm{and} \quad  w_{k+1}<w_{k+2}<\cdots,
\]
the first condition being vacuous if $k=1$. If $w$ is a
$k$-Grassmannian element of $\wt{W}_\infty$, then $\la(w)$ is the
typed $k$-strict partition associated to $w$ in
\cite[Sec.\ 6.1]{BKT3}.

\smallskip \noin (c) Suppose that the primary index of $w\in
\wt{W}_n$ is $k\geq 1$, and let $m=\ell(\mu)$ and
$\ell=\ell(\la)$. Then $\mu$ is a strict partition and $\nu\subset
k^{n-k}+\delta_{n-k-1}$, with $\nu_j \geq k$ for all $j\in [1,m]$. We
therefore have
\[
\la_1>\la_2>\cdots > \la_m >\max(\la_{m+1},k) \geq \la_{m+1}\geq \la_{m+2}\geq \cdots \geq \la_{\ell}.
\]
\end{example}

Let $b:=(b_1,b_2,\ldots)$ be a sequence of commuting variables, and
set $b_0:=1$ and $b_p:=0$ for $p<0$. Consider the graded ring $\Gamma'$
which is the quotient of the polynomial ring $\Z[b]$ modulo the ideal
generated by the relations
\[
b_pb_p+2\sum_{i=1}^{p-1}(-1)^i b_{p+i}b_{p-i}+(-1)^p b_{2p}=0, \ \ \ 
\text{for all $p\geq 1$}.
\]
We regard $\Gamma$ as a subring of $\Gamma'$ via the injective ring
homomorphism which sends $c_p$ to $2b_p$ for every $p\geq 1$.

Following \cite{BH, IMN}, we define an action of $\wt{W}_\infty$ on
$\Gamma'[X,Y]$ by ring automorphisms as follows. The simple
reflections $s_i$ for $i \geq 1$ act by interchanging $x_i$ and $x_{i+1}$
and leaving all the remaining variables fixed. The reflection $s_\Box$
maps $(x_1,x_2)$ to $(-x_2,-x_1)$, fixes the $x_j$ for $j\geq 3$ and
all the $y_j$, and satisfies, for any $p\geq 1$,
\begin{align*}
s_\Box(b_p) :=
b_p+(x_1+x_2)\sum_{j=0}^{p-1}\left(\sum_{a+b=j}x_1^ax_2^b\right)
c_{p-1-j}.
\end{align*}
For each $i\in \N_\Box$, define the divided difference operator
$\partial_i^x$ on $\Gamma'[X,Y]$ by
\[
\partial_\Box^xf := \frac{f-s_\Box f}{-x_1-x_2}, \qquad
\partial_i^xf := \frac{f-s_if}{x_i-x_{i+1}} \ \ \ \text{for $i\geq 1$}.
\]
Consider the ring involution $\varphi':\Gamma'[X,Y]\to\Gamma'[X,Y]$
determined by
\[
\varphi'(x_j) = -y_j, \qquad
\varphi'(y_j) = -x_j, \qquad
\varphi'(b_p)=b_p
\]
and set $\partial_i^y:=\varphi'\partial_i^x\varphi'$ for each $i\in \N_\Box$.
The right and left divided difference operators $\partial^x_i$ and
$\partial^y_i$ on $\Gamma'[X,Y]$ satisfy the right and left Leibnitz
rules
\begin{equation}
\label{LeibRND}
\partial^x_i(fg) = (\partial^x_if)g+(s_if)\partial^x_ig \quad \ \mathrm{and} \quad \
\partial^y_i(fg) = (\partial^y_if)g+(s^y_if)\partial^y_ig,
\end{equation}
where $s_i^y:=\varphi' s_i \varphi'$, for every $i\in \N_\Box$.

Let $r\geq 0$ and set $\dis {}^rc_p:= \sum_{i=0}^p
c_{p-i}h^{-r}_i(X)$.  Define ${}^rb_p := {}^rc_p$ for $p<r$, $\dis
{}^rb_p := \frac{1}{2}{}^rc_p$ for $p>r$, and set
\[
{}^rb_r := \frac{1}{2}{}^rc_r + \frac{1}{2}e^r_r(X) \quad \text{and} \quad
{}^r\wt{b}_r := \frac{1}{2}{}^rc_r - \frac{1}{2}e^r_r(X).
\]
For $s\in \{0,1\}$, let
$\dis {}^ra^s_p:=\frac{1}{2}{}^rc_p + \sum_{i=1}^p {}^rc_{p-i}h^s_i(-Y)$, and define 
\[
{}^rb^s_r:={}^rb_r + \sum_{i=1}^r {}^rc_{r-i}h^s_i(-Y),
\quad \text{and} \quad
{}^r\wt{b}^s_r:={}^r\wt{b}_r + \sum_{i=1}^r {}^rc_{r-i}h^s_i(-Y).
\]

We have the following propositions, which are proved as in \cite[Sec.\ 2]{T4}.

\begin{prop}
\label{prop1new}
Suppose that $p,q\in \Z$ and $r,s\geq 1$.

\medskip
\noin
{\em (a)} For all $i \geq 1$, we have 

\[\partial^x_i ({}^rc_p^q)= 
\begin{cases}
{}^{r-1}c_{p-1}^q & \text{if $r=\pm i$}, \\
0 & \text{otherwise}
\end{cases}
\quad \ and \ \quad
\partial^y_i ({}^rc_p^q)= 
\begin{cases}
{}^rc_{p-1}^{q+1} & \text{if $q=\pm i$}, \\
0 & \text{otherwise}.
\end{cases}
\]
We have 
\[
\partial^y_\Box \left({}^rc_p^q\right)= 
\begin{cases}
{}^rc_{p-1}^2 & \text{if $q=1$}, \\
2\left({}^rc_{p-1}^2\right) & \text{if $q=0$}, \\
2\left({}^rc_{p-1}^1\right) -{}^rc_{p-1} & \text{if $q=-1$}, \\
0 & \text{if $|q|\geq 2$}.
\end{cases}
\]

\medskip
\noin
{\em (b)} For all $i\geq 1$, we have 
\[
\partial^y_i({}^rc_p^{-i}\,{}^sc_q^i) = {}^rc_{p-1}^{-i+1}\,{}^sc_q^{i+1} +
{}^rc_p^{-i+1}\,{}^sc_{q-1}^{i+1}.
\]
\end{prop}

We also require certain variations of the above identities. 
Let $f_r$ be an indeterminate of degree $r$, which will equal
${}^rb_r$, ${}^r\wt{b}_r$, or $\dis\frac{1}{2}\,{}^rc_r$ in the sequel.
We also let $f_0 \in\{0,1\}$. For any $p,s\in \Z$, define
${}^r\wh{c}_p^{\,s}$ by
\[
{}^r\wh{c}_p^{\,s}:= {}^rc_p^s + 
\begin{cases}
(2f_r-{}^rc_r)e^{p-r}_{p-r}(-Y) & \text{if $s = r - p < 0$}, \\
0 & \text{otherwise}.
\end{cases}
\]
For $s\in \{0,1\}$, define
\[
f_r^s:= f_r+\sum_{j=1}^r {}^rc_{r-j}h_j^s(-Y), 
\]
set $\wt{f}_r:={}^rc_r-f_r$ and $\wt{f}_r^s:={}^rc_r-2f_r+f_r^s$.

\begin{prop}
\label{prop1hnew}
Suppose that $p\in \Z$ and $p>r$.

\medskip
\noin
{\em (a)} For all $i \geq 1$, we have
\[
\partial^x_i({}^r\wh{c}_p^{\,r-p}) =
\begin{cases}
  {}^{r-1}\wh{c}_{p-1}^{\,r-p} & \text{if $i=p-r\geq 2$}, \\
  2\varphi'(f_r) & \text{if $i=p-r=1$}, \\
0 & \text{otherwise}
\end{cases}
\]
and 
\[
\partial^y_i ({}^r\wh{c}_p^{\,r-p})= 
\begin{cases}
  {}^r\wh{c}_{p-1}^{\,r-p+1} & \text{if $i=p-r\geq 2$}, \\
  2f_r & \text{if $i=p-r= 1$}, \\
0 & \text{otherwise}.
\end{cases}
\]
We have
\[
\partial^y_\Box \left({}^r\wh{c}_p^{\,r-p}\right)= 
\begin{cases}
2\wt{f}^1_r & \text{if $r-p=-1$}, \\
0 & \text{if $r-p < -1$}.
\end{cases}
\]

\medskip
\noin
{\em (b)} For all $i\geq 2$, we have 
\[
\partial^y_i({}^r\wh{c}_p^{\,-i}\,{}^sc_q^i) = {}^r\wh{c}_{p-1}^{\,-i+1}\,{}^sc_q^{i+1} +
{}^r\wh{c}_p^{\,-i+1}\,{}^sc_{q-1}^{i+1}.
\]
\end{prop}

Fix $r,s\geq 0$, and define $\frakc_p:={}^rc^{-s}_p$ for each $p\in \Z$. 
For $p=r+s$, set $\frakd_p:=e^r_r(X)e^s_s(-Y)$. Then, in addition to the relations
(\ref{fundrelsC}), we have the relations
\begin{equation}
\label{fundrelD}
(\frakc_p+\frakd_p)(\frakc_p-\frakd_p)+
2\sum_{i=1}^p (-1)^i \frakc_{p+i}\frakc_{p-i}=0 \ \ \text{for} \ \, p = r+s
\end{equation}
in $\Gamma'[X,Y]$.

Following \cite{IMN}, the type D double Schubert polynomials $\DS_w$
for $w\in \wt{W}_{\infty}$ are the unique family of elements of
$\Gamma'[X,Y]$ satisfying the equations
\begin{equation}
\label{Dddeq}
\partial_i^x\DS_w = \begin{cases}
\DS_{ws_i} & \text{if $\ell(ws_i)<\ell(w)$}, \\ 
0 & \text{otherwise},
\end{cases}
\quad
\partial_i^y\DS_w = \begin{cases}
\DS_{s_iw} & \text{if $\ell(s_iw)<\ell(w)$}, \\ 
0 & \text{otherwise},
\end{cases}
\end{equation}
for all $i\in \N_\Box$, together with the condition that the constant
term of $\DS_w$ is $1$ if $w=1$, and $0$ otherwise.

\section{Raising operators}
\label{ro}

For each pair $i<j$ of distinct positive integers, the operator 
$R_{ij}$ acts on integer sequences $\al=(\al_1,\al_2,\ldots)$ by
\[
R_{ij}(\al):=(\alpha_1,\ldots,\alpha_i+1,\ldots,\alpha_j-1,
\ldots).
\] 
A {\em raising operator} $R$ is any monomial in these $R_{ij}$'s.

Following \cite[Sec.\ 1.2]{BKT2}, let $\Delta^\circ := \{(i,j) \in \Z
\times \Z \mid 1\leq i<j \}$ and define a partial order on
$\Delta^\circ$ by agreeing that $(i',j')\leq (i,j)$ if $i'\leq i$ and
$j'\leq j$.  We call a finite subset $D$ of $\Delta^\circ$ a {\em
  valid set of pairs} if it is an order ideal in $\Delta^\circ$.  Any
valid set of pairs $D$ defines the raising operator expression
\[
R^D := \prod_{i<j}(1-R_{ij})\prod_{i<j\, :\, (i,j)\in D}(1+R_{ij})^{-1}.
\]
We also use the raising operator expressions
\[
R^{\emptyset}:=\prod_{i<j}(1-R_{ij}) \ \quad \mathrm{and} \ \quad 
R^{\infty}:= \prod_{i<j}\frac{1-R_{ij}}{1+R_{ij}}.
\]

\subsection{Alternating properties in types A, B, and C}
 
For each $r\geq 1$, let $\s^r=(\s^r_i)_{i\in \Z}$ be a sequence of
variables, with $\s^r_0=1$ and $\s^r_i=0$ for each $i<0$, and let
$\Z[\s]$ denote the polynomial ring in the variables $\s^r_i$ for
$i,r\geq 1$. For any integer sequence $\al$, let
$\s_\al:=\s^1_{\al_1}\s^2_{\al_2}\cdots$, and for any raising operator
$R$, set $R\, \s_\al:=\s_{R\al}$.

Fix $j\geq 1$, let $z$ be a variable, set $\tau^r:=\s^r$ for each
$r\neq j$ and $\tau^j_p = \s^j_p+z\,\s^j_{p-1}$ for each $p\in \Z$.
If $\al:=(\al_1,\ldots,\al_\ell)$ and
$\al':=(\al'_1,\ldots,\al'_{\ell'})$ are two integer vectors and
$r,s\in \Z$, we let $(\al,r,s,\al')$ denote the integer vector
$(\al_1,\ldots,\al_\ell,r,s,\al'_1,\ldots,\al'_{\ell'})$.  The
following two lemmas are generalizations of \cite[Lemmas 1.2 and
  1.3]{BKT2}.

\begin{lemma}\label{commuteA}
Let $\lambda=(\lambda_1,\ldots,\lambda_{j-1})$ and
$\mu=(\mu_{j+2},\ldots,\mu_\ell)$ be integer vectors, and $D$ be a
valid set of pairs.  Assume that $\s^j=\s^{j+1}$, $(j,j+1)\notin D$,
and that for each $h<j$, $(h,j)\in D$ if and only if $(h,j+1)\in D$.

\medskip
\noin
{\em (a)} For any integers $r$ and $s$, we have
\[ R^D\,\s_{\lambda,r,s,\mu} = - R^D\, \s_{\lambda,s-1,r+1,\mu}   \]
in $\Z[\s]$. 

\medskip
\noin
{\em (b)} For any integer $r$, we have
\[
R^D\,\tau_{\lambda,r,r,\mu} = R^D\, \s_{\lambda,r,r,\mu}   
\]
in $\Z[\s,z]$. 
\end{lemma}
\begin{proof}
The proof of (a) is identical to that of \cite[Lemma 1.2]{BKT2}.  For
part (b), we use linearity in the $j$-th position to obtain
$R\,\tau_{\lambda,r,r,\mu} =R\, \s_{\lambda,r,r,\mu} + z\,R\,
\s_{\lambda,r-1,r,\mu}$, for any raising operator $R$ that appears in
the expansion of the power series $R^D$. Adding these equations gives
\[
R^D\,\tau_{\lambda,r,r,\mu} = 
R^D\, \s_{\lambda,r,r,\mu} + z\,R^D\, \s_{\lambda,r-1,r,\mu}.
\]
Now part (a) implies that $R^D\, \s_{\lambda,r-1,r,\mu}=0$.
\end{proof}

Fix $k\geq 0$, let $\frakc=(\frakc_i)_{i\in\Z}$ be another sequence of variables, 
and consider the relations
\begin{equation}
\label{fundrels}
\frakc_p\frakc_p+2\sum_{i=1}^p (-1)^i \frakc_{p+i}\frakc_{p-i}=0 \ \ \text{for all} \ \, p > k.
\end{equation}

\begin{lemma}\label{commuteC}
Let $\lambda=(\lambda_1,\ldots,\lambda_{j-1})$ and
$\mu=(\mu_{j+2},\ldots,\mu_\ell)$ be integer vectors, and $D$ be a
valid set of pairs. Assume that $\s^j=\s^{j+1}=\frakc$,
$(j,j+1)\in D$, and that for each $h>j+1$, $(j,h)\in D$ if and only if
$(j+1,h)\in D$. 

\medskip
\noin
{\em (a)}
If $r,s\in \Z$ are such that $r+s > 2k$, then we have
\[ 
R^D\,\s_{\lambda,r,s,\mu} = - R^D\,\s_{\lambda,s,r,\mu}  
\]
in the ring $\Z[\s]$ modulo the relations coming from {\em (\ref{fundrels})}.

\medskip
\noin
{\em (b)} For any integer $r>k$, we have
\[
R^D\,\tau_{\lambda,r+1,r,\mu} = R^D\, \s_{\lambda,r+1,r,\mu}   
\]
in the ring $\Z[\s,z]$ modulo the relations coming from {\em (\ref{fundrels})}.
\end{lemma}
\begin{proof}
The proof of (a) is identical to that of \cite[Lemma 1.3]{BKT2}.  For
part (b), we expand $R^D\,\tau_{\lambda,r+1,r,\mu}$ and use linearity
in the $j$-th position to obtain
\[
R^D\,\tau_{\lambda,r+1,r,\mu} = 
R^D\, \s_{\lambda,r+1,r,\mu} + z\,R^D\, \s_{\lambda,r,r,\mu}.
\]
Now part (a) implies that $R^D\, \s_{\lambda,r,r,\mu}$ vanishes modulo the
relations (\ref{fundrels}).
\end{proof}

\subsection{Alternating properties in type D} 

In type D we will require certain variations of Lemma \ref{commuteA}
and Lemma \ref{commuteC}.  For each $r\geq 1$, we introduce a new
sequence of variables $\up^r=(\up^r_i)_{i\in \Z}$ such that
$\up^r_i=0$ for each $i\leq 0$.  Let $\Z[\s, \up]$ denote the
polynomial ring in the variables $\s^r_i$, $\up^r_i$ for $i,r\geq
1$. For each $r\geq 1$, define the sequence $\wh{\s}^r$ by
$\wh{\s}^r_i:=\s^r_i+(-1)^r\up^r_i$ for each $i$, and for any integer
sequence $\al$, let
$\wh{\s}_\al:=\wh{\s}^1_{\al_1}\wh{\s}^2_{\al_2}\cdots$.

Fix an integer $d\geq 0$ such that $\up^r_i=0$ for all $i$ whenever
$r>d$.  If $R:=\prod_{i<j} R_{ij}^{n_{ij}}$ is a raising operator,
denote by $\supp_d(R)$ the set of all indices $i$ and $j$ such that
$n_{ij}>0$ and $j \leq d$.  Let $D$ be a valid set of pairs and $R$ be
any raising operator appearing in the expansion of the power series
$R^D$. Let $\la=(\la_1,\ldots,\la_\ell)$ be any integer vector and set
$\rho:=R\la$. Define
\[
R \star \wh{\s}_{\la} = \ov{\s}_{\rho} :=
\ov{\s}^1_{\rho_1}\cdots \ov{\s}^{\ell}_{\rho_\ell}
\]
where, for each $i\geq 1$ and $p\in \Z$, 
\[
\ov{\s}^i_p:= 
\begin{cases}
\s^i_p & \text{if $i\in\supp_d(R)$}, \\
\wh{\s}^i_p & \text{otherwise}.
\end{cases}
\]

Fix $j\geq 1$, set $\wh{\tau}^i:=\wh{\s}^i$ for each $i\neq j$ and
$\wh{\tau}^j_p = \wh{\s}^j_p+z\,\wh{\s}^j_{p-1}$ for each $p\in \Z$.

\begin{lemma}\label{commuteAD}
Let $\lambda=(\lambda_1,\ldots,\lambda_{j-1})$ and
$\mu=(\mu_{j+2},\ldots,\mu_\ell)$ be integer vectors, and $D$ be a valid set of 
pairs.  Assume that $j>d$, $\s^j=\s^{j+1}$, 
$(j,j+1)\notin D$, and that for each $h<j$, $(h,j)\in D$ if and only if
$(h,j+1)\in D$.

\medskip
\noin
{\em (a)} For any integers $r$ and $s$, we have
\[ R^D\star\wh{\s}_{\lambda,r,s,\mu} = - R^D\star \wh{\s}_{\lambda,s-1,r+1,\mu}   \]
in $\Z[\s,\up]$. 

\medskip
\noin
{\em (b)} For any integer $r$, we have
\[
R^D\star\wh{\tau}_{\lambda,r,r,\mu} = R^D\star \wh{\s}_{\lambda,r,r,\mu}   
\]
in $\Z[\s,\up,z]$. 
\end{lemma}
\begin{proof}
Since $j>d$, the argument used in the proof of \cite[Lemma 1.2]{BKT2}
works here as well to establish part (a). Part (b) is an easy consequence of (a).
\end{proof}

Fix $k\geq 0$, let $\frakc=(\frakc_i)_{i\in\Z}$ and
$\frakd=(\frakd_i)_{i\in\Z}$ be two other sequences of variables such
that $\frakd_p=0$ for all $p>k+1$, and consider the relations
\begin{equation}
\label{fundrelsD}
(\frakc_p+\frakd_p)(\frakc_p-\frakd_p)+
2\sum_{i=1}^p (-1)^i \frakc_{p+i}\frakc_{p-i}=0 \ \ \text{for all} \ \, p > k.
\end{equation}

\begin{lemma}\label{commuteCD}
Let $\lambda=(\lambda_1,\ldots,\lambda_{j-1})$ and
$\mu=(\mu_{j+2},\ldots,\mu_\ell)$ be integer vectors, and $D$ be a
valid set of pairs. Assume that $j<d$, $\s^j=\s^{j+1}=\frakc$,
$\up^j=\up^{j+1}=\frakd$,
$(j,j+1)\in D$, and that for each $h>j+1$, $(j,h)\in D$ if and only if
$(j+1,h)\in D$. 

\medskip
\noin
{\em (a)} If $r,s\in \Z$ are such that $r+s > 2k+2$, then we have
\begin{equation}
\label{Drel1} 
R^D\star\wh{\s}_{\lambda,r,s,\mu} = - R^D\star\wh{\s}_{\lambda,s,r,\mu}  
\end{equation}
and 
\begin{equation}
\label{Drel2}
R^D\star\wh{\s}_{\lambda,k+1,k+1,\mu} = 0
\end{equation}
in the ring $\Z[\s,\up]$ modulo the relations coming from {\em
  (\ref{fundrelsD})}.

\medskip
\noin
{\em (b)} For any integer $r>k$, we have
\[
R^D\star\wh{\tau}_{\lambda,r+1,r,\mu} = R^D\star \wh{\s}_{\lambda,r+1,r,\mu}   
\]
in the ring $\Z[\s,\up,z]$ modulo the relations coming from {\em
  (\ref{fundrelsD})}.
\end{lemma}
\begin{proof}
The proof of (\ref{Drel1}) is identical to that of \cite[Lemma
  1.3]{BKT2}. The proof of (\ref{Drel2}) follows the same argument,
using (\ref{Drel1}) and induction to reduce to the case when $\mu$ is
empty. For any integer vector $\rho$ with at most $d$ components,
define $T_\rho:=R^D \star \wh{\s}_{\rho}$.  If $g$ is the least integer
such that $2g\geq \ell$ and $\rho:= (\la,r,s)$, then we have the
relation
\[
T_\rho = \sum_{j=2}^{2g} (-1)^j T_{\rho_1,\rho_j} T_{\rho_2,\ldots,\wh{\rho}_j,\ldots,\rho_{2g}}.
\]
The proof is now completed by induction, as in loc.\ cit. For part
(b), we expand $R^D\star\wh{\tau}_{\lambda,r+1,r,\mu}$ and use
linearity in the $j$-th position to obtain
\[
R^D\star\wh{\tau}_{\lambda,r+1,r,\mu} = 
R^D\star \wh{\s}_{\lambda,r+1,r,\mu} + z\,R^D\, \wh{\s}_{\lambda,r,r,\mu}.
\]
Now part (a) implies that $R^D\star \wh{\s}_{\lambda,r,r,\mu}$ vanishes modulo the
relations (\ref{fundrelsD}).
\end{proof}

\section{Amenable elements: Type A theory}
\label{aeA}

\subsection{Definitions and main theorem}
\label{Aamenable}

As Lie theorists know well, type A is very special when compared to
the other Lie types. In the theory of amenable elements, this
manifests itself in the fact that we can work with dominant elements
instead of leading elements. The result is the simplified treatment
given here, which does not have a direct analogue in types B, C, and
D. Another difference in type A is that the order of application of
the divided difference operators is switched: we first use the left
divided differences, then the right ones. But by far the main
distinction between type A and the other classical types is that one
can use Jacobi-Trudi determinants, represented here by
$R^{\emptyset}$, instead of the more general raising operator
expressions $R^D$ that define theta and eta polynomials, which are
essential ingredients of the theory for the symplectic and orthogonal
groups.

If $\om\in S_n$ and $v\in S_m$, then $\om$ is called {\em
  $v$-avoiding} if $\om$ does not contain a subword
$(\om_{i_1},\ldots, \om_{i_m})$ having the same relative order as
$(v_1,\ldots,v_m)$.  The notion of $v$-avoidance also makes sense when
$\om$ is any integer vector $(\om_1,\ldots,\om_n)$ with distinct
components $\om_i$. We say that $\om$ is {\em dominant} if its code
$\gamma(\om)$ is a partition, or equivalently, if $\om$ is
$132$-avoiding (see \cite[(1.30)]{M} and \cite[Thm.\ 2.2]{R}).

For the next result, we refer to \cite[Exercise 2.2.1.5]{Kn} and
\cite[\S 2.2]{Stu}.

\begin{lemma}
\label{312avoid}
The following conditions on a permutation $\ome \in S_n$ are equivalent:

\medskip
\noin
{\em (a)} $\ome$ is $312$-avoiding; \qquad \qquad
{\em (b)} $\ome^{-1}$ is $231$-avoiding;

\medskip
\noin {\em (c)} $\ome$ has a reduced decomposition of the form $R_1\cdots
R_{n-1}$ where each $R_j$ is a (possibly empty) subword of $s_1\cdots
s_{n-1}$ and furthermore all simple reflections in $R_p$ are also
contained in $R_{p+1}$, for each $p<n-1$. 
\end{lemma}

\begin{defn}
\label{amenA}
A (right) {\em modification} of $\om\in S_n$ is a permutation $\om
\ome$, where $\ome\in S_n$ is such that $\ell(\om
\ome)=\ell(\om)-\ell(\ome)$, and $\ome$ is $231$-avoiding. A
permutation is {\em amenable} if it is a modification of a
dominant permutation.
\end{defn}

For any three integer vectors $\al,\be,\rho\in \Z^\ell$, which we view
as integer sequences with finite support, define
${}^{\rho}h^\be_\al:={}^{\rho_1}h^{\be_1}_{\al_1}\,{}^{\rho_2}h^{\be_2}_{\al_2}\cdots$.
Given any raising operator $R=\prod_{i<j}R_{ij}^{n_{ij}}$, let
$R\, {}^{\rho}h^\be_{\al} := {}^{\rho}h^\be_{R\al}$.

\begin{prop}[{\cite[(6.14)]{M}}]
\label{Adom}
Suppose that $\om\in S_n$ is dominant. Then we have
\[
\AS_\om = R^{\emptyset} \, {}^{\delta^\vee_{n-1}}h^{\la(\om)}_{\la(\om)}.
\]
\end{prop}
\begin{proof}
We use descending induction on $\ell(\om)$. Let $\om_0:=(n,\ldots,1)$
denote the longest element in $S_n$. One knows from \cite{Las} and
\cite[(3.5)]{M} that the equation
\[
\AS_{\om_0} = R^{\emptyset}\,
   {}^{\delta^\vee_{n-1}}h_{\delta_{n-1}}^{\delta_{n-1}}
\]
holds in $\Z[X,Y]$, so the result is true when $\om=\om_0$.

Suppose that $\om\neq \om_0$ and $\om$ is dominant of shape
$\la$. Then $\la\subset \delta_{n-1}$ and $\la\neq \delta_{n-1}$. Let
$r\geq 1$ be the largest integer such that $\la_i=n-i$ for $i\in
[1,r]$, and let $j:=\la_{r+1}+1=\om_{r+1}\leq n-r-1$.  Then $s_j\om$
is dominant of length $\ell(\om)+1$ and
$\la(s_j\om)=\la(\om)+\epsilon_{r+1}$.  Using Lemma \ref{ddylemmA} and
the left Leibnitz rule, we deduce that for any integer sequence
$\al=(\al_1,\ldots,\al_n)$, we have
\[
\partial^y_j\left({}^{\delta^\vee_{n-1}}h^{\la(s_j\om)}_{\al}\right) = 
{}^{\delta^\vee_{n-1}}h^{\la(\om)}_{\al-\epsilon_{r+1}}.
\]
We conclude that
\[
\AS_\om = \partial^y_j(\AS_{s_j\om}) = 
\partial^y_j\left(R^{\emptyset} \, {}^{\delta^\vee_{n-1}}h^{\la(s_j\om)}_{\la(s_j\om)}\right)=
R^{\emptyset} \, {}^{\delta^\vee_{n-1}}h^{\la(\om)}_{\la(\om)}.
\]
\end{proof}

\begin{defn}
Let $\om$ be an amenable permutation with code $\gamma$ and shape
$\la$, with $\ell=\ell(\la)$. Define two sequences $\f=\f(\om)$ and
$\g=\g(\om)$ of length $\ell$ as follows. For $1\leq j \leq \ell$,
set
\[
\f_j:=\max(i\ |\ \gamma_i\geq \la_j)
\]
and let
\[
\g_j:=\f_\ci+\la_\ci-\ci,
\]
where $\ci$ is the least integer such that $\ci\geq j$ and
$\la_\ci>\la_{\ci+1}$. We call $\f$ the {\em right flag} of $\om$, and
$\g$ the {\em left flag} of $\om$.
\end{defn}

It is clear from Lemma \ref{Mcdlemma} that the right flag $\f$ of an
amenable permutation is a weakly increasing sequence consisting of
right descents of $\om$. We will show that the left flag $\g$ is a
weakly decreasing sequence consisting of left descents of $\om$.

\begin{prop}
\label{amenableA}
Suppose that $\wh{\om}\in S_n$ is dominant with $\wh{\la}:=\la(\wh{\om})$.
Let $\ome$ be a $231$-avoiding permutation such that 
$\ell(\wh{\om}\ome)=\ell(\wh{\om})-\ell(\ome)$, and set $\om:=\wh{\om} \ome$, 
$\gamma:=\gamma(\om)$, and $\la:=\la(\om)$.  Then the sequence
$\delta^\vee_{n-1}+\wh{\la}-\la$ is weakly increasing, and 
\[
\AS_\om = R^{\emptyset} \, {}^{\delta^\vee_{n-1}+\wh{\la}-\la}h_{\la}^{\wh{\la}}.
\]
Moreover, if $\la_\ci>\la_{\ci+1}$, then $\wh{\la}_\ci$ is a left descent of 
$\om$, $\ci+\wh{\la}_\ci-\la_\ci$ is a right descent of $\om$, and we have
$\ci+\wh{\la}_\ci-\la_\ci= \max(i\ |\ \gamma_i\geq \la_\ci)$.
\end{prop}
\begin{proof}
Suppose that $\wh{\om}$ is of shape $\wh{\la} = \wh{\gamma}
= (p_1^{n_1},p_2^{n_2}\ldots,p_t^{n_t})$, where $p_1>\cdots > p_t$.
Then the right descents of
$\wh{\om}$ are at positions $d_1:=n_1, d_2:=n_1+n_2, \ldots, d_t:=
n_1+\cdots+n_t$.  Since we have $\wh{\om}_j<\wh{\om}_{j+1}$ for all $j\neq d_r$
for $r\in [1,t]$, we deduce that
\[
\wh{\om}_1 = p_1+1, \wh{\om}_{d_1+1}=p_2+1,\ldots, \wh{\om}_{d_{t-1}+1}=p_t+1,
\wh{\om}_{d_t+1}=1.
\]
Moreover, since $\wh{\om}$ is $132$-avoiding, it follows that the left
descents of $\wh{\om}$ are $p_1,\ldots,p_t$. Finally, Proposition
\ref{Adom} gives
\begin{equation}
\label{propAdcor}
\AS_{\wh{\om}} =  R^{\emptyset} \, {}^{\delta^\vee_{n-1}}h^{\wh{\la}}_{\wh{\la}}
\end{equation}
so the result holds when $\ome=1$ and $\om=\wh{\om}$ is dominant.

Suppose next that $\om:=\wh{\om} \ome$ for some $231$-avoiding
permutation $\ome$ such that
$\ell(\wh{\om}\ome)=\ell(\wh{\om})-\ell(\ome)$. Lemma
\ref{312avoid} implies that $\ome$ has a reduced decomposition of the
form $R_1\cdots R_{n-1}$ where each $R_j$ is a (possibly empty)
subword of $s_{n-1}\cdots s_1$ and all simple reflections in $R_{p+1}$
are also contained in $R_p$, for every $p\geq 1$.  Now repeated
application of (\ref{ddAeqin}), Lemma \ref{ddylemmA}, and the right
Leibnitz rule (\ref{LeibRN}) in equation (\ref{propAdcor}) give
\[
\AS_\om= \partial^x_{\ome^{-1}}(\AS_{\wh{\om}}) = R^{\emptyset} \,
   {}^{\delta_{n-1}^{\vee}
     +\wh{\la}-\la}h^{\wh{\la}}_{\la}.
\]
We will show that the sequence $\delta^\vee_{n-1}+\wh{\la}-\la$
is weakly increasing and verify the last assertion, about the left
and right descents of $\om$.

Using Lemma \ref{Mcdlemma}, we study the right action of the
successive simple transpositions in the reduced decomposition
$R_1\cdots R_{n-1}$ for $\ome$ on the code $\wh{\gamma}$ of
$\wh{\om}$.  The action of these on $\wh{\gamma}$ is by a finite
sequence of {\em moves} $\al\mapsto \al'$, where $\al:=\gamma(v)$ and
$\al':=\gamma(v')$ for some $v,v'\in S_n$. Here $v'=vs_{j-1}\cdots
s_i$ for some $i<j$ such that $\ell(v')=\ell(v)-j+i$, and
$s_{j-1}\cdots s_i$ is a subword of some $R_p$ with $j-i$
maximal. Since the initial code $\wh{\gamma}$ is weakly decreasing, we
have $\al_i\geq \cdots \geq \al_{j-1}>\al_j$, and
\[
\al'=(\al_1,\ldots,\al_{i-1},\al_j,\al_i-1,\ldots,\al_{j-1}-1,\al_{j+1},
\al_{j+2},\ldots).
\]
We say that the move is performed on the interval $[i,j]$, or is an
$[i,j]$-move. The {\em procedure} is defined as the performance of
finitely many moves to $\wh{\gamma}$, ending in the code
$\gamma$. This describes the effect of multiplying $\wh{\om}$ on the
right by $\ome$.

\begin{example}
Suppose that $\wh{\om}:=(5,6,7,4,3,8,2,1)$ in $S_8$ with code 
$$\wh{\gamma}=(4,4,4,3,2,2,1,0).$$ If $\ome:=s_7s_6s_5s_4s_3s_2s_7s_6s_5s_4s_6s_5$,
then $\wh{\om} \ome = (5,1,6,2,3,7,4,8)$ and $\gamma=\gamma(\om \ome) = (4,0,3,0,0,1,0,0)$.
The procedure from $\wh{\gamma}$ to $\gamma$ consists of a $[2,8]$-move, followed by a 
$[4,8]$-move, followed by a $[5,7]$-move:
\[
(4,4,4,3,2,2,1,0) \mapsto (4,0,3,3,2,1,1,0) \mapsto 
(4,0,3,0,2,1,0,0)\mapsto (4,0,3,0,0,1,0,0).
\]
\end{example}

Notice that after an $[i,j]$-move $\al\mapsto \al'$, we have
\begin{equation}
\label{moabeq}
\al'_i=\gamma_i=\min(\al'_r\ |\ r\in [i,j]) \geq \max(\al'_r\ | \ r>j).
\end{equation}
Let $\mu$ and $\mu'$ be the shapes of $v$ and $v'$, respectively, and 
set
$f:=\delta^\vee_{n-1}+\wh{\la}-\mu$
(respectively, $f':=\delta^\vee_{n-1}+\wh{\la}-\mu'$). We then have
\[
\mu'=(\mu_1,\ldots,\mu_{r-1},\mu_r-1,\ldots,\mu_{s-1}-1,\mu_s,\mu_{s+1},\ldots)
\]
for some $r<s$ with $s-r=j-i$, and 
\[
f'=(f_1,\ldots,f_{r-1},f_r+1,\ldots,f_{s-1}+1,f_s,f_{s+1},\ldots).
\]
Since $\mu_s = \al_j \leq \al_{j-1}-1 = \mu_{s-1}-1$, we deduce that
$f'_s-f'_{s-1} = f_s-f_{s-1}-1 = \mu_s-\mu_{s-1}-1\geq 0$. It follows
by induction on the number of moves that the sequence $f$ is weakly
increasing. 

Suppose that $\mu'_d>\mu'_{d+1}$ for some $d$. Using (\ref{moabeq})
and induction on the number of moves, we deduce that
$f'_d=\max(i\ |\ \al'_i\geq \mu'_d)$. This implies that for any $\ci$
such that $\la_\ci>\la_{\ci+1}$, we have $\ci+\wh{\la}_\ci-\la_\ci =
\max(i\ |\ \gamma_i\geq \la_\ci)$, and hence that
$\ci+\wh{\la}_\ci-\la_\ci$ is a right descent of $\om$, in view of
Lemma \ref{Mcdlemma}.

We claim that $\wh{\la}_d$ is a left descent of $v'$. Clearly the left
descents of $v$ and $v'$ are subsets of $\{p_1,\ldots, p_t\}$. There
is at most one left descent $p_e$ of $v$ that is not a left descent of
$v'$, and this occurs if and only if $v_j=p_e$ and $v_h=p_e+1$ for
some $h\in [i,j-1]$. Since $\al_h\geq \cdots \geq \al_{j-1}>\al_j$, we
deduce that $\al_h=\cdots = \al_{j-1}=\al_j+1$, and hence
$\mu'_{s-(j-h)+1} = \cdots = \mu'_s=\mu_s$. We conclude that
$\wh{\la}_d\neq p_e$, completing the proof of the claim, and the
proposition.
\end{proof}

\begin{example}
Let $\wh{\om}:=(4,5,6,2,1,3)$, a dominant permutation in $S_6$ with
shape $\wh{\la} = \gamma(\wh{\om})=(3,3,3,1)$. Take
$\ome:=s_4s_3s_2s_1s_4s_3$ in Proposition \ref{amenableA}, so that
$\om=\wh{\om} \ome = (1,4,2,5,6,3)$, with
$\gamma(\wh{\om}\ome)=(0,2,0,1,1,0)$ and $\la=(2,1,1)$. We have
$\delta_5^\vee+\wh{\la}-\la= (2,4,5,5,5)$, and deduce that
\[
\AS_\om = R^{\emptyset} \, {}^{(2,4,5,5,5)}h_{(2,1,1,0)}^{(3,3,3,1)} =
R^{\emptyset} \, {}^{(2,4,5)}h_{(2,1,1)}^{(3,3,3)}.
\]
\end{example}

\begin{thm}
\label{TamvexA}
For any amenable permutation $\om$, we have 
\[
\AS_\om = R^{\emptyset} \, {}^{\f(\om)}h_{\la(\om)}^{\g(\om)}.
\]
\end{thm}
\begin{proof}
We may assume we are in the situation of Proposition
\ref{amenableA}, so that $\om=\wh{\om} \ome$, with
$\wh{\la}=\la(\wh{\om})$ and $\la=\la(\om)$.  Choose $j\in
    [1,\ell]$ and let $\ci$ be the least integer such that $\ci\geq j$
    and $\la_\ci>\la_{\ci+1}$. Then we have $\la_j=\la_{j+1}=\cdots =
    \la_\ci$.  As the sequence $f:=\delta_{n-1}^\vee+\wh{\la}-\la$ is
    weakly increasing, we deduce that if $\la_r=\la_{r+1}>0$, then
    either (i) $\wh{\la}_r = \wh{\la}_{r+1}$ and $f_r=f_{r+1}-1$, or
    (ii) $\wh{\la}_r = \wh{\la}_{r+1}+1$ and $f_r=f_{r+1}$. Theorem
    \ref{TamvexA} follows from this and induction on $\ci-j$, using
    Lemma \ref{commuteA}(b) in Proposition \ref{amenableA}.
\end{proof}

\begin{example}
Consider the amenable permutation $\om:=(3,4,6,1,5,2)$ in $S_6$. We
then have $\gamma(\om)=(2,2,3,0,1,0)$, $\la(\om)=(3,2,2,1)$,
$\f(\om)=(3,3,3,5)$, and $\g(\om)=(5,2,2,2)$. Theorem \ref{TamvexA}
gives
\[
\AS_{346152} = R^{\emptyset} \, {}^{(3,3,3,5)}h_{(3,2,2,1)}^{(5,2,2,2)}.
\]
\end{example}

Recall from \cite{LS1, LS2} that a permutation $\om$ is {\em
  vexillary} if and only if it is $2143$-avoiding.  Equivalently,
$\om$ is vexillary if and only if $\la(\om^{-1})=\la(\om)'$.

\begin{thm}
\label{amenvexthm}
The permutation $\om$ is amenable if and only if $\om$ is vexillary.
\end{thm}
\begin{proof}
According to \cite[(1.32)]{M}, a permutation is vexillary if and only
if its code $\gamma$ satisfies the following two conditions, for any
$i<j$: (i) If $\gamma_i\leq \gamma_j$, then $\gamma_i\leq\gamma_k$ for
any $k$ with $i<k<j$; (ii) If $\gamma_i>\gamma_j$, then the number of
$k$ with $i<k<j$ and $\gamma_k<\gamma_j$ is at most
$\gamma_i-\gamma_j$.

Assume first that $\om$ is amenable, so that $\om=\wh{\om}\ome$ for
some dominant permutation $\wh{\om}$ and $231$-avoiding permutation
$\ome$. Using Lemma \ref{Mcdlemma}, we see that the code $\wh{\gamma}$
of $\wh{\om}$ is transformed into the code $\gamma$ of $\om$ by the
moves of the procedure described in the proof of Proposition
\ref{amenableA}.

We claim that the sequence $\gamma$ is a vexillary code.  It follows
from the inequalities (\ref{moabeq}) that for any $[i,j]$-move of the
procedure, we have $\gamma_s\leq \gamma_i\leq \gamma_r$ for every
$r\in [i,j]$ and $s>j$. Moreover, if $r\neq i$ for all $[i,j]$-moves
of the procedure, then $\gamma_s\leq \gamma_r$ for all $s>r$.  It is
easy to see from this that $\gamma$ satisfies the vexillary conditions
(i) and (ii). Indeed, choose $r<s$ such that $\gamma_r\leq \gamma_s$,
and some $t\in [r,s]$.  If $r=i$ for some $[i,j]$-move, then we have
$\gamma_r\leq \gamma_k$ for all $k\in [i,j]$, while if $k>j$, then we
must have $\gamma_r=\gamma_s=\gamma_k$. Therefore, $\gamma_r\leq
\gamma_t$. If $r\neq i$ for all $[i,j]$-moves, then
$\gamma_r=\gamma_s$ and hence $\gamma_r=\gamma_t$. To prove (ii),
suppose that $r<k<s$ and $\gamma_r>\gamma_s>\gamma_k$. Then we must
have $k=i$ for some $[i,j]$-move of the procedure, where $s\leq j$. We
conclude that the number of such $k$ is at most $\gamma_r-\gamma_s$.

Conversely, suppose that $\om\in S_n$ is a vexillary permutation with
code $\gamma$. We call an integer $i\geq 1$ an {\em initial index} if
there exists an $s>i$ with $\gamma_i<\gamma_s$. 
We claim that there is a canonical $312$-avoiding permutation $\ome$
such that $\ell(\om\ome)= \ell(\om)+\ell(\ome)$, $\om\ome$ is
dominant, and $\ome_a=a$ if $a<i$ for every initial index $i$. This
will complete the proof of the theorem, by applying Lemma \ref{312avoid}.

To establish the claim, we argue by descending induction on the length of
$\om$. Observe that $\om$ has no initial index if and only if $\om$ is
a dominant permutation.  Hence, if $\om$ is already dominant, then we
must take $\ome$ to be the identity.  

Assume that $\om$ is not dominant.  We say that the index $j$ is
{\em associated} to the initial index $i$ of $\om$ if $j$ is the
maximum $s$ such that $\gamma_i<\gamma_s$.  Let $i$ be the smallest
initial index, let $j$ be associated to $i$, and set $\om':=\om
s_i\cdots s_{j-1}$. The vexillary condition (i) and Lemma \ref{Mcdlemma} imply that
\[
\gamma(\om') = (\gamma_i,\ldots, \gamma_{i-1}, \gamma_{i+1}+1,\ldots,
\gamma_j+1,\gamma_i,\gamma_{j+1},\ldots)
\]
and $\ell(\om')=\ell(\om)+j-i>\ell(\om)$.  It follows by checking
conditions (i) and (ii) that $\om'$ is vexillary and that every
initial index $i'$ of $\om'$ satisfies $i'\geq i$.

By the inductive hypothesis, there exists a canonical $312$-avoiding 
permutation $\ome'$ with $\ell(\om'\ome')=\ell(\om')+\ell(\ome')$, $\om'\ome'$ 
dominant, and $\ome'_a=a$ if $a<i'$ for every initial index $i'$ of $\om'$.
The claim is proved with $\ome:=s_i\cdots s_{j-1}\ome'$, once we check that 
$\ome$ is $312$-avoiding. Indeed, since $\ome'_a=a$ for all $a<i$ and
$\ell(s_i\cdots s_{j-1}\ome')=j-i+\ell(\ome')$, we must have
\[
\ome'=(1,2,\ldots,i-1,\ldots,a_1,\ldots,a_2, \ldots, a_{j-i},\ldots,j,\ldots)
\]
and 
\[
\ome = (1,2,\ldots,i-1,\ldots,a_1+1,\ldots,a_2+1, \ldots, a_{j-i}+1,\ldots,i,\ldots)
\]
where the set $\{a_1,\ldots,a_{j-i}\}$ is equal to $\{i,\ldots,
j-1\}$. As $\ome'$ is $312$-avoiding, there are no integers $a<b<c$
such that $\ome'_a>\ome'_c >\ome'_b$. It is easy to see from this and
the above relation between $\ome'$ and $\ome$ that the latter
permutation has the same property, and therefore is also
$312$-avoiding.
\end{proof}

\begin{remark}
(a) Define a {\em left modification} of $\om\in S_n$ to be a
permutation $\ome\om$, where $\ome\in S_n$ is $312$-avoiding and such
that $\ell(\ome\om)=\ell(\om)-\ell(\ome)$.  Then a permutation is
amenable if and only if it is a left modification of a dominant
permutation. This follows from Lemma \ref{312avoid}, Theorem
\ref{amenvexthm}, and the fact that $\om$ is dominant
(respectively vexillary) if and only if $\om^{-1}$ is dominant
(respectively vexillary).

\medskip
\noin (b) It is not hard to show that a definition of amenable
permutations as left modifications of leading permutations, in the
same manner as Definition \ref{amenC} in type C, results in the same
class of permutations as that given in Definition \ref{amenA}.
\end{remark}

Let $\om$ be a vexillary permutation with code $\gamma$ and shape
$\la$, and let $\ome$ be the canonical $312$-avoiding permutation associated 
to $\om$ in the proof of Theorem \ref{amenvexthm}.  Define a new sequence
$\wh{\gamma}$ by the prescription
\[
\wh{\gamma}_\al:=\gamma_\al+\#\{i\ |\ \text{$i$ is an initial 
index with associated index $j$ and $i<\al\leq j$}\}
\]
for each $\al\geq 1$. Let $\wh{\la}$ be the partition obtained by
listing the entries of $\wh{\gamma}$ in weakly decreasing order. 
Then $\wh{\la}$ is the shape of $\om \ome$.

Consider the skew Young diagram $\tau(\om):=\wh{\la}/\la$.  For each
$i\geq 1$, fill the boxes in row $i$ of $\tau(\om)$ with a strictly
decreasing sequence of consecutive positive integers ending in $i$. In
this way, we obtain a tableaux $T=T(\om)$ of shape $\tau(\om)$ with
strictly decreasing rows. Define the {\em depth} of a box $B$ of $T$
to be the distance from $B$ to the end of the row it occupies. Form a
reduced decomposition for a permutation $\ome_T$ by listing the
entries in the boxes of $T$ in decreasing order of depth, with the
entries of a fixed depth listed in {\em increasing} order. It then
follows from the definition of $\ome$ that $\ome_T=\ome$.

\begin{example}
Let $\om:=(1,3,6,7,9,4,8,2,5)$ be the vexillary permutation in $S_9$
with code $\gamma=(0,1,3,3,4,1,2,0,0)$. The initial indices are $1$,
$2$, $3$, $4$, and $6$ with associated indices $7$, $7$, $5$, $5$, and
$7$, respectively. The reduced decomposition for the canonical
permutation $\ome$ is
\[          
s_4s_3s_4s_6s_2s_3s_4s_5s_6s_1s_2s_3s_4s_5s_6
\]
and we have $\om \ome=(9,7,6,8,4,3,1,2,5)$, with code
$(8,6,5,5,3,2,0,0,0)$. We also have 
$\la=(4,3,3,2,1,1)$, $\wh{\gamma}=(0,2,5,6,8,3,5,0)$, and
$\wh{\la}=(8,6,5,5,3,2)$. The tableau $T(\om)$ on the skew diagram
$\tau(\om)$ is displayed in Figure \ref{skewfig}. 
\begin{figure}
\centering
\includegraphics[scale=0.30]{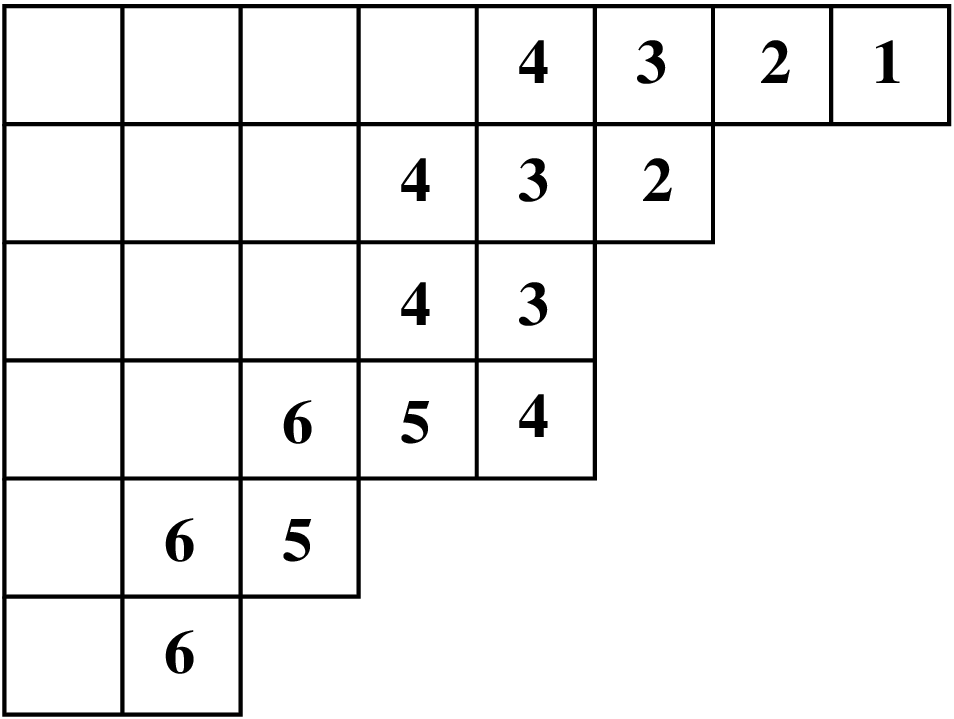}
\caption{The tableau $T$ on the skew diagram $\wh{\la}/\la$.}
\label{skewfig}
\end{figure}
\end{example}

It would be interesting to find analogues of the canonical permutation 
$\ome$ and the tableaux $T(\om)$ for the amenable elements in the 
other classical Lie types.

\subsection{Type A degeneracy loci}

Let $E\to \X$ be a vector bundle of rank $n$ on a complex algebraic variety
$\X$, assumed to be smooth for simplicity. Let $\om\in S_n$ be amenable 
of shape $\la$, and let $\f$ and $\g$ be the left and right flags of 
$\om$, respectively. Consider two complete flags of subbundles of $E$
\[
0 \subset E_1\subset \cdots \subset E_n=E \ \ \, \mathrm{and} \, \ \ 
0 \subset F_1\subset \cdots \subset F_n=E
\]
with $\rank E_r=\rank F_r=r$ for each $r$. Define
the {\em degeneracy locus} $\X_\om\subset \X$ as the locus of $x \in \X$
such that
\[
\dim(E_r(x)\cap F_s(x))\geq \#\,\{\,i \leq r \ |\ \om_i> n-s\,\}
\ \, \forall \, r,s.
\]
Assume further that $\X_\om$ has pure codimension $\ell(\om)$ in
$\X$. The next result, which follows from Theorem \ref{TamvexA} and
Fulton's work \cite{F1}, will be a formula for the cohomology class
$[\X_\om]$ in $\HH^{2\ell(\om)}(\X)$ in terms of the Chern classes of
the bundles $E_r$ and $F_s$. Recall that for any integer $p$, the
class $c_p(E-E_r-F_s)$ is defined by the equation
\[
c(E-E_r-F_s) := c(E)c(E_r)^{-1}c(F_s)^{-1}
\]
of total Chern classes.

\begin{thm}[\cite{F1}]
\label{dbleAloci}
For any amenable permutation $\om\in S_n$, we have
\begin{equation}
\label{AChern}
[\X_\om] = s_{\la'}(E-E_{\f}-F_{n-\g}) =
R^\emptyset\, c_{\la}(E-E_{\f}-F_{n-\g})
\end{equation}
in the cohomology ring $\HH^*(\X)$.
\end{thm}

The Chern polynomial in
(\ref{AChern}) is interpreted as the image of the Schur polynomial
$s_{\la'}(\mathfrak{c}) :=R^\emptyset \mathfrak{c}_{\la}$ under the
$\Z$-linear map which sends the noncommutative monomial
$\mathfrak{c}_\al$ to $\prod_j c_{\al_j}(E-E_{\f_j}-F_{n-\g_j})$, for
every integer sequence $\al$.

\begin{remark}
Theorem \ref{dbleAloci} and its companion Theorems \ref{dbleCloci},
\ref{dbleBloci}, and \ref{dbleDloci} in the other classical Lie types
are results about cohomology groups, taken with rational coefficients
in types B and D.  However, from these, one may obtain corresponding
results for cohomology with integer coefficients, and for the Chow
groups of algebraic cycles modulo rational equivalence. For the latter
transition, see \cite{F2, G}.
\end{remark}

\section{Amenable elements: Type C theory}
\label{Ctheory}

\subsection{Definitions and main theorem}
\label{dmtC}

Let $w$ be a signed permutation with A-code $\gamma$ and
shape $\la=\mu+\nu$, with $\ell=\ell(\la)$ and $m=\ell(\mu)$.
Choose $k\geq 0$, and assume that $w$ is increasing up to $k$. If 
$k=0$, this condition is vacuous, while if $k\geq 1$ it means that
$0<w_1<\cdots<w_k$. Eventually, $k$ will be the first right descent of
$w$, but the increased flexibility is useful.

List the entries $w_{k+1},\ldots,w_n$ in increasing order:
\[
u_1<\cdots < u_m < 0 < u_{m+1} < \cdots < u_{n-k}.
\]
Define a sequence $\beta(w)$ by
\[
\beta(w):=(u_1+1,\ldots,u_m+1,u_{m+1},\ldots,u_{n-k}).
\]
and the {\em denominator set} $D(w)$ by
\begin{equation}
\label{Dset}
D(w):=\{(i,j) \ |\ 1\leq i<j \leq n-k \ \, \text{and} \ \, 
 u_i+ u_j < 0 \}.
\end{equation}
This notation suppresses the dependence of $\beta(w)$ and $D(w)$ on
$k$. Observe that the inequality $u_i+u_j<0$ in (\ref{Dset}) is
equivalent to $\beta_i(w)+ \beta_j(w) \leq 0$.

\begin{defn}
\label{Ctrunc}
Suppose that $w\in W_n$ has code $\gamma=\gamma(w)$ and $k\geq 0$.
The {\em $k$-truncated A-code} ${}^k\gamma={}^k\gamma(w)$ is defined
by
\[
{}^k\gamma(w):=(\gamma_{k+1}, \gamma_{k+2},\ldots, \gamma_n).
\]
If $k$ is the first right descent of $w$, then we call ${}^k\gamma(w)$
the {\em truncated A-code} of $w$. We let $\xi=\xi(w)$ be the
conjugate of the partition whose parts are the non-zero entries of
${}^k\gamma(w)$ arranged in weakly decreasing order.
\end{defn}

Clearly an element $w\in W_n$ increasing up to $k$ with a given
$k$-truncated A-code $C$ is uniquely determined by the {\em set} of
elements $\{w_{k+1},\ldots,w_n\}$, or equivalently, by the sequence
$\beta(w)$.

Let $v(w)$ be the unique $k$-Grassmannian element obtained by
reordering the entries $w_{k+1},\ldots,w_n$ to be increasing.  For
example, if $w:=(2,4,7,5,8,\ov{3},1,\ov{6})$ and $k:=3$, then
$v(w)=(2,4,7,\ov{6},\ov{3},1,5,8)$.  Note that the map $w \mapsto
v(w)$ is a bijection from the set of elements in $W_n$ increasing up
to $k$ with $k$-truncated A-code $C$ onto the set of $k$-Grassmannian
elements in $W_n$, such that $\beta(v(w))=\beta(w)$ and
\[
\ell(v(w)) = \ell(w)-\sum_{i=k+1}^n C_i = \ell(w)-\sum_{j=1}^{n-k} \gamma_{k+j}.
\]
In particular, if $w,\ov{w}$ are two such elements, then 
$\ell(w)>\ell(\ov{w})$ if and only if $\ell(v(w))>\ell(v(\ov{w}))$.

\begin{lemma}
Let $w$ and $\ov{w}$ be elements in $W_n$ increasing up to $k$
and with the same $k$-truncated A-code $C$, 
such that $\ell(w)=\ell(\ov{w})+1$. Suppose that 
$v(\ov{w})=s_iv(w)$ for some simple reflection $s_i$. 
Then $\ov{w}=s_iw$.
\end{lemma}
\begin{proof} 
There are 4 possible cases for $i$ and $v(w)$: 
(a) $i=0$ and $v(w) = (\cdots \ov{1} \cdots)$; 
(b) $i \geq 1$ and $v(w) = (\cdots i \cdots \ov{i+1} \cdots)$; 
(c) $i \geq 1$ and $v(w) = (\cdots \ov{i+1} \cdots i \cdots)$;
(d) $i \geq 1$ and $v(w) = (\cdots i+1 \cdots i \cdots)$. In the 
first three cases, the result is clear. In case (d), the
$i+1$ must be among the first $k$ entries of $v(w)$, which
coincide with the first $k$ entries of $w$, while $i$ lies among
the last $n-k$ entries of $w$. Hence the the result follows.
\end{proof}

For any three integer vectors $\al,\be,\rho\in \Z^\ell$, define
${}^{\rho}c^\be_\al:={}^{\rho_1}c^{\be_1}_{\al_1}\,{}^{\rho_2}c^{\be_2}_{\al_2}\cdots$.
Given any raising operator $R=\prod_{i<j}R_{ij}^{n_{ij}}$, let $R\,
{}^{\rho}c^\be_{\al} := {}^{\rho}c^\be_{R\al}$.

\begin{prop}
\label{gampropN}
Fix an integer $k\geq 0$. Suppose that $w$ and $\ov{w}$ are elements
in $W_n$ increasing up to $k$ with the same $k$-truncated A-code $C$,
such that $\ell(w)=\ell(\ov{w})+1$ and $s_iv(w)=v(\ov{w})$ for some
simple reflection $s_i$. Assume that we have
\[
\CS_w = R^{D(w)}\,{}^{\kk}c^{\beta(w)}_{\la(w)} 
\]
for some integer sequence $\kk$. Then we have 
\[
\CS_{\ov{w}} = R^{D(\ov{w})}\,{}^{\kk}c^{\beta(\ov{w})}_{\la(\ov{w})}.
\]
\end{prop} 
\begin{proof}
Set $F_w:=R^{D(w)} \, {}^{\kk}c^{\beta(w)}_{\la(w)}$, so we know that
$\CS_w=F_w$. As equation (\ref{ddCeqinitN}) gives $\partial^y_i\CS_w =
\CS_{\ov{w}}$, it will suffice to show that $\partial^y_i F_w =
F_{\ov{w}}$.  The proof of this will follow the argument of
\cite[Prop.\ 5]{TW}.

Let $\mu:=\mu(w)$, $\nu:=\nu(w)$, $\la:=\la(w)=\mu+\nu$, 
$\ov{\mu}:=\mu(\ov{w})$, $\ov{\nu}:=\nu(\ov{w})$,
$\ov{\la}:=\la(\ov{w})= \ov{\mu}+\ov{\nu}$, 
$\beta=\beta(w)$, and $\ov{\beta}=\beta(\ov{w})$. 
There are 4 possible cases for $w$, discussed below. In each case, we have
$\ov{\la}\subset\la$, so that $\ov{\la}_p=\la_p-1$ for some $p\geq 1$ and
$\ov{\la}_j=\la_j$ for all $j\neq p$.

\medskip
\noin (a) $v(w) = (\cdots \ov{1} \cdots)$ with $i=0$. In this case we have 
$D(w)=D(\ov{w})$. Since clearly $\nu=\ov{\nu}$ and $\ov{\mu}_p=\mu_p-1$
for $p=\ell(\mu)$, while $\ov{\mu}_j=\mu_j$ for all $j\neq p$, it follows that
$\beta_p=i$, $\ov{\beta}_p= i + 1$, while $\beta_j=\ov{\beta}_j$ for all $j \neq p$.

\medskip
\noin (b) $v(w) = (\cdots i \cdots \ov{i+1} \cdots)$. In this case
$D(w)=D(\ov{w})$, and we have $\nu=\ov{\nu}$ and $\ov{\mu}_p=\mu_p-1$,
while $\ov{\mu}_j=\mu_j$ for all $j\neq p$. It follows that 
$\beta_p = -i$, $\ov{\beta}_p=-i+1$, and $\beta_j=\ov{\beta}_j$ for all $j \neq p$.

\medskip
\noin (c) $v(w) = (\cdots \ov{i+1} \cdots i \cdots)$. In this case
$D(w)=D(\ov{w})\cup\{(p,q)\}$, where $\beta_p=-i$ and $\beta_q=i$. We see
similarly that $\ov{\beta}_p=-i+1$ and $\ov{\beta}_q=i+1$, while
$\beta_j=\ov{\beta}_j$ for all $j \notin \{p,q\}$.

\medskip
\noin
(d) $v(w) = (\cdots i+1 \cdots i \cdots)$. In this case we have 
$D(w)=D(\ov{w})$ while clearly $\mu=\ov{\mu}$. We deduce that
$\ov{\nu}_p=\nu_p-1$, while $\ov{\nu}_j=\nu_j$ for all $j\neq p$.
We must show that $\beta_p=i$, and hence $\ov{\beta}_p=i+1$, 
and $\beta_j=\ov{\beta}_j$ for all $j\neq p$.

Note that if $w_r=i+1$, then $r\in [1,k]$. Since $w_j\geq w_r$ for all
$j\in [r,k]$, and the sequence $\beta(w)$ is strictly increasing, we
deduce that $\beta_g=i$ exactly when $g=\gamma_r(w)$.  We have
$\gamma_r(\ov{w})=\gamma_r(w)-1=g-1$, while $\gamma_j(w)=\gamma_j(\ov{w})$
for $j\neq r$. It follows that $\ov{\nu}_g=\nu_g-1$, while
$\ov{\nu}_j=\nu_j$ for all $j\neq g$. In other words, $g=p$, as
desired.

\medskip
To simplify the notation, set $c^{\rho}_{\al}:={}^{\kk}c^{\rho}_\al$, for
any integer sequences $\al$ and $\rho$.  In cases (a), (b), or (d), it
follows using the left Leibnitz rule and Lemma \ref{ddlemN}(a) that
for any integer sequence $\al=(\al_1,\ldots,\al_\ell)$, we have
\begin{align*}
\partial^y_i c^{\be}_\al &=
c^{(\be_1,\ldots,\be_{p-1})}_{(\al_1,\ldots,\al_{p-1})}
\left(\partial^y_i(c^{\be_p}_{\al_p})\,c^{(\be_{p+1},\ldots,\be_\ell)}_{(\al_{p+1},\ldots,\al_\ell)}
+s^y_i(c^{\be_p}_{\al_p})\,\partial^y_i
(c^{(\be_{p+1},\ldots,\be_\ell)}_{(\al_{p+1},\ldots,\al_\ell)})\right)
\\ &=c^{(\be_1,\ldots,\be_{p-1})}_{(\al_1,\ldots,\al_{p-1})}
\left(c^{\be_p+1}_{\al_p-1}\,c^{(\be_{p+1},\ldots,\be_\ell)}_{(\al_{p+1},\ldots,\al_\ell)}
+s^y_i(c^{\be_p}_{\al_p})\cdot 0\right)
=c^{(\be_1,\ldots,\be_p+1,\ldots,\be_\ell)}_{(\al_1,\ldots,\al_p-1,\ldots,\al_\ell)}
= c^{\ov{\be}}_{\al-\epsilon_p}.
\end{align*}
Since $\la-\epsilon_p=\ov{\la}$,
we deduce that if $R$ is any raising operator, then
\[
\partial^y_i (R \,c^{\be}_{\la}) = \partial^y_i (c^{\be}_{R\la}) = 
c^{\ov{\be}}_{R\la-\epsilon_p} = R \,c^{\ov{\be}}_{\ov{\la}}.
\]
As $R^{D(w)}= R^{D(\ov{w})}$, we conclude that 
\[
\partial^y_i(F_w) = \partial^y_i (R^{D(w)} c^{\be}_{\la})   =
R^{D(\ov{w})} c^{\ov{\be}}_{\ov{\la}} = F_{\ov{w}}.
\]

In case (c), it follows from the left Leibnitz rule as in the proof of 
Lemma \ref{ddlemN}(b) that for any integer sequence
$\al=(\al_1,\ldots,\al_\ell)$, we have
\begin{align*}
\partial^y_i c^{\be}_\al &= \partial^y_i
c^{(\be_1,\ldots,-i,\ldots,i,\ldots,\be_{\ell})}_{(\al_1,\ldots,\al_p,\ldots,\al_q,
  \ldots,\al_\ell)} \\ &=
c^{(\be_1,\ldots,-i+1,\ldots,i+1,\ldots,\be_{\ell})}_{(\al_1,\ldots,\al_p-1,\ldots,\al_q,\ldots,\al_\ell)}
+c^{(\be_1,\ldots,-i+1,\ldots,i+1,\ldots,\be_{\ell})}_{(\al_1,\ldots,\al_p,\ldots,\al_q-1,\ldots,\al_\ell)}
= c^{\ov{\be}}_{\al-\epsilon_p} + c^{\ov{\be}}_{\al-\epsilon_q}.
\end{align*}
Since $\la-\epsilon_p=\ov{\la}$, we deduce that if $R$ is any raising
operator, then
\[
\partial^y_i (R \,c^{\be}_\la) = \partial^y_i (c^{\be}_{R\la}) = 
c^{\ov{\be}}_{R\la-\epsilon_p} + c^{\ov{\be}}_{R\la-\epsilon_q} = 
R \,c^{\ov{\be}}_{\ov{\la}} + RR_{pq}\,c^{\ov{\be}}_{\ov{\la}}.
\]
As $R^{D(w)}+R^{D(w)} R_{pq} = R^{D(\ov{w})}$, we conclude that
\[
\partial^y_i(F_w) = 
\partial^y_i(R^{D(w)} c^{\be}_\la) = 
R^{D(w)} c^{\ov{\be}}_{\ov{\la}} +R^{D(w)} R_{pq}c^{\ov{\be}}_{\ov{\la}} = 
R^{D(\ov{w})} c^{\ov{\be}}_{\ov{\la}} = F_{\ov{w}}.
\]
\end{proof}

\begin{prop}
\label{leadingN}
Suppose that $w\in W_n$ is such that $\gamma(w)$ is a partition. Then we have
\begin{equation}
\label{domC}
\CS_w = R^{D(w)} \, {}^{\nu(w)}c^{\beta(w)}_{\la(w)}.
\end{equation}
\end{prop}
\begin{proof}
Assume first that $w_i<0$ for each $i$, and $\gamma(w)$ is a
partition. We claim that
\begin{equation}
\label{fneqt}
\CS_w = R^{\infty} \, {}^{\nu(w)}c^{-\delta_{n-1}}_{\la(w)}.
\end{equation}
The proof of (\ref{fneqt}) is by descending induction on
$\ell(w)$. One knows from \cite[Thm.\ 1.2]{IMN} and
\cite[Prop.\ 3.2]{T5} that (\ref{fneqt}) is true for the longest
element $w_0$ in $W_n$, since $\nu(w_0)=\delta_{n-1}$ and
$\la(w_0)=\delta_n+\delta_{n-1}$.

Suppose that $w\neq w_0$ is such that $\gamma(w)$ is a partition, and
the shape of $w$ equals $\delta_n+\nu$. Then $\nu\subset \delta_{n-1}$
and $\nu\neq \delta_{n-1}$. Let $r\geq 1$ be the largest integer such
that $\nu_i=n-i$ for $i\in [1,r]$, and let $j:=\nu_{r+1}+1\leq n-r-1$.
Then $ws_j$ is of length $\ell(w)+1$ and satisfies the same
conditions, $\nu(ws_j)=\nu(w)+\epsilon_{r+1}$, and
$\la(ws_j)=\la(w)+\epsilon_{r+1}$. Using Lemma \ref{ddlemN}(a), for
any integer sequence $\al=(\al_1,\ldots,\al_n)$, we have
\[
\partial^x_j\left({}^{\nu(ws_j)}c^{-\delta_{n-1}}_{\al}\right) = 
{}^{\nu(w)}c^{-\delta_{n-1}}_{\al-\epsilon_{r+1}}.
\]
We deduce that
\[
\CS_w = \partial^x_j(\CS_{ws_j}) = 
\partial^x_j\left(R^{\infty} \, {}^{\nu(ws_j)}c^{-\delta_{n-1}}_{\la(ws_j)}\right)=
R^{\infty} \, {}^{\nu(w)}c^{-\delta_{n-1}}_{\la(w)},
\]
proving the claim. Equation (\ref{domC}) now follows, by combining (\ref{fneqt}) 
with the $k=0$ case of Proposition \ref{gampropN}.
\end{proof}

\begin{cor}
\label{kleadingN}
Suppose that $w\in W_n$ is increasing up to $k$ and the $k$-truncated
A-code ${}^k\gamma$ is a partition. Let
$k^{n-k}+\xi(w)=(k+\xi_1,\ldots,k+\xi_{n-k})$.  Then we have
\begin{equation}
\label{directapp}
\CS_w = R^{D(w)} \, {}^{k^{n-k}+\xi(w)}c^{\beta(w)}_{\la(w)}.
\end{equation}
\end{cor}
\begin{proof}
If $w=(1,\ldots,k,w_{k+1},\ldots,w_n)$ with $w_{k+j}<0$ for all $j\in
[1,n-k]$, then $\gamma(w)=k^{n-k}+\xi(w)$ is a partition, and
(\ref{directapp}) is a direct application of Proposition
\ref{leadingN}. In this case, $v(w)=(1,\ldots,k,-n,\ldots,-k-1)$ is
the longest $k$-Grassmannian element in $W_n$. The general result now
follows from Proposition \ref{gampropN}, as the $k$-Grassmannian
elements of $W_n$ form an ideal for the left weak Bruhat order
(see for example \cite[Prop.\ 2.5]{Ste}).
\end{proof}

\begin{defn}
\label{modifdefn}
Let $k\geq 0$ denote the first right descent of $w\in W_n$. 
List the entries $w_{k+1},\ldots,w_n$ in increasing order:
\[
u_1<\cdots < u_m < 0 < u_{m+1} < \cdots < u_{n-k}.
\]
We say that a simple transposition $s_i$ for $i\geq 1$ is {\em
  $w$-negative} (respectively, {\em $w$-positive}) if $\{i,i+1\}$ is a
subset of $\{-u_1,\ldots,-u_m\}$ (respectively, of
$\{u_{m+1}\ldots,u_{n-k}\}$). Let $\sigma^-$ (respectively,
$\sigma^+$) be the longest subword of $s_{n-1}\cdots s_1$
(respectively, of $s_1\cdots s_{n-1}$) consisting of $w$-negative
(respectively, $w$-positive) simple transpositions.  A {\em
  modification} of $w\in W_n$ is an element $\ome w$, where $\ome\in
S_n$ is such that $\ell(\ome w)=\ell(w)-\ell(\ome)$, and $\ome$ has a
reduced decomposition of the form $R_1\cdots R_{n-1}$ where each $R_j$
is a (possibly empty) subword of $\sigma^-\sigma^+$ and all simple
reflections in $R_p$ are also contained in $R_{p+1}$, for each
$p<n-1$.  
\end{defn}

\begin{defn}
\label{amenC}
Suppose that $w\in W_n$ has first right descent at $k\geq 0$ and A-code
$\gamma$.  We say that $w$ is {\em leading} if $(\gamma_{k+1},
\gamma_{k+2},\ldots, \gamma_n)$ is a partition.  We say that $w$ is
{\em amenable} if $w$ is a modification of a leading element.
\end{defn}

\begin{remark}
(a) The integer vector $\al=(\al_1,\ldots,\al_p)$ 
is called {\em unimodal} if for some $j\in [1,p]$, we have
\[
\al_1\leq \al_2 \leq \cdots \leq \al_j\geq \al_{j+1} \geq \cdots \geq \al_p.
\]
The element $w\in W_n$ is leading if and only if the A-code of the
extended sequence $(0,w_1,w_2,\ldots,w_n)$, where we have set $w(0):=0$,
is unimodal.

\medskip
\noin (b) Given an element $w\in W_n$, there is an easy algorithm to
decide whether or not $w$ is amenable. One simply applies all possible
{\em inverse modifications} to $w$ and checks if any of these result
in a leading element.
\end{remark}

\begin{example}
Consider the leading element $w:=(2,4,6,5,\ov{1},\ov{3})$ in $W_6$,
with $k=3$, $\gamma(w)= (2,2,3,2,1,0)$, $\mu(w)=(3,1)$,
$\nu(w)=(5,4,1)$, $\xi(w)=(2,1)$, and $\la(w)=(8,5,1)$. We have
$\beta(w)=(-2,0,5)$ and $D(w)=\{(1,2)\}$, so Corollary 
\ref{kleadingN} gives
\[
\CS_w = R^{\{12\}} \, {}^{(5,4,3)}c_{(8,5,1)}^{(-2,0,5)} = 
\frac{1-R_{12}}{1+R_{12}}(1-R_{13})(1-R_{23})\, {}^{(5,4,3)}c_{(8,5,1)}^{(-2,0,5)}.
\]
\end{example}

In the following we will assume that $w$ has first right descent at
$k\geq 0$ and $\xi$ is as in Definition \ref{Ctrunc}.  Let
$\psi:=(\gamma_k,\ldots,\gamma_1)$, $\phi:=\psi'$, $\ell:=\ell(\la)$
and $m:=\ell(\mu)$. We then have
\begin{equation}
\label{pxmN}
\la = \phi+\xi+\mu
\end{equation}
and $\la_1>\cdots > \la_m > \la_{m+1}\geq \cdots \geq \la_{\ell}$.

\begin{defn}
Say that $\ci\in [1,\ell]$ is a {\em critical index} if $\beta_{\ci+1}>
\beta_\ci+1$, or if $\la_\ci>\la_{\ci+1}+1$ (respectively, $\la_\ci>\la_{\ci+1}$)
and $\ci<m$ (respectively, $\ci>m$). Define two
sequences $\f=\f(w)$ and $\g=\g(w)$ of length $\ell$ as follows.  For
$1\leq j \leq \ell$, set
\[
\f_j:=k+\max(i\ |\ \gamma_{k+i}\geq j)
\]
and let $$\g_j:=\f_\ci+\beta_\ci-\xi_\ci-k,$$ where $\ci$ is the least
critical index such that $\ci\geq j$. We call $\f$ the {\em right flag}
of $w$, and $\g$ the {\em left flag} of $w$.
\end{defn}

If $m\geq 1$, then $m$ is a critical
index, since $u_m$ is the largest negative entry of $w$.
We will show that for any amenable element $w$, $\f$ is a weakly
decreasing sequence consisting of right descents of $w$, and $\g$ is a
weakly increasing sequence whose absolute values consist of left
descents of $w$.

\begin{lemma}
\label{philem}
{\em (a)} If $\beta_{s+1} > \beta_s+1$, then $|\beta_s|$ is a left
descent of $w$.

\medskip
\noin
{\em (b)} If $s\leq m$, then $\phi_s=k$, while if $s>m$ and 
$\beta_{s+1} = \beta_s+1$, then $\phi_s=\phi_{s+1}$.
\end{lemma}
\begin{proof}
Let $i:=|\beta_s|$, and suppose
that $1\leq s\leq m$. If $i=0$ then $u_s=-1$, so clearly $i$ is a left descent of $w$.
If $i\geq 1$ and $\beta_{s+1} > \beta_s+1$, then $i$ is
a left descent of $w$, since $w^{-1}(i)>0$ and $w^{-1}(i+1)<0$.  As
$w_j>0$ for all $j\in [1,k]$, we have $\psi_j\geq m$ for all $j\in
[1,k]$, and hence $\phi_s=k$.

Next suppose that $s>m$. If $\beta_{s+1} > \beta_s+1=i+1$, then we have
$w^{-1}(i+1)<0$ or $w_j=i+1$ for some $j\in [1,k]$. In either case, it
is clear that $i$ is a left descent of $w$. Finally, assume that
$\beta_{s+1} = \beta_s+1$. If $\phi_s>\phi_{s+1}$, there must exist
$j\in [1,k]$ such that $\gamma_j=s$, that is, $\#\{r>k\ |\ w_r<w_j\} =
s$. We deduce that 
\[
\{w_r\ |\ r>k \ \,\text{and}\, \ w_r<w_j\} =
\{u_1,\ldots, u_s\}, 
\]
which is a contradiction, since $u_s<w_j
\Rightarrow u_{s+1}=u_s+1<w_j$, for any $j\in [1,k]$. This 
completes the proof of (a) and (b).
\end{proof}

\begin{prop}
\label{amenableC}
Suppose that $\wh{w}\in W_n$ is leading with first right descent at
$k\geq 0$, let $\wh{\la}:=\la(\wh{w})$, and $\wh{\xi}:=\xi(\wh{w})$.
Let $w=\ome\wh{w}$ be a modification of $\wh{w}$, and set
$\gamma:=\gamma(w)$, $\la:=\la(w)$, $\beta:=\beta(w)$, and
$\xi:=\xi(w)$. Then the sequence $\beta+\wh{\la}-\la$ is weakly
increasing, and
\[
\CS_w = R^{D(w)} \, {}^{k^{n-k}+\wh{\xi}}c^{\beta(w)+\wh{\xi}-\xi}_{\la(w)} =
R^{D(w)} \, {}^{k^{n-k}+\wh{\xi}}c^{\beta+\wh{\la}-\la}_{\la}.
\]
Moreover, if $\ci\in [1,\ell]$ is a critical index of $w$, then
$k+\wh{\xi}_\ci$ is a right descent of $w$, the absolute value of
$\beta_\ci+\wh{\xi}_\ci-\xi_\ci$ is a left descent of $w$, and
$\wh{\xi}_\ci=\max(i\ |\ \gamma_{k+i}\geq \ci)$.
\end{prop}
\begin{proof}
Suppose that the truncated A-code of $\wh{w}$ is 
\[
{}^k\wh{\gamma}=(p_1^{n_1},\ldots, p_t^{n_t})
\]
for some parts $p_1>p_2>\cdots > p_t>0$,
and we let $d_j:=n_1+\cdots+n_j$ for $j\in [1,t]$. Then we have
\[
\wh{\xi}=(d_t^{p_t},d_{t-1}^{p_{t-1}-p_t},\ldots,d_1^{p_1-p_2})
\]
and it follows that 
\[
\wh{w}_{k+1}=u_{p_1+1},\ \wh{w}_{k+d_1+1}=u_{p_2+1},\ \ldots,
\ \wh{w}_{k+d_{t-1}+1}=u_{p_t+1}
\]
and $\wh{w}_j<\wh{w}_{j+1}$ for all $j\notin \{k,
k+d_1,\ldots,k+d_t\}$.  Hence the set of components of the vector
$k^{n-k}+\wh{\xi}$ coincides with the set of all right descents of
$\wh{w}$.

If $\ci\in [1,\ell]$ is a critical index, we have shown that $f_\ci$
is a right descent of $\wh{w}$. We claim that $i:=|g_\ci|=|\beta_\ci|$ is a
left descent of $\wh{w}$. By Lemma \ref{philem}(a), we may assume that
$\beta_{\ci+1}=\beta_\ci+1$, which implies that $\ci\neq m$.

Suppose that $\ci < m$. Then we have $\wh{\la}_\ci>\wh{\la}_{\ci+1}+1$ and 
$\wh{\mu}_\ci=\wh{\mu}_{\ci+1}+1$, so (\ref{pxmN}) gives
$\wh{\xi}_\ci>\wh{\xi}_{\ci+1}$. We therefore have $\ci=p_j$ for some $\ci\in
   [1,t]$, and hence $i=\wh{\mu}_{p_j}-1 =
   \wh{\mu}_{p_j+1}=-u_{p_j+1}=-\wh{w}_{k+d_{j-1}+1}$. Since we have
\[
\wh{w}_{k+1} > \wh{w}_{k+d_1+1} >\cdots > \wh{w}_{k+d_{j-1}+1} =-i,
\]
and the sequence $(\wh{w}_{k+1},\ldots,\wh{w}_n)$ is $132$-avoiding,
we conclude that $\wh{w}^{-1}(-i) = k+d_{j-1}+1 < \wh{w}^{-1}(-i-1)$, 
as desired.

Suppose next that $\ci>m$. Then we have $\wh{\la}_\ci>\wh{\la}_{\ci+1}$, so
Lemma \ref{philem}(b) and equation (\ref{pxmN}) imply that
$\wh{\xi}_\ci>\wh{\xi}_{\ci+1}$. We deduce that $\ci=p_j$ for some $j$,
hence $i+1=u_{p_j+1}$ and the result follows.

According to Corollary  \ref{kleadingN}, we have
\begin{equation}
\label{propCcor}
\CS_{\wh{w}} = R^{D(\wh{w})} \, {}^{k^{n-k}+\wh{\xi}}c^{\beta(\wh{w})}_{\la(\wh{w})}, 
\end{equation}
so the proposition holds for leading elements. Suppose next that
$w:=\ome\wh{w}$ is a modification of $\wh{w}$. Then repeated
application of (\ref{ddCeqinitN}), Lemma \ref{ddlemN}(a), and the left
Leibnitz rule (\ref{LeibRN}) in equation (\ref{propCcor}) give
\[
\CS_w = R^{D(w)} \, {}^{k^{n-k}+\wh{\xi}}c^{\beta(w)+\wh{\xi}-\xi}_{\la(w)}.
\]
It remains to check the last assertion, about the left and right 
descents of $w$.

Let $R_1\cdots R_{n-1}$ be the reduced decomposition for $\ome$ from
Definition \ref{modifdefn}. We will study the left action of the
successive simple transpositions in $R_1\cdots R_{n-1}$ on $\wh{w}$.
Observe that $\sigma^-$ and $\sigma^+$ are disjoint and $\sigma^-\sigma^+ =
\sigma^+ \sigma^-$. Moreover, the actions of the $\wh{w}$-positive and
$\wh{w}$-negative simple transpositions on $\wh{w}$ are similar, and
we can consider them separately. Let $A:=\{k+1,k+d_1+1,\ldots,
k+d_t+1\}$.

We begin with the $\wh{w}$-positive simple transpositions.  The action
of these on $\wh{w}$ is by a finite sequence of {\em moves} $v\mapsto
v'$, where $v'=s_i\cdots s_{j-1}v$ for some $i,j$ with $1\leq i < j$,
$\ell(v')=\ell(v)-j+i$, and $s_i\cdots s_{j-1}$ is a subword of some
$R_p$ with $j-i$ maximal. We call such a move an $[i,j]$-move at
position $r$ if $v_r=j$ and $v'_r=i$, so that $v'$ is obtained from
$v$ by cyclically permuting the values $i,i+1,\ldots,j$. Observe that
we must have $r\in A$, and subsequent $[i',j']$-moves for
$[i',j']\subset [i,j]$ are at positions $r'\in A$ with $r< r'$.  This
follows from the fact that the sequence
$(\wh{w}_{k+1},\ldots,\wh{w}_n)$ is $132$-avoiding, and by induction
on the number of moves.

Let $\al$ denote the truncated A-code of $v$, and $\xi_v$,
$g_v:=\beta+\wh{\xi}-\xi_v$ the associated statistics, with $\al',
\xi':=\xi_{v'}$, $g'=g_{v'}:=\beta+\wh{\xi}-\xi'$ the corresponding
ones for $v'$.  If $\al_r=e$, then we have $\al'_r=d$ for $d=e+i-j$,
and $\al'_s=\al_s$ for all $s\neq r$. If $\xi_v=(\xi_1,\xi_2,\ldots)$
and $g_v=(g_1,g_2,\ldots)$, then $g_{d+1}=i, \ldots, g_e = j-1$, while
\[
\xi'=(\xi_1,\ldots,\xi_d,\xi_{d+1}-1,\ldots,\xi_e-1,\xi_{e+1},\ldots)
\]
and
\begin{align*}
g' &= (g_1,\ldots,g_d,g_{d+1}+1,\ldots,g_e+1,g_{e+1},\ldots) \\
& = (g_1,\ldots,g_d,i+1,\ldots,j,g_{e+1},\ldots).
\end{align*}

Lemma \ref{philem} implies that the critical indices of $v$ and $v'$
can only differ in positions $d$ and $e$. Since the simple
transpositions $s_i,\ldots, s_{j-1}$ are all $\wh{w}$-positive, we
have $\beta_s+1=\beta_{s+1}$ for all $s\in [d+1,e]$. If
$\beta_d+1<\beta_{d+1}$, then $|g_d| = |\beta_d|$ is a left descent of
both $\wh{w}$ and $v'$, by Lemma \ref{philem}. We may therefore assume
that $\beta_s+1=\beta_{s+1}$ for all $s\in [d,e]$, and only need to
study the $s\in [d,e]$ where $\xi'_s>\xi'_{s+1}$. If $s\in [d+1,e]$,
since the values $i,\ldots,j$ of $v$ are cyclically permuted in $v'$,
it follows by induction on the number of moves that $g'_s$ is a left
descent of $v'$. Notice that we must have $i\geq 2$ in this
situation. It remains to prove that $g'_d=g_d=i-1$ is a left descent
of $v'$. But since $\wh{w}^{-1}(i-1)>k$ and the sequence
$(\wh{w}_{k+1},\ldots,\wh{w}_n)$ is $132$-avoiding, we deduce that
$(v')^{-1}(i)=r<(v')^{-1}(i-1)$, as desired.

The above procedure shows that for any critical index $h$ of $v'$,
we must have 
\[
\wh{\xi}_h=\max(i\ |\ \wh{\gamma}_{k+i}\geq h) =
\max(i\ |\ \al'_i \geq h), 
\]
while $k+\max(i\ |\ \al'_i \geq h)$ is a right descent of $v'$, by
Lemma \ref{Mcdlemma}. Finally, the action of the $\wh{w}$-negative
simple transpositions on $\wh{w}$ is studied in the same way.
\end{proof}

\begin{thm}
\label{TamvexC}
For any amenable element $w\in W_\infty$, we have 
\begin{equation}
\label{CSeq}
\CS_w = R^{D(w)} \, {}^{\f(w)}c_{\la(w)}^{\g(w)}
\end{equation}
in $\Gamma[X,Y]$.
\end{thm}
\begin{proof}
We may assume that we are in the situation of Proposition
\ref{amenableC}, so that $w=\ome\wh{w}$, with $\wh{\la}=\la(\wh{w})$
and $\la=\la(w)$.  Suppose that $j\in [1,\ell]$ and let $\ci$ be the
least critical index of $w$ such that $\ci\geq j$. Then we have
$\la_j=\la_{j+1}=\cdots = \la_\ci$, if $\ci>m$, and $\la_j=\la_{j+1}+1
= \cdots = \la_\ci+(\ci-j)$, if $\ci\leq m$. Moreover, in either case,
we have $\xi_j=\cdots = \xi_\ci$, and the values
$\beta_j,\ldots,\beta_\ci$ are consecutive integers. As the sequence
$g:=\beta+\wh{\xi}-\xi$ is weakly increasing, we deduce that for any
$r\in [j,\ci-1]$, either (i) $\wh{\xi}_r=\wh{\xi}_{r+1}$ and
$g_r=g_{r+1}-1$, or (ii) $\wh{\xi}_r=\wh{\xi}_{r+1}+1$ and
$g_r=g_{r+1}$. Theorem \ref{TamvexC} follows from this and induction
on $\ci-j$, by employing Lemmas \ref{commuteA}(b) and \ref{commuteC}(b)
in Proposition \ref{amenableC}. The required conditions on $D(w)$ in
these two lemmas and the corresponding relations (\ref{fundrelsC}) are
both easily checked.
\end{proof}

\begin{remark}
The equalities such as (\ref{CSeq}) in this section occur in
$\Gamma[X,Y]$, which is a ring with relations coming from
$\Gamma$. Therefore they are not equalities of polynomials in
independent variables, in contrast to the situation in type A. The
same remark applies to the corresponding equalities in Section
\ref{dmtD}.
\end{remark}

Anderson and Fulton \cite{AF2} have introduced a family of signed
permutations, each determined by an algorithm starting from an
equivalence class of `triples'. These were named `theta-vexillary' and
studied further by Lambert \cite{Lam}. It seems plausible that the
theta-vexillary signed permutations coincide with our amenable
elements in types B and C, but we do not examine this question here.

\subsection{Flagged theta polynomials}
\label{ftps}

In this section, we define a family of polynomials $\Theta_w$ indexed
by amenable elements $w\in W_\infty$ that generalize Wilson's double
theta polynomials \cite{W, TW}. For each $k\geq 0$, let
${}^k\frakc:=({}^k\frakc_p)_{p\in \Z}$ be a family of variables, such
that ${}^k\frakc_0=1$ and ${}^k\frakc_p=0$ for $p<0$, and let
$t:=(t_1,t_2,\ldots)$. The polynomial $\Theta_w$ represents an
equivariant Schubert class in the $T$-equivariant cohomology ring of
the symplectic partial flag variety associated to the right flag $\f(w)$,
which is defined in \cite[Sec.\ 4.1]{T1}.  The $t$ variables come from
the characters of the maximal torus $T$, as explained in \cite{TW}.

For any integers $p$ and $r$, define
\[
{}^k\frakc_p^r:=\sum_{j=0}^p {}^k\frakc_{p-j}h_j^r(-t).
\]
Given integer sequences $\kk$, $\al$, and $\rho$, let ${}^\kk
\frakc^{\rho}_\al:= {}^{\kk_1}\frakc_{\al_1}^{\rho_1}\,
      {}^{\kk_2}\frakc_{\al_2}^{\rho_2} \cdots$, and let any raising
      operator $R$ act in the usual way, by $R\,{}^\kk \frakc^{\rho}_\al:=
      {}^\kk \frakc^{\rho}_{R\al}$.

If $w\in W_n$ is amenable with left flag $\f(w)$ and right flag $\g(w)$, then the 
{\em flagged double theta polynomial} $\Ti_w(\frakc\, |\, t)$ is defined by 
\begin{equation}
\label{fgtieq}
\Ti_w(\frakc\, |\, t):= R^{D(w)}\, {}^{\f(w)}\frakc_{\la(w)}^{\g(w)}.
\end{equation}
The {\em flagged single theta polynomial} is given by
$\Ti_w(\frakc):=\Ti_w(\frakc\, |\, 0)$. If $w$ is a leading element,
then (\ref{fgtieq}) can be written in the `factorial' form
\[
\Ti_w(\frakc\, |\, t)= R^{D(w)}\, {}^{\f(w)}\frakc_{\la(w)}^{\be(w)}.
\]
When $w$ is a $k$-Grassmannian element, the above formulas specialize
to the double theta polynomial $\Ti_\la(\frakc\, |\, t)$ found in
\cite{TW}; here $\la$ is the $k$-strict partition corresponding to
$w$. Moreover, the single theta polynomial $\Ti_\la(\frakc)$ agrees
with that of \cite{BKT2}.

\subsection{Symplectic degeneracy loci}
\label{sdl}

Let $E\to \X$ be a vector bundle of rank $2n$ on a smooth complex algebraic
variety $\X$. Assume that $E$ is a {\em symplectic} bundle, so that $E$
is equipped with an everywhere nondegenerate skew-symmetric form
$E\otimes E\to \C$. Let $w\in W_n$ be amenable 
of shape $\la$, and let $\f$ and $\g$ be the left and right flags of 
$w$, respectively. Consider two complete flags of subbundles of $E$
\[
0 \subset E_1\subset \cdots \subset E_{2n}=E \ \ \, \mathrm{and} \, \ \ 
0 \subset F_1\subset \cdots \subset F_{2n}=E
\]
with $\rank E_r=\rank F_r=r$ for each $r$, while $E_{n+s}=E_{n-s}^{\perp}$ 
and $F_{n+s}=F_{n-s}^{\perp}$ for $0\leq s < n$. 

There is a group monomorphism $\zeta:W_n\hra S_{2n}$ with image
\[
\zeta(W_n)=\{\,\om\in S_{2n} \ |\ \om_i+\om_{2n+1-i} = 2n+1,
 \ \ \text{for all}  \ i\,\}.
\]
The map $\zeta$ is determined by setting,
for each $w=(w_1,\ldots,w_n)\in W_n$ and $1\leq i \leq n$,
\[
\zeta(w)_i :=\begin{cases}
             n+1-w_{n+1-i} & \text{if $w_{n+1-i}$ is unbarred}, \\
             n+\ov{w}_{n+1-i} & \text{otherwise}.
             \end{cases} 
\]
Define the {\em degeneracy locus}
$\X_w\subset \X$ as the locus of $x \in \X$ such that
\[
\dim(E_r(x)\cap F_s(x))\geq \#\,\{\,i \leq r
\ |\ \zeta(w)_i > 2n-s\,\} \ \, \mathrm{for} \ 1\leq r \leq n, \, 1\leq s \leq 2n.
\]
We assume that $\X_w$ has pure codimension $\ell(w)$ in $\X$, and give
a formula for the class $[\X_w]$ in $\HH^{2\ell(w)}(\X)$.

\begin{thm}
\label{dbleCloci}
For any amenable element $w\in W_n$, we have
\begin{equation}
\label{CChern}
[\X_w] = \Theta_w(E-E_{n-\f}-F_{n+\g}) =
R^{D(w)}\, c_{\la}(E-E_{n-\f}-F_{n+\g})
\end{equation}
in the cohomology ring $\HH^*(\X)$.
\end{thm}

As in \cite[Eqn.\ (7)]{TW}, the Chern polynomial in (\ref{CChern}) is
interpreted as the image of the polynomial $R^{D(w)} \mathfrak{c}_\la$
under the $\Z$-linear map which sends the noncommutative monomial
$\frakc_\al=\frakc_{\al_1}\frakc_{\al_2}\cdots$ to $\prod_j
c_{\al_j}(E-E_{n-\f_j}-F_{n+\g_j})$, for every integer sequence $\al$.
Theorem \ref{dbleCloci} is proved by applying the type C
geometrization map of \cite[Sec.\ 10]{IMN} to both sides of
(\ref{CSeq}), following \cite[Sec.\ 4.2]{T1}.

\section{Amenable elements: Type D theory}
\label{Dtheory}

\subsection{Definitions and main theorem}
\label{dmtD}

Let $w$ be an element in $\wt{W}_\infty$ with A-code $\gamma$ and
shape $\la=\mu+\nu$, with $\ell=\ell(\la)$ and $m=\ell(\mu)$.  Choose
$k\geq 1$, and assume that $w$ is increasing up to $k$.  If $k=1$,
this condition is vacuous, while if $k>1$ it means that
$|w_1|<w_2<\cdots<w_k$. Eventually, $k$ will be set equal to the
primary index of $w$.

List the entries $w_{k+1},\ldots,w_n$ in increasing order:
\[
u_1<\cdots < u_{m'} < 0 < u_{m'+1} < \cdots < u_{n-k},
\]
where $m'\in\{m,m+1\}$. Define a sequence $\beta(w)$ by
\[
\beta(w):=(u_1+1,\ldots,u_{m'}+1,u_{m'+1},\ldots,u_{n-k}).
\]
and the {\em denominator set} $D(w)$ by
\[
D(w):=\{(i,j) \ |\ 1\leq i<j \leq n-k \ \, \text{and} \ \, 
u_i+ u_j < 0 \}.
\]
As in Section \ref{dmtC}, the notation suppresses the dependence of
$\beta(w)$ and $D(w)$ on $k$.

\begin{defn}
\label{Dtrunc}
Suppose that $w\in \wt{W}_n$ has code $\gamma=\gamma(w)$ and $k\geq
1$.  The {\em $k$-truncated A-code} ${}^k\gamma={}^k\gamma(w)$ is
defined by ${}^k\gamma(w):=(\gamma_{k+1}, \gamma_{k+2},\ldots,
\gamma_n)$.  If $k$ is the primary index of $w$, then we call
${}^k\gamma(w)$ the {\em truncated A-code} of $w$. Let $\xi=\xi(w)$
be the partition whose parts satisfy
$\xi_j:=\#\{i\ |\ \gamma_{k+i}\geq j\}$ for each $j\geq 1$.
\end{defn}

We let $v(w)$ be the unique $k$-Grassmannian element obtained by
reordering the entries $w_{k+1},\ldots,w_n$ to be increasing.  For
example, if $w:=(\ov{2},4,7,5,\ov{8},\ov{3},1,\ov{6})$ and $k:=3$,
then $v(w)=(\ov{2},4,7,\ov{8},\ov{6},\ov{3},1,5)$.  The map $w \mapsto
v(w)$ is a type-preserving bijection from the set of elements in
$\wt{W}_n$ increasing up to $k$ with a given $k$-truncated A-code $C$
onto the set of $k$-Grassmannian elements in $\wt{W}_n$, such that
$\beta(v(w))=\beta(w)$ and
\[
\ell(v(w)) = \ell(w)-\sum_{i=k+1}^n C_i = \ell(w)-\sum_{j=1}^{n-k} \gamma_{k+j}.
\]

\begin{lemma}
\label{Dcases}
Let $w$ and $\ov{w}$ be elements in $\wt{W}_n$ increasing up to $k\geq 1$
and with the same $k$-truncated A-code $C$, 
such that $\ell(w)=\ell(\ov{w})+1$. Suppose that 
$v(\ov{w})=s_iv(w)$ for some simple reflection $s_i$. 
Then $\ov{w}=s_iw$. 
\end{lemma}
\begin{proof} 
We have seven possible cases for $i$ and $v(w)$: (a) $i=\Box$ and
$v(w) = (\wh{1} \cdots \ov{2} \cdots)$; (b) $i=\Box$ and $v(w) =
(\cdots \ov{2} \cdots \ov{1} \cdots)$; (c) $i=\Box$ and $v(w) = (2
\cdots \ov{1} \cdots)$; (d) $i\geq 1$ and $v(w) = (\cdots i \cdots
\ov{i+1} \cdots)$; (e) $i\geq 1$ and $v(w) = (\cdots \ov{i+1} \cdots i
\cdots)$; (f) $i\geq 1$ and $v(w) = (\ov{i} \cdots \ov{i+1} \cdots)$;
(g) $i \geq 1$ and $v(w) = (\cdots i+1 \cdots i \cdots)$. In the first
five cases, it is clear that $\ov{w}=s_iw$. In case (f) (respectively (g)),
the $\ov{i}$ (respectively $i+1$) must be among the first $k$ entries of
$v(w)$, which coincide with the first $k$ entries of $w$, while
$\ov{i+1}$ (respectively $i$) lies among the last $n-k$ entries of
$w$. Hence we again deduce that $\ov{w}=s_iw$. 
\end{proof}

If $R:=\prod_{i<j} R_{ij}^{n_{ij}}$ is any
raising operator and $d \geq 0$, denote by $\supp_d(R)$ the set of all
indices $i$ and $j$ such that $n_{ij}>0$ and $j \leq d$.  

\begin{defn}
\label{Etadef} 
Let $w\in\wt{W}_n$ be of shape $\la=\mu+\nu$, with $\ell=\ell(\la)$
and $m=\ell(\mu)$. Let $\al=(\al_1,\ldots,\al_{\ell})$ be a
composition such that $\al_{m+1}=\la_{m+1}$, if $\type(w)>0$, and
$\upsilon=(\upsilon_1,\ldots,\upsilon_\ell)$ be an integer vector such
that $\upsilon_{m+1}\in \{0,1\}$. For any integer vector $\rho$,
define
\[
{}^\al\wh{c}_\rho^{\,\upsilon}:= {}^{\al_1}\wh{c}_{\rho_1}^{\,\upsilon_1}
{}^{\rho_2}\wh{c}_{\rho_2}^{\,\upsilon_2}\cdots
\]
where, for each $i\geq 1$, 
\begin{equation}
\label{f_requ}
{}^{\al_i}\wh{c}_{\rho_i}^{\,\upsilon_i}:= {}^{\al_i}c_{\rho_i}^{\upsilon_i} + 
\begin{cases}
(-1)^ie^{\al_i}_{\al_i}(X)e^{\rho_i-\al_i}_{\rho_i-\al_i}(-Y) & 
\text{if $\upsilon_i = \al_i - \rho_i < 0$}, \\
0 & \text{otherwise}.
\end{cases}
\end{equation}
Let $R$ be any raising
operator appearing in the expansion of the power series $R^{D(w)}$ and
set $\rho:=R\la$. If $\type(w)=0$, then define
\[
R \star {}^{\al}\wh{c}^{\,\upsilon}_{\la} = {}^{\al}\ov{c}^\upsilon_{\rho} :=
{}^{\al_1}\ov{c}_{\rho_1}^{\upsilon_1}\cdots {}^{\al_\ell}\ov{c}^{\upsilon_\ell}_{\rho_\ell}
\]
where, for each $i\geq 1$, 
\[
{}^{\al_i}\ov{c}_{\rho_i}^{\upsilon_i}:= 
\begin{cases}
{}^{\al_i}c_{\rho_i}^{\upsilon_i} & \text{if $i\in\supp_m(R)$}, \\
{}^{\al_i}\wh{c}_{\rho_i}^{\,\upsilon_i} & \text{otherwise}.
\end{cases}
\]
If $\type(w)>0$ and $R$ involves any factors $R_{ij}$ with $i=m+1$ or
$j=m+1$, then define
\[
R \star {}^{\al}\wh{c}^{\,\upsilon}_{\la} := {}^{\al_1}\ov{c}_{\rho_1}^{\upsilon_1} \cdots 
{}^{\al_m}\ov{c}_{\rho_m}^{\upsilon_m} \, {}^{\al_{m+1}}a^{\upsilon_{m+1}}_{\rho_{m+1}} \,{}^{\al_{m+2}} c_{\rho_{m+2}}^{\upsilon_{m+2}}
\cdots {}^{\al_{\ell}}c^{\upsilon_\ell}_{\rho_\ell}.
\]
If $R$ has no such factors, then define
\[
R \star {}^{\al}\wh{c}^{\,\upsilon}_{\la} := \begin{cases}
  {}^{\al_1}\ov{c}_{\rho_1}^{\upsilon_1} \cdots
  {}^{\al_m}\ov{c}_{\rho_m}^{\upsilon_m} \,
  {}^{\al_{m+1}}b^{\upsilon_{m+1}}_{\la_{m+1}} \,
  {}^{\al_{m+2}}c_{\rho_{m+2}}^{\upsilon_{m+2}} \cdots
  {}^{\al_\ell}c^{\upsilon_\ell}_{\rho_\ell} & \text{if $\,\type(w)
    = 1$}, \\ {}^{\al_1}\ov{c}_{\rho_1}^{\upsilon_1} \cdots
  {}^{\al_m}\ov{c}_{\rho_m}^{\upsilon_m} \,
  {}^{\al_{m+1}}\wt{b}^{\upsilon_{m+1}}_{\la_{m+1}} \,
  {}^{\al_{m+2}}c_{\rho_{m+2}}^{\upsilon_{m+2}} \cdots
  {}^{\al_\ell}c^{\upsilon_\ell}_{\rho_\ell} & \text{if $\,\type(w)
    = 2$}.
\end{cases}
\]
\end{defn}

\begin{prop}
\label{domnD}
Suppose that $w\in \wt{W}_n$ is an element with primary index $k$ such that
$(w_1,\ldots,w_k)=(\wh{1},2,\ldots,k)$, $w_{k+j}<0$ for $1\leq j\leq n-k$, and
the truncated A-code ${}^k\gamma(w)$ is a partition. Then we have
\begin{equation}
\label{domnegD}
\DS_w = 2^{k-n}\, R^{D(w)} \star {}^{\nu(w)}\wh{c}^{\,(1-n,2-n,\ldots, -k)}_{\la(w)}.
\end{equation}
\end{prop}
\begin{proof}
The proof of (\ref{domnegD}) is by
descending induction on $\ell(w)$. One knows from 
\cite[\S 4.4]{T5} that (\ref{domnegD}) is true for the longest element
$w_0^{(k,n)}:=(\wh{1},2,\ldots,k, -k-1,\ldots,-n)$, which has shape 
$(n+k-1,\ldots,2k)+\delta_{n-k-1}$ of type 0.

Suppose that $w\neq w_0^{(k,n)}$ satisfies the conditions of the
proposition, and the shape of $w$ equals $(n+k-1,\ldots,2k)+\rho(w)$
(of type 0). Then $\rho\subset \delta_{n-k-1}$ and $\rho\neq
\delta_{n-k-1}$. Let $r\geq 1$ be the largest integer such that
$\rho_i=n-k-i$ for $i\in [1,r]$, and let $j:=\rho_{r+1}+1\leq
n-k-r-1$.  Then $ws_{k+j}$ is of length $\ell(w)+1$ and satisfies the
same conditions, $\nu(ws_{k+j})=\nu(w)+\epsilon_{r+1}$, and
$\la(ws_{k+j})=\la(w)+\epsilon_{r+1}$.  Using Proposition
\ref{prop1new}(a), for any integer sequence
$\al=(\al_1,\ldots,\al_n)$, we have
\[
\partial^x_{k+j}\left({}^{\nu(ws_{k+j})}\wh{c}^{\,(1-n,\ldots, -k)}_{\al}\right) = 
{}^{\nu(w)}\wh{c}^{\,(1-n,\ldots, -k)}_{\al-\epsilon_{r+1}}.
\]
By induction, we deduce that
\begin{align*}
\DS_w &= \partial^x_{k+j}\left(\DS_{ws_{k+j}}\right) = 2^{k-n}\,
\partial^x_{k+j}\left(R^{\infty} \star
{}^{\nu(ws_{k+j})}\wh{c}^{\,(1-n,\ldots, -k)}_{\la(ws_{k+j})}\right) \\
&= 2^{k-n}\,R^{\infty} \star {}^{\nu(w)}\wh{c}^{\,(1-n,\ldots, -k)}_{\la(w)},
\end{align*}
proving the proposition.
\end{proof}

Let $w\in \wt{W}_\infty$ have shape $\la = \mu+\nu$ and $\kk$ be any
integer sequence. We say that $\kk$ is {\em compatible with $w$} if
$\kk_p=\nu_p$ for $p\in [1,m]$, and $\kk_{m+1}=\nu_{m+1}$ whenever
$\type(w)>0$.

\begin{prop}
\label{gampropDN}
Fix an integer $k\geq 1$. Suppose that $w$ and $\ov{w}$ are elements
in $\wt{W}_n$ increasing up to $k$ with the same $k$-truncated A-code $C$, such
that $\ell(w)=\ell(\ov{w})+1$ and $s_iv(w)=v(\ov{w})$ for some
simple reflection $s_i$. Assume that we have
\[
\DS_w = 2^{-\ell(\mu(w))}\,R^{D(w)}\star {}^{\kk}\wh{c}^{\,\beta(w)}_{\la(w)} 
\]
in $\Gamma'[X,Y]$, for some integer sequence $\kk$ compatible with
$w$.  Moreover, if $i\in\{\Box,1\}$ and $|w_1|>2$, assume that
$\kk_m=\kk_{m+1}$.  Then $\kk$ is compatible with $\ov{w}$, and we
have
\[
\DS_{\ov{w}} = 2^{-\ell(\mu(\ov{w}))}\,R^{D(\ov{w})}\star {}^{\kk}\wh{c}^{\,\beta(\ov{w})}_{\la(\ov{w})}
\]
in $\Gamma'[X,Y]$.
\end{prop} 
\begin{proof}
Set 
\begin{equation}
\label{Feq}
F_w:=2^{-\ell(\mu(w))}\,R^{D(w)} \star
{}^{\kk}\wh{c}^{\,\beta(w)}_{\la(w)}, 
\end{equation}
so we know that $\DS_w=F_w$. As equation (\ref{Dddeq}) gives
$\partial^y_i\DS_w = \DS_{\ov{w}}$, it will suffice to show that
$\partial^y_i F_w = F_{\ov{w}}$.  The proof of this will follow the
argument of \cite[Prop.\ 5]{T4}, and correct it by including the case
(h) below, which was missing there.

Let $\mu:=\mu(w)$, $\nu:=\nu(w)$, $\la:=\la(w)=\mu+\nu$,
$\ov{\mu}:=\mu(\ov{w})$, $\ov{\nu}:=\nu(\ov{w})$,
$\ov{\la}:=\la(\ov{w})= \ov{\mu}+\ov{\nu}$, $\beta=\beta(w)$, and
$\ov{\beta}=\beta(\ov{w})$.  Using Lemma \ref{WDlem}, we distinguish
eight possible cases for $w$. In each case, we have
$\ov{\la}\subset\la$, so that $\ov{\la}_p=\la_p-1$ for some $p\geq 1$
and $\ov{\la}_j=\la_j$ for all $j\neq p$. Moreover, we must have
$\type(w)+\type(\ov{w}) \neq 3$.

First, we consider the four cases with $i\geq 1$:

\medskip
\noin (a) $v(w) = (\cdots i+1 \cdots i \cdots)$.  In this case
$D(w)=D(\ov{w})$ while clearly $\mu=\ov{\mu}$.  We deduce that
$\ov{\nu}_p=\nu_p-1$, while $\ov{\nu}_j=\nu_j$ for all $j\neq p$.  We
must show that $\be_p=i$, and hence $\ov{\be}_p=i+1$, while
$\be_j=\ov{\be}_j$ for all $j\neq p$.  

Note that if $w_r=i+1$, then $r\in [1,k]$.  Since $w_j\geq w_r$ for
all $j\in [r,k]$, and the sequence $\beta(w)$ is strictly increasing,
we deduce that $\beta_g=i$ exactly when $g=\gamma_r(w)$.  We have
$\gamma_r(\ov{w})=\gamma_r(w)-1=g-1$, while
$\gamma_j(w)=\gamma_j(\ov{w})$ for $j\neq r$. It follows that
$\ov{\nu}_g=\nu_g-1$, while $\ov{\nu}_j=\nu_j$ for all $j\neq g$. In
other words, $g=p$, as desired. 

Finally, observe that (i) if $i=1$, then $p=m+1$, $\type(w)>0$, and
$\type(\ov{w})=0$; (ii) if $i\geq 2$, then $p>m$ and
$\type(w)=\type(\ov{w})$, while $p>m+1$ if $\type(w)>0$. It follows
that $\kk$ is compatible with $\ov{w}$.

\medskip
\noin 
(b) $v(w) = (\cdots i \cdots \ov{i+1} \cdots)$.  In this case $w^{-1}(i)\in [1,k]$,
$D(w)=D(\ov{w})$, $\nu=\ov{\nu}$, $\be_p = -i$, $\ov{\be}_p=-i+1$, and
$\be_j=\ov{\be}_j$ for all $j\neq p$.

\medskip
\noin (c) $v(w) = (\ov{i} \cdots \ov{i+1} \cdots)$.  In this case
$w_1=\ov{i}$, $\type(w)=2$ if $i \geq 2$, $D(w)=D(\ov{w})$,
$\nu=\ov{\nu}$, $\be_p =-i$, $\ov{\be}_p = -i+1$, and
$\be_j=\ov{\be}_j$ for all $j\neq p$.

\medskip
\noin (d) $v(w)= (\cdots \ov{i+1} \cdots i \cdots)$. We distinguish
two subcases here: 

\medskip
\noin
Case (d1): $w_1\neq\ov{i+1}$. Then $\nu=\ov{\nu}$, $\be_p=-i$,
$\ov{\be}_p=-i+1=\beta_p+1$, and $D(w)=D(\ov{w})\cup\{(p,q)\}$, where
$v(w)_{k+p}=\ov{i+1}$ and $v(w)_{k+q}=i$. It follows that $\be_q=i$
and $\ov{\be}_q=i+1 = \be_q+1$, while $\be_j=\ov{\be}_j$ for all $j
\notin \{p,q\}$. 
 
\medskip
\noin Case (d2): $w_1=\ov{i+1}$ and we have $w^{-1}(i)>k$. In this
case $\type(w)=2$, $D(w)=D(\ov{w})$, while clearly $\mu=\ov{\mu}$.
We deduce that $\ov{\nu}_p=\nu_p-1$, while $\ov{\nu}_j=\nu_j$ for all
$j\neq p$.  We must show that $\be_p =i$, and hence $\ov{\be}_p=i+1$,
while $\be_j=\ov{\be}_j$ for all $j\neq p$.  Indeed, observe that
$\nu(w)=\nu(\iota(w))$, and $\iota(w)_1=i+1$, the argument used in
case (a) applies; this is true even when $i=1$.

\medskip

Next, we consider the four cases where $i=\Box$. 

\medskip
\noin (e) $v(w) = (\wh{1} \cdots \ov{2} \cdots)$.  In this case
$\type(w)=0$, $D(w)=D(\ov{w})$, $\nu=\ov{\nu}$, $\be_p =-1$, and
$\ov{\be}_p=1$.  We also have $\be_j=\ov{\be}_j$ for all $j\neq p$.

\medskip
\noin (f) $v(w) = (\ov{2} \cdots \ov{1} \cdots)$.  In this case
$\type(w)=2$, $D(w)=D(\ov{w})$, and $\mu=\ov{\mu}$. We deduce that
$\ov{\nu}_p=\nu_p-1$, while $\ov{\nu}_j=\nu_j$ for all $j\neq p$.  We
must show that $\be_p=0$, and hence $\ov{\be}_p =2$, while
$\ov{\be}_j=\be_j$ for all $j\neq p$. Indeed, we have $\iota(v(w))= (2
\cdots 1 \cdots)$, so the analysis in case (a) applies.

\medskip
\noin (g) $v(w) = (\cdots \ov{2} \ov{1} \cdots)$, with $|w_1|>2$.  In
this case $\type(w)$ and $\type(\ov{w})$ are both positive, $\nu=\ov{\nu}$, and 
$D(w)=D(\ov{w})\cup\{(p,p+1)\}$, where $v(w)_{k+p}=\ov{2}$
and $v(w)_{k+p+1}=\ov{1}$, and thus $p=\ell(\mu)=m$. It follows that 
$\be_p=-1$, $\be_{p+1}=0$, $\ov{\be}_p=1$,
$\ov{\be}_{p+1}=2$, and $\be_j=\ov{\be}_j$ for all $j \notin
\{p,p+1\}$. We also have $\la_m=k+\xi_m+1$, $\la_{m+1}=k+\xi_{m+1}=\la_m-1$.

\medskip
\noin
(h) $v(w) = (2 \cdots \ov{1} \cdots)$. In this case $\type(w)=1$,
$D(w)=D(\ov{w})$, and $\mu=\ov{\mu}$. We deduce that
$\ov{\nu}_p=\nu_p-1$, while $\ov{\nu}_j=\nu_j$ for all $j\neq p$. We must 
show that $\be_p=0$, so that $\ov{\be}_p=2$, and $\be_j=\ov{\be}_j$ for all $j\neq p$.
This is proved as in case (f).

\medskip

To simplify the notation, set
$c^{\upsilon}_{\al}:={}^{\kk}c^{\upsilon}_\al$ and
$\wh{c}^{\,\upsilon}_{\al}:={}^{\kk}\wh{c}^{\,\upsilon}_\al$, for any
integer sequences $\al$ and $\upsilon$. We now distinguish the
following cases.

\medskip
\noin
{\bf Case 1.} $\type(w)=\type(\ov{w})=0$.

\medskip
Note that we have $|w_1|=|\ov{w}_1|=1$, and hence $i\geq 2$ and
$\ell(\mu) = \ell(\ov{\mu})$. We must be in one among cases (a), (b),
or (d1) above.  In cases (a) or (b), it follows from
Propositions \ref{prop1new} and \ref{prop1hnew} and the left Leibnitz rule that
for any integer sequence $\al=(\al_1,\ldots,\al_\ell)$, we have
\begin{align*}
\partial_i^y\, \wh{c}^{\,\be(w)}_\al &=
\wh{c}^{\,(\be_1,\ldots,\be_{p-1})}_{(\al_1,\ldots,\al_{p-1})}
\left(\partial_i(\wh{c}^{\,\be_p}_{\al_p})\,\wh{c}^{\,(\be_{p+1},\ldots,\be_\ell)}_{(\al_{p+1},\ldots,\al_\ell)}
+s^y_i(\wh{c}^{\,\be_p}_{\al_p})\,\partial_i
(\wh{c}^{\,(\be_{p+1},\ldots,\be_\ell)}_{(\al_{p+1},\ldots,\al_\ell)})\right)
\\ &=\wh{c}^{\,(\be_1,\ldots,\be_{p-1})}_{(\al_1,\ldots,\al_{p-1})}
\left(\wh{c}^{\,\be_p+1}_{\al_p-1}\,\wh{c}^{\,(\be_{p+1},\ldots,\be_\ell)}_{(\al_{p+1},\ldots,\al_\ell)}
+0\right)
=\wh{c}^{\,(\be_1,\ldots,\be_p+1,\ldots,\be_\ell)}_{(\al_1,\ldots,\al_p-1,\ldots,\al_\ell)}
= \wh{c}^{\,\be(\ov{w})}_{\al-\epsilon_p}.
\end{align*}
Since $\la-\epsilon_p=\ov{\la}$, it follows that if $R$ is any raising
operator, then
\[
\partial^y_i \left(R \star\wh{c}^{\,\be(w)}_\la\right)
= \partial^y_i \left(\ov{c}^{\be(w)}_{R\la}\right) = 
\ov{c}^{\be(\ov{w})}_{R\la-\epsilon_p} = R \star\wh{c}^{\,\be(\ov{w})}_{\ov{\la}}.
\]
As $R^{D(w)}= R^{D(\ov{w})}$, we deduce that 
\[
\partial^y_i(F_w)= 2^{-\ell(\mu)}\partial^y_i \left(R^{D(w)} \star
\wh{c}^{\,\be(w)}_\la\right) = 2^{-\ell(\ov{\mu})} R^{D(\ov{w})} \star
\wh{c}^{\,\be(\ov{w})}_{\ov{\la}} = F_{\ov{w}}.
\]
In case (d1), for any integer sequence $\al=(\al_1,\ldots,\al_\ell)$,
we compute that
\begin{align*}
\partial^y_i\, \wh{c}^{\,\be(w)}_\al &= \partial^y_i\,
\wh{c}^{\,(\be_1,\ldots,-i,\ldots,i,\ldots,\be_{\ell})}_{(\al_1,\ldots,\al_p,\ldots,\al_q,
  \ldots,\al_\ell)} \\ &=
\wh{c}^{\,(\be_1,\ldots,-i+1,\ldots,i+1,\ldots,\be_{\ell})}_{(\al_1,\ldots,\al_p-1,\ldots,\al_q,\ldots,\al_\ell)}
+
\wh{c}^{\,(\be_1,\ldots,-i+1,\ldots,i+1,\ldots,\be_{\ell})}_{(\al_1,\ldots,\al_p,\ldots,\al_q-1,\ldots,\al_\ell)}
= \wh{c}^{\,\be(\ov{w})}_{\al-\epsilon_p} + \wh{c}^{\,\be(\ov{w})}_{\al-\epsilon_q}.
\end{align*}
This follows from the left Leibnitz rule, as in the proof of Proposition
\ref{prop1new}(b). Since $i\geq 2$, we must have $q>\ell(\mu)$. Hence
if $R$ is any raising operator, then $q\notin\supp_m(RR_{pq})$,
where $m=\ell(\mu)$. As $\la-\epsilon_p=\ov{\la}$, we deduce that
\[
\partial^y_i\left(R \star \wh{c}^{\,\be(w)}_\la\right) = 
\partial^y_i \left(\ov{c}^{\be(w)}_{R\la}\right) = 
\ov{c}^{\be(\ov{w})}_{R\la-\epsilon_p} + \ov{c}^{\be(\ov{w})}_{R\la-\epsilon_q}
=  R \star \wh{c}^{\,\be(\ov{w})}_{\ov{\la}} + RR_{pq} \star \wh{c}^{\,\be(\ov{w})}_{\ov{\la}}.
\]
Since $R^{D(w)}+R^{D(w)} R_{pq} = R^{D(\ov{w})}$, it follows that
$\partial^y_i F_w = F_{\ov{w}}$.

\medskip
\noin
{\bf Case 2.} $\type(w)=0$ and $\type(\ov{w})>0$.

\medskip
In this case, we have $|w_1|=1$ and $|\ov{w}_1|>1$, so $i\in \{\Box, 1\}$.  We
must be in one of cases (b), (c), or (e) above, hence
$D(w)=D(\ov{w})$. We also have $(p,p+1)\notin D(w)$, $\ell(\mu) = \ell(\ov{\mu})+1$, 
$\be_p=-1$, $\ov{\be}_p=0$ if $i=1$, and $\ov{\be}_p=1$ if $i=\Box$.

Observe that for any integer sequence
$\al=(\al_1,\ldots,\al_\ell)$, we have
\begin{align*}
\partial^y_i \wh{c}^{\,\be(w)}_\al &=
\wh{c}^{\,(\be_1,\ldots,\be_{p-1})}_{(\al_1,\ldots,\al_{p-1})}
\left(\partial^y_i(\wh{c}^{\,-1}_{\al_p})\,
\wh{c}^{\,(\be_{p+1},\ldots,\be_\ell)}_{(\al_{p+1},\ldots,\al_\ell)}
+s^y_i(\wh{c}^{\,-1}_{\al_p})\,\partial_i^y
(\wh{c}^{\,(\be_{p+1},\ldots,\be_\ell)}_{(\al_{p+1},\ldots,\al_\ell)})\right)
\\ &=
\wh{c}^{\,(\be_1,\ldots,\be_{p-1})}_{(\al_1,\ldots,\al_{p-1})}\,\partial^y_i(\wh{c}^{\,-1}_{\al_p})
\,\wh{c}^{\,(\be_{p+1},\ldots,\be_\ell)}_{(\al_{p+1},\ldots,\al_\ell)}.
\end{align*}
We now compute using Propositions \ref{prop1new} and \ref{prop1hnew}(a) that
\[
\partial_1^y\left({}^r\wh{c}_q^{\,-1}\right)= 
\begin{cases}
2\left({}^ra_{q-1}^0\right) & \text{if $q\neq r+1$} \\
2f_r & \text{if $q=r+1$}.
\end{cases}
\]
Propositions \ref{prop1new}(a) and \ref{prop1hnew} give
\[
\partial^y_\Box\left({}^r\wh{c}_q^{\,-1}\right)=  
\begin{cases}
2\left({}^ra_{q-1}^1\right) & \text{if $q\neq r+1$} \\
2\wt{f}^1_r & \text{if $q=r+1$}.
\end{cases}
\]
Note that the choice of $f_r$ in these equations is specified by 
formula (\ref{f_requ}). The rest is straightforward from the
definitions, arguing as in Case 1.

\medskip
\noin
{\bf Case 3.} $\type(w)>0$ and $\type(\ov{w})=0$. 

\medskip
We have $|w_1|>1$ and $|\ov{w}_1|=1$, so $i\in\{\Box,1\}$, and we are in one of
cases (a), (d2), (f), or (h) above, hence $D(w)=D(\ov{w})$.

We also have $\be_p\in \{0,1\}$, $\ov{\be}_p=2$, and
$\ell(\mu) = \ell(\ov{\mu})$. Recall that ${}^r\wh{c}_q^{\,s} = {}^rc_q^s$ whenever
$q\leq r$, ${}^rb_r^1={}^rc_r^1-{}^r\wt{b}_r$, ${}^r\wt{b}_r^1={}^rc_r^1-{}^rb_r$,
and $\dis {}^ra^s_q={}^rc^s_q-\frac{1}{2} {}^rc_q$.  We deduce the calculations
\begin{gather*}
\partial^y_\Box \left({}^rb_r\right) = \partial^y_\Box \left({}^r\wt{b}_r\right) = 
\partial^y_1 \left({}^rb^1_r\right) = \partial^y_1\left(
{}^r\wt{b}_r^1\right) = {}^rc_{r-1}^2 \\ 
\partial^y_\Box \left({}^ra^0_q\right) = \partial^y_1 \left({}^ra_q^1 \right)
= {}^rc_{q-1}^2.
\end{gather*}
As in the previous cases, it follows that 
$\partial^y_i F_w = F_{\ov{w}}$.

\medskip
\noin
{\bf Case 4.} $\type(w)=\type(\ov{w})>0$.

\medskip
We have $|w_1|>1$ and $|\ov{w}_1|>1$. If $i\geq 2$, then we must be in
one of cases (a), (b), (c), or (d1) above, and the result is proved by
arguing as in Case 1. It remains to study (i) case (d1) with
$v(w)=(\cdots \ov{2}1 \cdots)$ and $i=1$, or (ii) case (g) with $v(w)
= (\cdots \ov{2} \ov{1} \cdots)$ and $i=\Box$. In both of these
subcases, we have $p=m$, $\ell(\mu) =
\ell(\ov{\mu})+1$, $D(w)=D(\ov{w})\cup\{(m,m+1)\}$, $\be_m(w)=-1$, $\ov{\be}_{m+1}=2$, and
$\be_j=\ov{\be}_j$ for all $j \notin \{m,m+1\}$. In subcase (i), we
have $\be_{m+1}=1$ and $\ov{\be}_m=0$, while in subcase (ii), we have
$\be_{m+1}=0$ and $\ov{\be}_m=1$. Finally, we have $\la_m=k+1+\xi_m$
and $\la_{m+1}=k+\xi_{m+1}=\la_m-1$, since the assumption
$\nu_m=\nu_{m+1}$ implies that $\xi_m=\xi_{m+1}$.  

The rest of the argument now follows the proof of \cite[Prop.\ 5]{T4},
by studying the effect of the raising operators $R$ in the expansion
of $R^{D(w)}$ which involve only basic operators $R_{ij}$ with $i\in
\{m,m+1\}$ or $j\in \{m,m+1\}$.  We first assume that $\la$ has length
$m+1$, let $r:=\la_{m+1}=\la_m-1$, and use \cite[Prop.\ 3]{T4} and the
key relations
\[
f_r\wt{f}_r + 2\sum_{j=1}^r (-1)^j \left({}^ra^0_{r+j}
{}^ra^0_{r-j}\right) = \wt{f}^1_rf^1_r + 2\sum_{j=1}^r (-1)^j
\left({}^ra^1_{r+j} {}^ra^1_{r-j}\right) = 0
\]
in $\Gamma'[X,Y]$, which are easily deduced from the relations
(\ref{fundrelD}), as in op.\ cit. Finally, if $\ell(\la)>m+1$, similar
arguments show that the contribution of all the residual terms in that
appear with a negative sign in \cite[Prop.\ 3]{T4} vanishes.
\end{proof}

\begin{defn}
\label{properdef}
An element $w\in \wt{W}_\infty$ is called {\em proper} if (i) $|w_1|\leq 2$, 
or (ii) $|w_1|>2$ and $w_j=2$ implies $j>2>w_{j-1}$.
\end{defn}

\begin{lemma}
\label{validlem}
Fix an integer $k\geq 1$.  We say that an element of $\wt{W}_n$ is
{\em valid} if $w$ is increasing up to $k$, has $k$-truncated A-code a
partition $C$, and is proper. Let $w(C)$ be the longest valid element
in $\wt{W}_n$, which has type 0. If $\ov{w}$ is valid and $\ov{w}\neq
w(C)$, then there exists a valid $w\in \wt{W}_n$ such that
$\ell(w)=\ell(\ov{w})+1$ and $v(\ov{w})=s_iv(w)$ for some $i\in
\N_\Box$.
\end{lemma}
\begin{proof}
We distinguish the following cases for $\ov{w}$:

\medskip
\noin Case 1: $|\ov{w}_1|=1$. Let $w$ be any element with the same
truncated A-code $C$ such that $\ell(w)=\ell(\ov{w})+1$ and
$v(\ov{w})=s_iv(w)$ for some $i\in \N_\Box$.

\medskip
\noin
Case 2: $|\ov{w}_1|=2$. If $\ov{w}=(\wh{2} \cdots 1 \cdots)$, then let $w:=s_\Box\ov{w}$, while
if $\ov{w}=(\wh{2} \cdots \ov{1}\cdots)$, then let $w:=s_1\ov{w}$.

\medskip
\noin Case 3: $|\ov{w}_1|>2$. If $\ov{w}_j=2$ then $j>2$ and $\ov{w}_{j-1}<2$. 
Since the A-code $C$ is a partition, the sequence
$(\ov{w}_{k+1},\ldots,\ov{w}_n)$ is 132-avoiding. We deduce that if $\ov{w}_i=\wh{1}$, then $i<j$.
 If $\ov{w}= (\cdots \ov{1} \cdots 2
\cdots)$, then let $w:=s_1\ov{w}$, while if $\ov{w}= (\cdots 1 \cdots
2 \cdots)$, then let $w:=s_\Box\ov{w}$.

\medskip
In all three cases, the element $w$ satisfies the required conditions.
\end{proof}

\begin{cor}
\label{kdominD}
Suppose that $w\in \wt{W}_n$ is increasing up to $k\geq 1$, proper,
and the $k$-truncated A-code ${}^k\gamma$ is a partition. Let
$k^{n-k}+\xi(w)=(k+\xi_1,\ldots,k+\xi_{n-k})$.  Then we have
\begin{equation}
\label{directappD}
\DS_w = 2^{-\ell(\mu(w))}\,R^{D(w)} \star {}^{k^{n-k}+\xi(w)}\wh{c}^{\,\beta(w)}_{\la(w)}.
\end{equation}
\end{cor}
\begin{proof}
If $w=(\wh{1},2,\ldots,k,w_{k+1},\ldots,w_n)$ with $w_{k+j}<0$ for all
$j\in [1,n-k]$, then (\ref{directappD}) follows from Proposition
\ref{domnD}. In this case, $v(w)=(\wh{1},2,\ldots,k,-n,\ldots,-k-1)$
is the longest $k$-Grassmannian element in $\wt{W}_n$. We deduce the
result in the general case from Proposition \ref{gampropDN} and Lemma
\ref{validlem}, using the fact that the $k$-Grassmannian elements of
$\wt{W}_n$ form an ideal for the left weak Bruhat order. Indeed, the
hypotheses required in Proposition \ref{gampropDN} are satisfied, as
long as $w$ and $\ov{w}$ are proper. The key point is to show that if
$i\in\{\Box,1\}$ and $|w_1|>2$, then $\xi_m=\xi_{m+1}$, which implies
that $\nu_m=\nu_{m+1}$, and hence $\kk_m=\kk_{m+1}$. For if not, then
$\xi_m>\xi_{m+1}$, so there exists a $j>k$ such that $\gamma_j=m$. As
${}^k\gamma$ is a partition, $(w_{k+1},\ldots,w_n)$ is a
$132$-avoiding sequence. It follows that $w_j=\wh{1}$, and furthermore
$j=2$, or $j>2$ and $w_{j-1}>w_j$. We conclude that $\ov{w}_j=2$ and
$\ov{w}$ is not proper, completing the proof.
\end{proof}

Notice that if $w$ (respectively $\ov{w}$) is increasing up to $k\geq
1$ and proper (respectively not proper) with truncated A-code equal to
a fixed partition $C$, such that $\ell(w)=\ell(\ov{w})+1$ and
$s_iv(w)=v(\ov{w})$, then $i=\Box$ or $i=1$. Moreover, referring to
equation (\ref{Feq}), in this situation we can have $\DS_w = F_w$ for
an integer sequence $\kappa$ compatible with $w$ but $\DS_{\ov{w}}
\neq F_{\ov{w}}$, as Corollary \ref{kdominD} and the following example
show, with $k=1$, $w=(3,\ov{1},\ov{2})$, and $\ov{w}=(3,2,1)$.

\begin{example}
\label{321ex}
Let $n:=3$ and $w:=(3,2,1)$. Then $\type(w)=1$, $\la=\nu=(2,1)$,
$\ell=2$, $\mu=0$, $m=0$, $k=1$, $\beta(w)=(1,2)$, $D(w)=\emptyset$,
and $w$ is not proper. We have ${}^1c_0^2 = 1$, while
\begin{gather*}
{}^2b_2^1 = {}^2b_2+{}^2c_1\, h_1^1(-Y)+h_2^1(-Y) \\
=\frac{1}{2}\left(c_2+c_1e_1^2(X)\right) +e_2^2(X)
+\left(c_1+e_1^2(X)\right)h_1^1(-Y)+h_2^1(-Y), \\
{}^1c_1^2 = c_1+e_1^1(X)+h_1^2(-Y),  \\
{}^2a^1_3 = \frac{1}{2}\left(c_3+c_2e_1^2(X)+c_1e_2^2(X)\right) +
{}^2c_2\,h_1^1(-Y) + {}^2c_1\,h_2^1(-Y)+{}^2c_0\,h^1_3(-Y) \\
= \frac{1}{2}c_3+c_2\left(\frac{1}{2}e_1^2(X)+h_1^1(-Y)\right)
+ c_1\left(\frac{1}{2}e_2^2(X)+e_1^2(X)h_1^1(-Y)+h_2^1(-Y)\right) \\
+ \,e_2^2(X)h_1^1(-Y)+e_1^2(X)h_2^1(-Y)+h_3^1(-Y).
\end{gather*}
One checks using the table of \cite[Sec.\ 13]{IMN} that 
\[
(1-R_{12})\star {}^{(2,1)}\wh{c}^{\,(1,2)}_{(2,1)}
={}^2b_2^1\, {}^1c_1^2 - {}^2a^1_3\, {}^1c_0^2 \neq \DS_{321}.
\]
Now consider $w':=\iota(w)=(\ov{3},2,\ov{1})$. Then $\type(w')=2$,
$\la=\nu=(2,1)$, $\ell=2$, $\mu=0$, $m=0$, $k=1$, $\beta(w')=(0,2)$,
$D(w')=\emptyset$, and $w'$ is not proper. We have
\begin{gather*}
{}^2b_2^0 = {}^2b_2 =\frac{1}{2}\left(c_2+c_1e_1^2(X)\right) +e_2^2(X), \ \
{}^2\wt{b}_2^0 = {}^2\wt{b}_2 = \frac{1}{2}\left(c_2+c_1e_1^2(X)\right),  \\
{}^2a^0_3 = \frac{1}{2} {}^2c^0_3 = \frac{1}{2}\left(c_3+c_2e_1^2(X)
+ c_1e_2^2(X)\right).
\end{gather*}
Using the table of \cite[Sec.\ 13]{IMN}, we observe that 
\[
(1-R_{12})\star {}^{(2,1)}\wh{c}^{\,(0,2)}_{(2,1)}
={}^2\wt{b}_2^0\, {}^1c_1^2 - {}^2a^0_3\, {}^1c_0^2 \neq \DS_{\ov{3}2\ov{1}}.
\]
Notice that we also have $\dis {}^2b_2^0\, {}^1c_1^2 -
{}^2a^0_3\, {}^1c_0^2 \neq \DS_{\ov{3}2\ov{1}}$.
\end{example}

\begin{defn}
Let $k\geq 1$ be the primary index of $w\in \wt{W}_n$, and
list the entries $w_{k+1},\ldots,w_n$ in increasing order:
\[
u_1<\cdots < u_{m'} < 0 < u_{m'+1} < \cdots < u_{n-k},
\]
where $m'\in\{m,m+1\}$.  We say that a simple transposition $s_i$ for
$i\geq 2$ is {\em $w$-negative} (respectively, {\em $w$-positive}) if
$\{i,i+1\}$ is a subset of $\{-u_1,\ldots,-u_m\}$ (respectively, of
$\{u_{m'+1}, \ldots,u_{n-k}\}$). Let $\sigma^-$ (respectively,
$\sigma^+$) be the longest subword of $s_{n-1}\cdots s_2$
(respectively, of $s_2\cdots s_{n-1}$) consisting of $w$-negative
(respectively, $w$-positive) simple transpositions.  A {\em
  modification} of $w\in \wt{W}_n$ is an element $\ome w$, where
$\ome\in S_n$ is such that $\ell(\ome w)=\ell(w)-\ell(\ome)$, and
$\ome$ has a reduced decomposition of the form $R_1\cdots R_{n-2}$
where each $R_j$ is a (possibly empty) subword of $\sigma^-\sigma^+$
and all simple reflections in $R_p$ are also contained in $R_{p+1}$,
for each $p<n-2$.
\end{defn}

\begin{defn}
\label{amenD}
Suppose that $w\in \wt{W}_n$ has primary index $k\geq 1$ and A-code
$\gamma$.  We say that $w$ is {\em leading} if $w$ is proper and
$(\gamma_{k+1}, \gamma_{k+2},\ldots, \gamma_n)$ is a partition.  We
say that $w$ is {\em amenable} if $w$ is a modification of a leading
element.
\end{defn}

\begin{remark}
The proper element $w\in \wt{W}_n$ of type 0 or 1 is leading if and
only if the A-code of the extended sequence $(0,w_1,w_2,\ldots,w_n)$
is unimodal. Indeed, if $\type(w)=0$ and $w_1=\ov{1}$, then this is
ensured since there is more than one negative entry in $w$.  If
$\type(w)=2$, then $w$ is leading if and only if it is proper and the
A-code of the extended sequence $(0,w'_1,w'_2,\ldots,w'_n)$ is
unimodal, where $w':=\iota(w)$.
\end{remark}

In the following we will assume that $w$ has primary index $k\geq 1$
and the partition $\xi$ is specified as in Definition \ref{Dtrunc}.
Let $\psi:=(\gamma_k,\ldots,\gamma_1)$, $\phi:=\psi'$, $\ell:=\ell(\la)$
and $m:=\ell(\mu)$. We then have
\begin{equation}
\label{pxmND}
\la = \phi+\xi+\mu
\end{equation}
and $\la_1>\cdots > \la_m > \la_{m+1}\geq \cdots \geq \la_{\ell}$.

\begin{defn}
Say that $\ci\in [1,\ell]$ is a {\em critical index} if $\beta_{\ci+1}>
\beta_\ci+1$, or $(\beta_\ci,\beta_{\ci+1})= (1,2)$, or if
$\la_\ci>\la_{\ci+1}+1$ (respectively, $\la_\ci>\la_{\ci+1}$) and $\ci\leq m$
(respectively, $\ci>m$).  Define two sequences $\f=\f(w)$ and $\g=\g(w)$
of length $\ell$ as follows.  For $1\leq j \leq \ell$, set
\[
\f_j:=k+\max(i\ |\ \gamma_{k+i}\geq j)
\]
and let $$\g_j:=\f_\ci+\beta_\ci-\xi_\ci-k,$$ where $\ci$ is the least
critical index such that $\ci\geq j$. We call $\f$ the {\em right flag}
of $w$, and $\g$ the {\em left flag} of $w$.
\end{defn}

We will show that for any amenable element $w$, $\f$ is a weakly
decreasing sequence, which consists of right descents of $w$, unless
$\f_j=1$ and $w_1<-|w_2|$, when $\f_j$ is a right descent of
$\iota(w)$. Moreover, $\g$ is a weakly increasing sequence, whose
absolute values consist of left descents of $w$, unless $\g_j=0$ and
$w=(\cdots \ov{1} \cdots 2 \cdots)$, or $\g_j=1$ and $w=(\cdots 1
\cdots 2 \cdots)$.

\begin{lemma}
\label{philemD}
{\em (a)} If $\beta_{s+1} > \beta_s+1$ and $(\beta_s,\beta_{s+1})\neq
(0,2)$, then $|\beta_s|$ is a left descent of $w$.

\medskip
\noin {\em (b)} If $s\leq m$ then $\phi_s=k$.  If $s>m$ and
$\beta_{s+1} = \beta_s+1$, then $\phi_s=\phi_{s+1}$.

\medskip
\noin
{\em (c)} If $\beta_s=0$ or $\beta_s=1$, then $s=m+1$, $\type(w)>0$, and $\phi_{m+1}=k$.
If $\beta_s=0$ then $\Box$ is a left descent of $w$, unless 
$w = (\cdots \ov{1} \cdots 2 \cdots)$. 
If $\beta_s=1$ then $1$ is a left descent of $w$, unless 
$w = (\cdots 1 \cdots 2 \cdots)$. 
\end{lemma}
\begin{proof}
Let $i:=|\beta_s|$, and suppose that $1\leq s\leq m$. If $\beta_{s+1}
> \beta_s+1\neq 1$ then $i\geq 1$ is a left descent of $w$. Indeed, if
$w_1=-i$ then $i$ is a left descent of $w$, while if $w_1\neq -i$,
then $w^{-1}(i)>0$ and $w^{-1}(i+1)<0$, so this is clear.
 If $\type(w)\neq 2$, then $w_j \geq
-1$ for all $j\in [1,k]$, and hence $\psi_j\geq m$ for all $j\in
[1,k]$, and so $\phi_s=k$. 

Next suppose that $s>m$. If $\beta_{s+1} > \beta_s+1\neq 1$, then we have
$w^{-1}(i+1)<0$ or $w_j=i+1$ for some $j\in [1,k]$. In either case, it
is clear that $i$ is a left descent of $w$. Assume that
$\beta_{s+1} = \beta_s+1$, so that $\beta_s=u_s\geq 1$.
If $\phi_s>\phi_{s+1}$, there must exist
$j\in [1,k]$ such that $\gamma_j=s$, that is, $\#\{r>k\ |\ w_r<|w_j|\} =
s$. We deduce that 
\[
\{w_r\ |\ r>k \ \,\text{and}\, \ w_r<|w_j|\} =
\{u_1,\ldots, u_s\}, 
\]
which is a contradiction, since $u_s<|w_j| \Rightarrow
u_{s+1}=u_s+1<|w_j|$, for any $j\in [1,k]$.  This completes the proof
of (a) and (b) except in the case $\beta_s=0$, which is dealt with
below.

If $\beta_s=0$ or $\beta_s=1$, then clearly $s=m+1$ and $\type(w)>0$.
We have $\psi_j\geq m+1$ for all $j\in [1,k]$, and hence
$\phi_{m+1}=k$. If $\beta_s=0$, then since $w_{m+1}=\ov{1}$ we see that
$\Box$ is a left descent of $w$, unless $w = (\cdots \ov{1}
\cdots 2 \cdots)$. Finally, if $\beta_s=1$, then since $w_{m+1}=1$, it
is clear that $1$ is a left descent of $w$, unless $w = (\cdots 1
\cdots 2 \cdots)$.
\end{proof}

\begin{prop}
\label{amenableD}
Suppose that $\wh{w}\in \wt{W}_n$ is leading with primary index $k\geq
1$, let $\wh{\la}:=\la(\wh{w})$, and $\wh{\xi}:=\xi(\wh{w})$.  Let
$w=\ome\wh{w}$ be a modification of $\wh{w}$, and set
$\gamma:=\gamma(w)$, $\la:=\la(w)$, $\beta:=\beta(w)$, and
$\xi:=\xi(w)$.  Then the sequence $\beta+\wh{\la}-\la$ is weakly
increasing, and
\[
\DS_w = 2^{-\ell(\mu(w))} \,R^{D(w)} \star {}^{k^{n-k}+\wh{\xi}}\wh{c}^{\,\beta(w)+\wh{\xi}-\xi}_{\la(w)} =
2^{-\ell(\mu(w))} \,R^{D(w)} \star {}^{k^{n-k}+\wh{\xi}}\wh{c}^{\,\beta+\wh{\la}-\la}_{\la}.
\]
If $\ci\in [1,\ell]$ is a critical index of $w$, then $k+\wh{\xi}_\ci$
is a right descent of $w$, unless $k+\wh{\xi}_\ci=1$ and $w_1<-|w_2|$,
when $k+\wh{\xi}_\ci$ is a right descent of $\iota(w)$.  The absolute
value of $g_\ci:=\beta_\ci+\wh{\xi}_\ci-\xi_\ci$ is a left descent of
$w$, unless $g_\ci=0$ and $w=(\cdots \ov{1} \cdots 2 \cdots)$, or
$g_\ci=1$ and $w=(\cdots 1 \cdots 2 \cdots)$. Moreover, we have
$\wh{\xi}_\ci=\max(i\ |\ \gamma_{k+i}\geq \ci)$.
\end{prop}
\begin{proof}
Suppose that the truncated A-code of $\wh{w}$ is 
\[
{}^k\wh{\gamma}=(p_1^{n_1},\ldots, p_t^{n_t})
\]
for some parts $p_1>p_2>\cdots > p_t>0$,
and we let $d_j:=n_1+\cdots+n_j$ for $j\in [1,t]$. Then we have
\[
\wh{\xi}=(d_t^{p_t},d_{t-1}^{p_{t-1}-p_t},\ldots,d_1^{p_1-p_2})
\]
and it follows that 
\[
\wh{w}_{k+1}=u_{p_1+1},\ \wh{w}_{k+d_1+1}=u_{p_2+1},\ \ldots,
\ \wh{w}_{k+d_{t-1}+1}=u_{p_t+1}
\]
and $\wh{w}_j<\wh{w}_{j+1}$ for all $j\notin \{k,
k+d_1,\ldots,k+d_t\}$. Recall that $1$ is a right descent of $w$ if
and only if $w_1>w_2$, and $\Box$ is a right descent of $w$ if and
only if $w_1<-w_2$.  Hence, if the primary index $k$ equals $1$, then
$k$ is not a right descent of $w$ if and only if and $w_1<-|w_2|$, in
which case $\type(w)=2$ and $k$ is a right descent of $\iota(w)$.  We
deduce that the set of components of $k^{n-k}+\wh{\xi}$ coincides with
the set of all positive right descents of $\wh{w}$, or of
$\iota(\wh{w})$ if $\wh{w}_1<-|\wh{w}_2|$.

If $\ci\in [1,\ell]$ is a critical index, we have shown that
$k+\wh{\xi}_\ci$ is a right descent of $\wh{w}$, except in the case
when $k+\wh{\xi}_\ci=1$ and $\wh{w}_1<-|\wh{w}_2|$, when
$k+\wh{\xi}_\ci$ is a right descent of $\iota(\wh{w})$.  We claim that
$i:=|g_\ci|=|\beta_\ci|$ is a left descent of $\wh{w}$, unless $i=0$
and $\wh{w}=(\cdots \ov{1} \cdots 2 \cdots)$, or $i=1$ and
$\wh{w}=(\cdots 1 \cdots 2 \cdots)$. By Lemma \ref{philemD}, we may
assume that $\beta_\ci\neq 0$ and $\beta_{\ci+1}=\beta_\ci+1$. 

We first prove that $\ci\neq m$. Indeed, $\beta_{m+1}=\beta_m+1$ implies that
$\beta_m=-1$ and $\beta_{m+1}=0$, so in particular $|\wh{w}_1|>2$. Since
$\wh{w}$ is proper and ${}^k\wh{\gamma}$ is a partition, it follows that 
there is no $j\geq 1$ such that $\wh{\gamma}_{k+j}=m$. This implies that 
$\xi_m=\xi_{m+1}$, and since $\phi_m=\phi_{m+1}=k$ by Lemma \ref{philemD}(b), we 
deduce that $\la_m=\la_{m+1}+1$, which contradicts the fact that $\ci$ is a
critical index.

Suppose that $\ci < m$ and let $\wh{\mu}:=\mu(\wh{w})$. Then we have
$\wh{\la}_\ci>\wh{\la}_{\ci+1}+1$ and
$\wh{\mu}_\ci=\wh{\mu}_{\ci+1}+1$, so (\ref{pxmND}) gives
$\wh{\xi}_\ci>\wh{\xi}_{\ci+1}$. We therefore have $\ci=p_j$ for some
$\ci\in [1,t]$, and hence $i=\wh{\mu}_{p_j}-1 =
\wh{\mu}_{p_j+1}=-u_{p_j+1}=-\wh{w}_{k+d_{j-1}+1}$. Since we have
\[
\wh{w}_{k+1} > \wh{w}_{k+d_1+1} >\cdots > \wh{w}_{k+d_{j-1}+1} =-i,
\]
and the sequence $(\wh{w}_{k+1},\ldots,\wh{w}_n)$ is $132$-avoiding,
we conclude that $\wh{w}^{-1}(-i) = k+d_{j-1}+1 < \wh{w}^{-1}(-i-1)$, 
as desired.

Suppose next that $\ci>m$. Then we have $\wh{\la}_\ci>\wh{\la}_{\ci+1}$, so
Lemma \ref{philemD}(b) and equation (\ref{pxmND}) imply that
$\wh{\xi}_\ci>\wh{\xi}_{\ci+1}$. We deduce that $\ci=p_j$ for some $j$,
hence $i+1=u_{p_j+1}$ and the claim follows.

According to Corollary  \ref{kdominD}, we have
\begin{equation}
\label{propDcor}
\DS_{\wh{w}} = 2^{-\ell(\mu(\wh{w}))}\,R^{D(\wh{w})} \star {}^{k^{n-k}+\wh{\xi}}\wh{c}^{\,\beta(\wh{w})}_{\la(\wh{w})}, 
\end{equation}
so the proposition holds for leading elements.  Suppose next that
$w:=\ome\wh{w}$ is a modification of $\wh{w}$. Then repeated
application of (\ref{Dddeq}), Propositions \ref{prop1new}(a),
\ref{prop1hnew}(a), and the left Leibnitz rule in equation
(\ref{propDcor}) give
\[
\DS_w = 2^{-\ell(\mu(w))}\,R^{D(w)} \star {}^{k^{n-k}+\wh{\xi}}\wh{c}^{\,\beta(w)+\wh{\xi}-\xi}_{\la(w)}.
\]
It remains to check the last assertion, about the left and right 
descents of $w$. This is done exactly as in the proof of Proposition
\ref{amenableC}.
\end{proof}

\begin{thm}
\label{TamvexD}
For any amenable element $w\in \wt{W}_\infty$, we have 
\begin{equation}
\label{DSeq}
\DS_w = 2^{-\ell(\mu(w))}\,R^{D(w)} \star {}^{\f(w)}\wh{c}_{\la(w)}^{\,\g(w)}
\end{equation}
in $\Gamma'[X,Y]$.
\end{thm}
\begin{proof}
We may assume that we are in the situation of Proposition
\ref{amenableD}, so that $w=\ome\wh{w}$, with $\wh{\la}=\la(\wh{w})$
and $\la=\la(w)$.  Suppose that $j\in [1,\ell]$ and let $\ci$ be the
least critical index of $w$ such that $\ci\geq j$. Then we have
$\la_j=\la_{j+1}=\cdots = \la_\ci$, if $\ci>m$, and $\la_j=\la_{j+1}+1
= \cdots = \la_\ci+(\ci-j)$, if $\ci\leq m$. Moreover, in either case,
we have $\xi_j=\cdots = \xi_\ci$, and the values
$\beta_j,\ldots,\beta_\ci$ are consecutive integers. As the sequence
$g:=\beta+\wh{\xi}-\xi$ is weakly increasing, we deduce that for any
$r\in [j,\ci-1]$, either (i) $\wh{\xi}_r=\wh{\xi}_{r+1}$ and
$g_r=g_{r+1}-1$, or (ii) $\wh{\xi}_r=\wh{\xi}_{r+1}+1$ and
$g_r=g_{r+1}$. Equation (\ref{DSeq}) follows from this and induction
on $\ci-j$, by using Lemmas \ref{commuteAD}(b) and \ref{commuteCD}(b)
in Proposition \ref{amenableD}.  The required conditions on $D(w)$ in
these two lemmas and the corresponding relations (\ref{fundrelsC}) and
(\ref{fundrelD}) are all easily checked.
\end{proof}

\subsection{Flagged eta polynomials}

In this section, we define a family of polynomials $\Eta_w$ indexed by
amenable elements $w\in \wt{W}_\infty$ that generalize the double eta
polynomials of \cite{T4}. As in Section \ref{ftps}, the polynomial
$\Eta_w$ represents an equivariant Schubert class in the
$T$-equivariant cohomology ring of the even orthogonal partial flag
variety associated to the right flag $\f(w)$.

For every $k\geq 1$, let
${}^k\frakb:=({}^k\wt{\frakb}_k,{}^k\frakb_1,{}^k\frakb_2,\ldots)$ and
${}^k\frakc:=({}^k\frakc_1,{}^k\frakc_2,\ldots)$ be families of
commuting variables, set ${}^k\frakb_0={}^k\frakc_0=1$ and
${}^k\frakb_p={}^k\frakc_p=0$ for each $p<0$, and let
$t:=(t_1,t_2,\ldots)$. These variables are related by the equations
\[
{}^k\frakc_p=
\begin{cases}
{}^k\frakb_p &\text{if $p< k$},\\
{}^k\frakb_k+{}^k\wt{\frakb}_k &\text{if $p=k$},\\
2\left({}^k\frakb_p\right) &\text{if $p> k$}.
\end{cases}
\]
For any $p,r\in \Z$ and for $s\in \{0,1\}$, 
define the polynomials ${}^k\frakc^r_p$ and ${}^k\fraka^s_p$ by
\[
{}^k\frakc^r_p:= \sum_{j=0}^p {}^k\frakc_{p-j} \, h^r_j(-t) \ \ \ \text{and} \ \ \
{}^k\fraka^s_p:= \frac{1}{2}\left({}^k\frakc_p\right)+\sum_{j=1}^p {}^k\frakc_{p-j} \, h^s_j(-t).
\]
Moreover, define
\[
{}^k\frakb^s_k:= {}^k\frakb_k+\sum_{j=1}^k {}^k\frakc_{k-j} \, h^s_j(-t) \ \ \ \text{and} \ \ \
{}^k\wt{\frakb}^s_k:= {}^k\wt{\frakb}_k+\sum_{j=1}^k {}^k\frakc_{k-j} \, h^s_j(-t).
\]

For any integer sequences $\al$, $\rho$, $\kk$ with $\kk_i\geq 1$ for
each $i$, let
\[
{}^\kk\wh{\frakc}_\al^{\,\rho}:=
{}^{\kk_1}\wh{\frakc}_{\al_1}^{\,\rho_1}\,
{}^{\kk_2}\wh{\frakc}_{\al_2}^{\,\rho_2}\cdots
\]
where, for each $i\geq 1$, 
\[
{}^{\kk_i}\wh{\frakc}_{\al_i}^{\,\rho_i}:= {}^{\kk_i}\frakc_{\al_i}^{\rho_i} + 
\begin{cases}
(2({}^{\kk_i}\wt{\frakb}_{\kk_i})-
  {}^{\kk_i}\frakc_{\kk_i})e^{\al_i-\kk_i}_{\al_i-\kk_i}(-t) &
  \text{if $\rho_i = \kk_i - \al_i < 0$ and $i$ is odd}, \\
  (2({}^{\kk_i}\frakb_{\kk_i})-{}^{\kk_i}\frakc_{\kk_i})e^{\al_i-\kk_i}_{\al_i-\kk_i}(-t)
  & \text{if $\rho_i = \kk_i - \al_i < 0$ and $i$ is even}, \\
  0 & \text{otherwise}.
\end{cases}
\]

If $w\in \wt{W}_n$ is amenable with left flag $\f(w)$ and right flag
$\g(w)$, the {\em flagged double eta polynomial} $\Eta_w(\frakc\, |\, t)$
is defined by
\begin{equation}
\label{fletaeq}
\Eta_w(\frakc\, |\, t):= 2^{-\ell(\mu(w))}\,R^{D(w)}\star {}^{\f(w)}\wh{\frakc}_{\la(w)}^{\,\g(w)},
\end{equation}
where the action $\star$ of the raising operator expression $R^{D(w)}$
is as in Definition \ref{Etadef}.  The {\em flagged single eta
polynomial} is given by $\Eta_w(\frakc):=\Eta_w(\frakc\, |\, 0)$.  If
$w$ is a leading element, then (\ref{fletaeq}) can be written in the
`factorial' form
\[
\Eta_w(\frakc\, |\, t)= 2^{-\ell(\mu(w))}\,R^{D(w)}\star {}^{\f(w)}\wh{\frakc}_{\la(w)}^{\,\be(w)}.
\]
When $w$ is a $k$-Grassmannian element, the above formulas specialize
to the double eta polynomial $\Eta_\la(\frakc\, |\, t)$ found in
\cite{T4}; here $\la$ is the typed $k$-strict partition corresponding
to $w$.  Moreover, the single eta polynomial $\Eta_\la(\frakc)$ agrees
with that introduced in \cite{BKT3}; see also \cite{T2, T3}.

\subsection{Orthogonal degeneracy loci}
\label{odl}

\subsubsection{Odd orthogonal loci}

Let $E\to \X$ be a vector bundle of rank $2n+1$ on a smooth complex
algebraic variety $\X$. Assume that $E$ is an {\em orthogonal} bundle,
i.e.\ $E$ is equipped with an everywhere nondegenerate symmetric form
$E\otimes E\to \C$. Let $w\in W_n$ be amenable of shape $\la$, and let
$\f$ and $\g$ be the left and right flags of $w$, respectively. Define
a new sequence $\ov{\g}$ by setting $\ov{\g}_j:= \g_j$, if $\g_j\geq 0$, and
$\ov{\g}_j:=\g_j-1$, if $\g_j<0$.  Consider two complete flags of
subbundles of $E$
\[
0 \subset E_1\subset \cdots \subset E_{2n+1}=E \ \ \, \mathrm{and} \, \ \ 
0 \subset F_1\subset \cdots \subset F_{2n+1}=E
\]
with $\rank E_r=\rank F_r=r$ for each $r$, while $E_{n+s}=E_{n+1-s}^{\perp}$ 
and $F_{n+s}=F_{n+1-s}^{\perp}$ for $1\leq s \leq n$. 

There is a group monomorphism $\zeta':W_n\hra S_{2n+1}$ with image
\[
\zeta'(W_n)=\{\,\om\in S_{2n+1} \ |\ \om_i +\om_{2n+2-i} = 2n+2,
 \ \ \text{for all}  \ i\,\}.
\]
The map $\zeta'$ is determined by setting,
for each $w=(w_1,\ldots,w_n)\in W_n$ and $1\leq i \leq n$,
\[
\zeta'(w)_i :=\begin{cases}
             n+1-w_{n+1-i} & \text{if $w_{n+1-i}$ is unbarred}, \\
             n+1+\ov{w}_{n+1-i} & \text{otherwise}.
             \end{cases}
\]
Define the {\em degeneracy locus}
$\X_w\subset \X$ as the locus of $x \in \X$ such that
\[
\dim(E_r(x)\cap F_s(x))\geq \#\,\{\,i \leq r
\ |\ \zeta'(w)_i > 2n+1-s\,\} \ \, \mathrm{for} \ 1\leq r \leq n, \, 1\leq s \leq 2n.
\]
As in the symplectic case, we assume that $\X_w$ has pure codimension
$\ell(w)$ in $\X$, and give a formula for the class $[\X_w]$ in
$\HH^{2\ell(w)}(\X)$.

\begin{thm}
\label{dbleBloci}
For any amenable element $w\in W_n$, we have
\begin{align*}
\label{BChern}
[\X_w] &= 2^{-\ell(\mu(w))}\,\Theta_w(E-E_{n-\f}-F_{n+1+\ov{\g}}) \\
&= 2^{-\ell(\mu(w))}\,R^{D(w)}\, c_{\la}(E-E_{n-\f}-F_{n+1+\ov{\g}})
\end{align*}
in the cohomology ring $\HH^*(\X)$.
\end{thm}

Theorem \ref{dbleBloci} is derived from equation (\ref{CSeq}) in the
same way as Theorem \ref{dbleCloci}, using the type B geometrization
map of \cite[Sec.\ 10]{IMN}; compare with \cite[Sec.\ 6.3.1]{T3}.

\subsubsection{Even orthogonal loci}

Let $E\to \X$ be an orthogonal vector bundle of rank $2n$ on a
smooth complex algebraic variety $\X$. Let $w\in \wt{W}_n$ be an amenable
element of shape $\la$, and let $\f$ and $\g$ be the left and right
flags of $w$, respectively. Two maximal isotropic subbundles $L$ and
$L'$ of $E$ are said to be in the same family if $\mathrm{rank}(L\cap
L')\equiv n \, (\text{mod 2})$. Consider two complete flags of subbundles of $E$
\[
0 \subset E_1\subset \cdots \subset E_{2n}=E \ \ \, \mathrm{and} \, \ \ 
0 \subset F_1\subset \cdots \subset F_{2n}=E
\]
with $\rank E_r=\rank F_r=r$ for each $r$, while
$E_{n+s}=E_{n-s}^{\perp}$ and $F_{n+s}=F_{n-s}^{\perp}$ for $0\leq s <
n$.  We assume that $E_n$ is in the same family as $F_n$, if $n$ is
even, and in the opposite family, if $n$ is odd.

We have a group monomorphism $\zeta:\wt{W}_n\hra S_{2n}$, defined by
restricting the map $\zeta$ of Section \ref{sdl} to $\wt{W}_n$. Let
$\wt{w}_0$ denote the longest element of $\wt{W}_n$, and define the
{\em degeneracy locus} $\X_w\subset \X$ as the closure of the locus of
$x \in \X$ such that
\[
\dim(E_r(x)\cap F_s(x)) = \#\,\{\,i \leq r
\ |\ \zeta(\wt{w}_0w\wt{w}_0)_i > 2n-s\,\} \ \,  
\mathrm{for} \ 1\leq r \leq n-1, \, 1\leq s \leq 2n
\]
with the reduced scheme structure. Assume further that $\X_w$ has
pure codimension $\ell(w)$ in $\X$, and consider its cohomology class
$[\X_w]$ in $\HH^{2\ell(w)}(\X)$.

\begin{thm}
\label{dbleDloci}
For any amenable element $w\in \wt{W}_n$, we have
\begin{equation}
\label{DChern}
[\X_w] = \Eta_w(E-E_{n-\f}-F_{n+\g}) =
2^{-\ell(\mu(w))}\,R^{D(w)}\star \wh{c}_{\la}(E-E_{n-\f}-F_{n+\g})
\end{equation}
in the cohomology ring $\HH^*(\X)$.
\end{thm}

The Chern polynomial in (\ref{DChern}) is defined by employing the
substitutions
\begin{gather*}
{}^r\frakb_p \longmapsto \begin{cases}
c_p(E-E_{n-r}-F_n)  & \text{if $p<r$}, \\
\dis\frac{1}{2} c_p(E-E_{n-r}-F_n)  & \text{if $p>r$},
\end{cases} \\
{}^r\frakb_r\longmapsto\frac{1}{2}(c_r(E-E_{n-r}-F_n)+c_r(E_n-E_{n-r})), \\
{}^r\wt{\frakb}_r\longmapsto\frac{1}{2}(c_r(E-E_{n-r}-F_n)-c_r(E_n-E_{n-r}))
\end{gather*}
in equation (\ref{fletaeq}), for any integer $p$ and $r\geq 1$. The
proof of Theorem \ref{dbleDloci} is obtained by applying the type D
geometrization map of \cite[Sec.\ 10]{IMN} to equation (\ref{DSeq}),
and using the computations in \cite[Sec.\ 7.4]{T3}.

\appendix

\section{Counterexamples to statements in \cite{AF2}}
\label{AFexs}

The following two examples exhibit errors in the proofs -- in all
types except type A -- and in the main type D result of \cite{AF2}.
We use the notation in op.\ cit.

\begin{exA1*}
We show that Lemma A.1(i) of \cite{AF2} is incorrect. Set
$\rho=(0,1,0)$, $\la=(2,1,1)$, $k=m=2$, and $n=\ell=3$. The
assumptions are that $c(2)=c(3)$, $c'(2)=c(2)(1+b_1)$, so that
$c'_j(2)=c_j(2)+b_1c_{j-1}(2)$ for each $j$, and $c'(i)=c(i)$ for
$i=1,3$.

We compute that $R^{(\rho,\ell)} = (1+R_{12})^{-1}(1-R_{12})(1-R_{13})(1-R_{23})$
and hence
\[
\Ti_\la^{(\rho)}(c)= 
c_2(1)c_1(2)c_1(2) -2c_3(1)c_1(2)+c_3(1)c_1(2)-c_2(1)c_2(2)
\]
while 
\[
\Ti_\la^{(\rho)}(c')=c_2(1)(c_1(2)+b_1)c_1(2) -2c_3(1)c_1(2)+c_3(1)(c_1(2)+b_1)
-c_2(1)(c_2(2)+b_1c_1(2)).
\]
It follows that
\[
\Ti_\la^{(\rho)}(c')-\Ti_\la^{(\rho)}(c) = 
b_1c_3(1)\neq 0.
\]
One can show similarly that Lemma A.1(ii) and Lemma A.2 of op.\ cit.\ are also wrong.
\end{exA1*}

\begin{exA2*}
We show that Theorem 4 of \cite{AF2} is false. Consider the type D
triple $\tau = (\mathrm{\bf k}, \mathrm{\bf p}, \mathrm{\bf q})
= ((1,2), (2,1), (0,-2))$, which corresponds to the Weyl group
element $(\ov{3},2,\ov{1})\in \wt{W}_3$ (or to the element $(3,2,1)$,
depending on the type convention). We have $\rho=(0,0)$ and
$\la=(2,1)$, while $\ell=2$ and $r=1$, so $\wt{R}^{(\rho,r,\ell)} =
1-R_{12}$ and
\[
\Eta^{\rho(\tau)}_{\la(\tau)}(c(1),c(2)) = 
(1-R_{12})(c(1)_2c(2)_1) = c_2(1)c_1(2) - c_3(1)c_0(2).
\]
The computations of Example \ref{321ex} show that
$2\,[\Omega_{\bf \tau}]\neq \Eta^{\rho(\tau)}_{\la(\tau)}(c(1),c(2))$.
\end{exA2*}

\end{document}